\documentclass[10pt,reqno]{amsart}

\usepackage{amsmath,amssymb,mathrsfs,graphicx,verbatim}
\usepackage[inline]{enumitem}

\usepackage[pdftex=true,
	bookmarks=true,
	bookmarksopen=true,
	bookmarksnumbered=true,
	pdftoolbar=true,
	pdfmenubar=true,
	pdffitwindow=false,
	pdfstartview={FitH},
	colorlinks=true,
	linkcolor=red,
	citecolor=blue,
	breaklinks=true]{hyperref}

\allowdisplaybreaks

\numberwithin{equation}{section}

\newtheorem{thm}{Theorem}[section]
\newtheorem{thmi}{Theorem}
\newtheorem{prop}[thm]{Proposition}
\newtheorem{lemma}[thm]{Lemma}
\newtheorem{cor}[thm]{Corollary}
\newtheorem*{thm*}{Theorem}
\newtheorem*{prop*}{Proposition}
\newtheorem*{cor*}{Corollary}
\newtheorem*{conj*}{Conjecture}

\theoremstyle{definition}
\newtheorem{definition}[thm]{Definition}

\theoremstyle{remark}
\newtheorem{rmk}[thm]{Remark}

\newlength{\enummargin}
\newcommand{\mc}{\mathcal}

\newcommand{\cC}{\mc C}
\newcommand{\cD}{\mc D}
\newcommand{\cE}{\mc E}

\newcommand{\cH}{\mc H}

\newcommand{\cK}{\mc K}

\newcommand{\cO}{\mc O}
\newcommand{\cP}{\mc P}
\newcommand{\cQ}{\mc Q}
\newcommand{\cR}{\mc R}

\newcommand{\cX}{\mc X}
\newcommand{\cY}{\mc Y}
\newcommand{\cZ}{\mc Z}

\newcommand{\ms}{\mathscr}

\newcommand{\sC}{\ms C}

\newcommand{\C}{\mathbb{C}}
\newcommand{\N}{\mathbb{N}}
\newcommand{\R}{\mathbb{R}}

\newcommand{\Sph}{\mathbb{S}}

\newcommand{\frakt}{\mathfrak{t}}

\newcommand{\ran}{\operatorname{ran}}

\newcommand{\im}{\operatorname{Im}}

\renewcommand{\Im}{\operatorname{Im}}
\renewcommand{\Re}{\operatorname{Re}}

\newcommand{\supp}{\operatorname{supp}}

\newcommand{\dR}{\mathrm{dR}}

\newcommand{\la}{\langle}
\newcommand{\ra}{\rangle}
\newcommand{\pa}{\partial}
\newcommand{\tn}{\textnormal}

\newcommand{\eps}{\epsilon}
\newcommand{\ep}{\epsilon}

\newcommand{\xra}{\xrightarrow}

\newcommand{\bl}{{\mathrm{b}}}
\newcommand{\cl}{{\mathrm{c}}}

\newcommand{\semi}{\hbar}

\newcommand{\Diff}{\mathrm{Diff}}
\newcommand{\Diffb}{\Diff_\bl}

\newcommand{\even}{\mathrm{even}}

\newcommand{\CI}{\cC^\infty}
\newcommand{\CIdot}{\dot\cC^\infty}
\newcommand{\CIc}{\cC^\infty_\cl}

\newcommand{\CmIdot}{\dot\cC^{-\infty}}

\newcommand{\Hext}{{\bar H}{}}

\newcommand{\ol}{\overline}
\newcommand{\oX}{\overline{X}}
\newcommand{\Xeven}{\overline{X}_{\mathrm{even}}}

\newcommand{\numin}{\nu_{\mathrm{min}}}

\newcommand{\wh}{\widehat}
\newcommand{\wt}{\widetilde}
\newcommand{\hd}{\wh{d}}
\newcommand{\hdel}{\wh{\delta}}
\newcommand{\td}{\wt{d}}
\newcommand{\tdel}{\wt{\delta}}

\begin{document}
\title[Asymptotics for differential forms]{Asymptotics for the wave equation on differential forms on Kerr--de Sitter space}

\author{Peter Hintz}
\address{Department of Mathematics, Stanford University, CA 94305-2125, USA}
\curraddr{Department of Mathematics, University of California, Berkeley, CA, 94720-3840, USA}
\email{phintz@berkeley.edu}

\author{Andr\'as Vasy}
\address{Department of Mathematics, Stanford University, CA 94305-2125, USA}
\email{andras@math.stanford.edu}
\date{March 6, 2015. Final revision: January 18, 2018.}
\thanks{The authors were supported in part by A.V.'s National Science
  Foundation grants DMS-1068742 and DMS-1361432 and P.H.\ was
  supported in part by a Gerhard Casper Stanford Graduate Fellowship. We are grateful to four anonymous referees for carefully reading the manuscript and for providing comments and suggestions which significantly improved the exposition.}

\begin{abstract}
  We study asymptotics for solutions of Maxwell's equations, in fact of the Hodge--de Rham equation $(d+\delta)u=0$ without restriction on the form degree, on a geometric class of stationary spacetimes with a warped product type structure (without any symmetry assumptions), which in particular include Schwarz\-schild-de Sitter spaces of all spacetime dimensions $n\geq 4$. We prove that solutions decay exponentially to $0$ or to stationary states in every form degree, and give an interpretation of the stationary states in terms of cohomological information of the spacetime. We also study the wave equation on differential forms and in particular prove analogous results on Schwarzschild--de Sitter spacetimes. We demonstrate the stability of our analysis and deduce asymptotics and decay for solutions of Maxwell's equations, the Hodge--de Rham equation and the wave equation on differential forms on Kerr--de Sitter spacetimes with small angular momentum.
\end{abstract}

\maketitle

\section{Introduction}
\label{SecIntro}

Maxwell's equations describe the dynamics of the electromagnetic field on a $4$-dimensional spacetime $(M,g)$. Writing them in the form $(d+\delta_g)F=0$, where $\delta_g$ is the codifferential, for the electromagnetic field $F$ (a $2$-form) suggests studying the operator $d+\delta_g$, whose square
\[
  \Box_g=(d+\delta_g)^2
\]
is the Hodge d'Alembertian, i.e.\ the wave operator on differential forms. It is then very natural to study solutions of $(d+\delta_g)u=0$ or $\Box_g u=0$ without restrictions on the form degree or the dimension of the spacetime. Here, we study quasinormal modes (or resonances) of $\Box_g$ (or $d+\delta_g$) when $M=\R_t\times X$, $X=\ol X^\circ$ with $\ol X$ compact, is equipped with a stationary Lorentzian metric $g$ which has a suitable warped product structure near $\pa\ol X$: resonances are complex numbers $\sigma\in\C$ for which there exists a smooth $t$-independent differential form $u(x)$ on $M$ satisfying outgoing boundary conditions at $\pa\ol X$, such that $\Box_g(e^{-i t\sigma}u)=0$ (or $(d+\delta_g)(e^{-i t\sigma}u)=0$): roughly, we show that all such resonances satisfy $\Im\sigma<0$, with the exception of a possible resonance at $\sigma=0$, corresponding to stationary solutions of the wave or Hodge--de Rham equation, for which we give a (rather subtle) description in terms of the cohomology of $M$, thus significantly refining the geometric understanding of asymptotics for waves on single black hole spacetimes studied in the literature to date (see~Section~\ref{SubsecPrevious} for references).

Important examples of spacetimes that fit into the class of spacetimes studied in the present paper are Schwarzschild--de Sitter spacetimes with spacetime dimension $\geq 4$; it is important to note that our results are much more general, allowing for an arbitrary topology of $\ol X$ (thus allowing e.g.\ for multiple black holes). For Schwarzschild--de Sitter spaces, or indeed perturbations of these, in particular on Kerr--de Sitter spaces with small angular momenta, we can use our results on the location and structure of resonances and prove a partial resonance expansion of waves into their stationary part plus an exponentially decaying remainder:

\begin{thmi}
\label{ThmIntroKDS}
  Let $(M,g_a)$ denote a neighborhood of the domain of outer communications of a non-degenerate Kerr--de Sitter space with black hole mass $M_\bullet>0$, cosmological constant $\Lambda>0$ and angular momentum $a$ which we assume to be very small, $|a|\ll M_\bullet$;\footnote{The non-degeneracy requirement ensures that the cosmological horizon lies outside the event horizon. For example, if $M_\bullet$ is fixed and $\Lambda>0$ is small, sufficiently small values of $a$ work.} Denote by $t_*$ a smooth time function which is equal to the Boyer--Lindquist coordinate $t$ away from the horizons, and a suitable (Kerr-star coordinate type) modification of $t$ near the horizon.\footnote{See~\eqref{EqTimeStar} for the definition in the warped product setting which applies to Schwarzschild--de~Sitter spacetimes (i.e.\ $a=0$), and Section~\ref{SecKdS} for references for the Kerr--de~Sitter case.} Suppose $u\in\CI(M;\Lambda M)$ is a solution of the equation
  \[
    (d+\delta_{g_a})u=0,
  \]
  with smooth initial data, and denote by $u_j$ the form degree $j$ part of $u$, $j=0,\ldots,4$. Then $u_2$ decays exponentially in $t_*$ to a stationary state, which is a linear combination of the $t_*$-independent $2$-forms $u_{a,1},u_{a,2}$. In the standard (Boyer--Lindquist) local coordinate system on Kerr--de Sitter space, $u_{a,1}$ and $u_{a,2}$ have explicit closed form expressions; in particular, on Schwarzschild--de Sitter space, $u_{0,1}=r^{-2}\,dt\wedge dr$, and $u_{0,2}=\omega$ is the volume element of the round unit $2$-sphere. Moreover, $u_1$ and $u_3$ decay exponentially to $0$, while $u_0$ decays exponentially to a constant, and $u_4$ to a constant multiple of the volume form.

  Suppose now $u\in\CI(M;\Lambda M)$ instead solves the wave equation
  \[
    \Box_{g_a}u=0
  \]
  with smooth initial data, then the same decay as before holds for $u_0,u_2$ and $u_4$, while $u_1$ decays exponentially to a member of a $2$-dimensional family of stationary states, likewise for $u_3$.
\end{thmi}

The Schwarzschild--de Sitter case of this theorem, i.e.\ the special case $a=0$, will be proved in Section~\ref{SubsecSDS}, and we give explicit expressions for all stationary states, see Theorems~\ref{ThmSDSDDel} and \ref{ThmSDSBox}. Section~\ref{SecKdS} provides the perturbation arguments, see in particular Theorem~\ref{ThmKDS}; we point out that while Schwarzschild--de~Sitter spacetimes fit directly into our framework, Kerr--de~Sitter spacetimes do not, as they do not have the requisite warped product structure described in Section~\ref{SubsecOutline} below, hence we can only treat them perturbatively here. For the explicit form of $u_{a,1}$ and $u_{a,2}$, see Remark~\ref{RmkKDSExplicit}. Note that asymptotics and exponential decay of differential form solutions to the wave equation are stronger statements than corresponding statements for Maxwell's equations or for the Hodge--de Rham equation, as any solution of one of the latter equations is automatically a solution of the former; improved results for the Maxwell or Hodge--de Rham equation can then be obtained in a second step.

Our arguments strongly use that we are dealing with the Hodge--d'Alembertian $\Box_g=(d+\delta_g)^2$ rather than related operators which differ from $\Box_g$ by lower order terms, e.g.\ the rough wave operator $-\mathrm{tr}\,\nabla^2$, or the Klein--Gordon type operator $\Box_g-m^2$, $m\in\R$. Indeed, the factorization of the Hodge--d'Alembertian is essential for us. Furthermore, as we rely heavily on integration by parts and symmetry considerations which exploit properties of the form bundle, we do not treat more general tensor bundles here.

We stress that the main feature of the spacetimes $(M,g)$ considered in this paper is a warped product type structure of the metric; \emph{we do not make any symmetry assumptions on $M$}. From a geometric point of view then, the main novelty of this paper is a general cohomological interpretation of stationary states, which in the above theorem are merely explicitly given. On a technical level, we show how to analyze quasinormal modes (also called resonances, further discussed below) for equations on vector bundles whose natural inner product is not positive definite. To stress the generality of the method, we point out that symmetries only become relevant in explicit calculations for specific examples such as Schwarzschild--de Sitter and Kerr--de Sitter spaces. Even then, the perturbation analysis around Schwarzschild--de Sitter space works without restrictions on the perturbation; only for the explicit form of the space $\la u_{a,1},u_{a,2}\ra$ of stationary states do we need the very specific form of the Kerr--de Sitter metric. Thus, combining the perturbation analysis with the non-linear framework developed by the authors in \cite{HintzVasyQuasilinearKdS}, we can immediately solve suitable \emph{quasilinear wave equations on differential forms} on Kerr--de Sitter spacetimes; see Remark~\ref{RmkQuasilinear}. To put this into context, part of the motivation for the present paper is the black hole stability problem, see the lecture notes by Dafermos and Rodnianski \cite{DafermosRodnianskiLectureNotes} for background on this, and we expect that the approach taken here will facilitate the linear part of the stability analysis, which, when accomplished, rather directly gives the non-linear result when combined with the non-linear analysis in \cite{HintzVasyQuasilinearKdS}.

\subsection{Outline of the general result}
\label{SubsecOutline}

Going back to the linear problem studied in this paper, we proceed to explain the general setup in more detail. Let $\oX$ be a connected, compact, orientable $(n-1)$-dimensional manifold with non-empty boundary $Y=\pa\oX\neq\emptyset$ and interior $X=\oX^\circ$, and let
\[
  M=\R_t\times X,
\]
which is thus $n$-dimensional. Denote the connected components of $Y$, which are of dimension $(n-2)$, by $Y_i$, for $i$ in a finite index set $I$. We assume that $M$ is equipped with the metric
\begin{equation}
\label{EqMetricSpacetime}
  g=\alpha(x)^2\,dt^2-h(x,dx),
\end{equation}
where $h$ is a smooth Riemannian metric on $\oX$ (in particular, incomplete) and $\alpha$ is a boundary defining function of $X$, i.e.\ $\alpha\in\CI(\oX)$, $\alpha=0$ on $Y$, $\alpha>0$ in $X$ and $d\alpha|_Y\neq 0$. (As we demonstrate in equations~\eqref{EqdSMetric} and \eqref{EqSdSMetric}, (Schwarzschild--)de~Sitter space indeed has this form.) We moreover assume that every connected component $Y_i$ of $Y$, $i\in I$, has a collar neighborhood $[0,\eps_i)_\alpha\times(Y_i)_y$ in which $h$ takes the form
\begin{equation}
\label{EqMetricSpacetime2}
  h=\wt\beta_i(\alpha^2,y)\,d\alpha^2 + k_i(\alpha^2,y,dy)
\end{equation}
with $\wt\beta_i(0,y)\equiv\beta_i>0$ constant along $Y_i$.\footnote{The constancy is required for the Fredholm analysis in Section~\ref{SecPrelim}, and is satisfied for all examples considered in this paper; in the case of Schwarzschild--de~Sitter spacetimes, it amounts to the constancy of the surface gravities of the event and the cosmological horizon.} In particular, $\alpha^{-2}h$ is an \emph{even} asymptotically hyperbolic metric in the sense of Guillarmou \cite{GuillarmouMeromorphic}; for the connection between horizons and asymptotically hyperbolic spaces, we also refer to~\cite{BachelotMotetBachelotSchwarzschild,SaBarretoZworskiResonances} and \cite[Chapter~4]{ChandrasekharBlackHoles}.\footnote{Thus, as we will show, de Sitter and Schwarzschild--de Sitter spaces fit into this framework, whereas asymptotically flat spacetimes like Schwarzschild (or Kerr) do not.} We change the smooth structure on $\oX$ to only include even functions of $\alpha$, and show how one can then extend the metric $g$ to a stationary metric (denoted $\wt g$, but dropped from the notation in the sequel) on a bigger spacetime $\wt M=\R_{t_*}\times\wt X$, where $\ol X\hookrightarrow\wt X^\circ$, and where $t_*$ is a shifted time coordinate. Since the operator $d+\delta$ commutes with time translations, it is natural to consider the normal operator family
\[
  \td(\sigma)+\tdel(\sigma) = e^{it_*\sigma}(d+\delta)e^{-it_*\sigma}
\]
acting on differential forms (valued in the form bundle of $M$) on a slice of constant $t_*$, identified with $\wt X$; that is, every $\pa_{t_*}$ is replaced by multiplication by $-i\sigma$. The normal operator family $\wt\Box(\sigma)$ of $\Box$ is defined completely analogously.

Since the Hodge d'Alembertian (and hence the normal operator family $\wt\Box(\sigma)$) has a scalar principal symbol, it can easily be shown to fit into the microlocal framework developed by Vasy \cite{VasyMicroKerrdS}; we prove this in Section~\ref{SecPrelim}, where we also recall the key elements of this framework. In particular, the family of inverses $\wt\Box(\sigma)^{-1}\colon\CI(\wt X)\to\CI(\wt X)$ is a meromorphic family of operators in $\sigma\in\C$,\footnote{Thus, the same is true for $(\td(\sigma)+\tdel(\sigma))^{-1}=(\td(\sigma)+\tdel(\sigma))\wt\Box(\sigma)^{-1}$.} and under the assumption that the inverse family $\wt\Box(\sigma)^{-1}$ verifies suitable bounds as $|\Re\sigma|\to\infty$ and $\Im\sigma>-C$ (for $C>0$ small), one can deduce exponential decay of solutions to $\Box u=0$, up to contributions from a finite dimensional space of \emph{resonances}. Here, resonances are poles of $\wt\Box(\sigma)^{-1}$, and resonant states (for simple resonances) are elements of the kernel of $\wt\Box(\sigma)$ for a resonance $\sigma$.\footnote{The outgoing boundary condition for an element $e^{-i t\sigma}a(x)$ in the kernel of $\Box_g$, with $a(x)$ a $t$-independent section of the form bundle on $M$, is precisely the condition that $e^{-i t\sigma}a(x)=e^{-i t_*\sigma}a_*(x)$ where $a_*$ is \emph{smooth down to} the boundary $\pa\ol X$.} Therefore, proving wave decay and asymptotics is reduced to studying \emph{high energy estimates}, which depend purely on geometric properties of the spacetime and will be further discussed below, and the location of resonances as well as the spaces of resonant states. (For instance, resonances in $\Im\sigma>0$ correspond to exponentially growing solutions and hence are particularly undesirable when studying non-linear problems.) Our main theorem is then:

\begin{thmi}
\label{ThmIntroSummary}
  Let $(M,g)$ be a manifold satisfying the assumptions stated at the beginning of this section. The only resonance of $d+\delta$ in $\Im\sigma\geq 0$ is then $\sigma=0$, and $0$ is a simple resonance. Zero resonant states are smooth, and the space $\wt\cH$ of these resonant states is equal to $\ker\td(0)\cap\ker\tdel(0)$. (In other words, resonant states, viewed as $t_*$-independent differential forms on $\wt M$, are annihilated by $d$ and $\delta$.) Using the grading $\wt\cH=\bigoplus_{k=0}^n\wt\cH^k$ of $\wt\cH$ by form degrees, there is a canonical exact sequence
  \begin{equation}
  \label{EqThmIntroHarmonicCoho}
    0\to H^k(\oX)\oplus H^{k-1}(\oX,\pa\oX) \to \wt\cH^k \to H^{k-1}(\pa\oX).
  \end{equation}

  Furthermore, the only resonance of $\Box$ in $\Im\sigma\geq 0$ is $\sigma=0$. Zero resonant states are smooth, and the space $\wt\cK=\bigoplus_{k=0}^n\wt\cK^k$ of these resonant states, graded by form degree and satisfying $\wt\cK^k\supset\wt\cH^k$, fits into the short exact sequence
  \begin{equation}
  \label{EqThmIntroBoxCoho}
    0 \to H^k(\oX)\oplus H^{k-1}(\oX,\pa\oX) \to \wt\cK^k \to H^{k-1}(\pa\oX) \to 0.
  \end{equation}

  Lastly, the Hodge star operator on $\wt M$ induces natural isomorphisms $\star\colon\wt\cH^k\xra{\cong}\wt\cH^{n-k}$ and $\star\colon\wt\cK^k\xra{\cong}\wt\cK^{n-k}$, $k=0,\ldots,n$.
\end{thmi}

See Theorem~\ref{ThmSummary} for the full statement, including the precise definitions of the maps in the exact sequences. In fact, the various cohomology groups in \eqref{EqThmIntroHarmonicCoho} and \eqref{EqThmIntroBoxCoho} correspond to various types of resonant differential forms, namely forms which are square integrable on $X$ with respect to a natural Riemannian inner product on forms on $M$, induced by the metric obtained by switching the sign in \eqref{EqMetricSpacetime}, that is,
\begin{equation}
\label{EqRiem}
  \alpha^2\,dt^2+h,
\end{equation}
as well as `tangential' and `normal' forms in a decomposition $u=u_T+\alpha^{-1}\,dt\wedge u_N$ of the form bundle corresponding to the warped product structure of the metric. Roughly speaking, \eqref{EqThmIntroBoxCoho} encodes the fact that resonant states for which a certain boundary component vanishes are square integrable with respect to the natural Riemannian inner product on $X$ and can be shown to canonically represent absolute (for tangential forms) or relative (for normal forms) de Rham cohomology of $\oX$, while the aforementioned boundary component is a harmonic form on $Y$ and can be specified freely for resonant states of $\Box$. (Notice by contrast that the last map in the exact sequence \eqref{EqThmIntroHarmonicCoho} for $d+\delta$ is not necessarily surjective.)

The proof of Theorem~\ref{ThmIntroSummary} proceeds in several steps. First, we exclude resonances in $\Im\sigma>0$ in Section~\ref{SubsecAbsenceUpperHalfPlane}; the idea here is to relate the normal operator family of $d+\delta$ (a family of operators on the extended space $\wt X$) to another normal operator family $\hd(\sigma)+\hdel(\sigma)=e^{it\sigma}(d+\delta)e^{-it\sigma}$, which is a family of operators on $X$ that degenerates at $\pa\oX$, but has the advantage of having a simple form in view of the warped product type structure \eqref{EqMetricSpacetime} of the metric: since one formally obtains $\hd(\sigma)+\hdel(\sigma)$ by replacing each $\pa_t$ in the expression for $d+\delta$ by $-i\sigma$, we see that on a formal level $\hd(\sigma)+\hdel(\sigma)$ for purely imaginary $\sigma$ resembles the normal operator family of the Hodge--de Rham operator of the Riemannian metric~\eqref{EqRiem}; then one can show the triviality of $\ker(\hd(\sigma)+\hdel(\sigma))$ in a way that is very similar to how one would show the triviality of $\ker(A+\sigma)$ for self-adjoint $A$ and $\Im\sigma>0$. For not purely imaginary $\sigma$, but still with $\Im\sigma>0$, one can change the tangential part of the metric on $M$ in \eqref{EqMetricSpacetime} by a complex phase and then run a similar argument, using that the resulting `inner product,' while complex, still has some positivity properties. Next, in Section~\ref{SubsecAbsenceNonzeroReal}, we exclude non-zero real resonances by means of a boundary pairing argument, which is a standard technique in scattering theory, see e.g.\ Melrose \cite{MelroseGeometricScattering}. Finally, the analysis of the zero resonance in Section~\ref{SubsecZeroResonance} relies on a boundary pairing type argument, and we again use the Riemannian inner product on forms on $M$. The fact that this Riemannian inner product is singular at $\pa\oX$ implies that resonant states are not necessarily square integrable, and whether or not a state is square integrable is determined by the absence of a certain boundary component of the state. This is a crucial element of the cohomological interpretation of resonant states in Section~\ref{SubsecZeroResCohomology}.

As already alluded to, deducing wave expansions and decay from Theorem~\ref{ThmIntroSummary} requires high energy estimates for the normal operator family. These are easy to obtain if the metric $h$ on $X$ is \emph{non-trapping}, i.e.\ all geodesics escape to $\pa\oX$, as is the case for the static patch of de Sitter space (discussed in Section~\ref{SubsecDS}). Another instance in which suitable estimates hold is when the only trapping within $X$ is \emph{normally hyperbolic trapping}, as is the case for Kerr--de Sitter spaces with parameters in a certain range. (See \cite[\S5.1]{DyatlovResonanceProjectors} for the definition of ($r$-)normal hyperbolicity, and \cite{DyatlovWaveAsymptotics} for details in the Kerr and Kerr--de~Sitter settings.) In the scalar setting, such estimates are now widely available, see for instance Wunsch and Zworski \cite{WunschZworskiNormHypResolvent}, Dyatlov \cite{DyatlovSpectralGaps} and their use in Vasy \cite{VasyMicroKerrdS}: the proof of exponential decay relies on high energy estimates in a strip below the real line. For $\Box$ acting on differential forms, obtaining high energy estimates requires a smallness assumption on the imaginary part of the subprincipal symbol of $\Box$ relative to a \emph{positive definite} inner product on the form bundle; the choice of inner product affects the size of the subprincipal symbol. Conceptually, the natural framework in which to find such an inner product involves \emph{pseudodifferential inner products}. This notion was introduced by Hintz \cite{HintzPsdoInner} and used there to prove high energy estimates for $\Box$ on tensors of arbitrary rank on perturbations of Schwarzschild--de Sitter space. In the present paper, we use the estimates provided in \cite{HintzPsdoInner} as black boxes.

\subsection{Previous and related work}
\label{SubsecPrevious}

The present paper seems to be the first to describe asymptotics for differential forms solving the wave or Hodge--de Rham equation in all form degrees and in this generality, and also the first to demonstrate the forward solvability of non-scalar quasilinear wave equations on black hole spacetimes, but we point out that for applications in general relativity, our results require the cosmological constant to be positive, whereas previous works on Maxwell's equations deal with asymptotically flat spacetimes: Sterbenz and Tataru \cite{SterbenzTataruMaxwellSchwarzschild} showed local energy decay for Maxwell's equations on a class of spherically symmetric asymptotically flat spacetimes including Schwarzschild.\footnote{One needs to assume the vanishing of the electric and the magnetic charge. For a positive cosmological constant, this precisely corresponds to assuming the absence of the $r^{-2}\,dt\wedge dr$ and $\omega$ asymptotics in form degree $2$ in the Schwarzschild--de Sitter case of Theorem~\ref{ThmIntroKDS}.} Blue \cite{BlueMaxwellSchwarzschild} established conformal energy and pointwise decay estimates in the exterior of the Schwarzschild black hole; Andersson and Blue \cite{AnderssonBlueMaxwellKerr} proved similar estimates on slowly rotating Kerr spacetimes. These followed earlier results for Schwarzschild by Inglese and Nicolo \cite{IngleseNicoloMaxwellSchwarzschild} on energy and pointwise bounds for integer spin fields in the far exterior of the Schwarzschild black hole, and by Bachelot \cite{BachelotSchwarzschildScattering}, who proved scattering for electromagnetic perturbations. There are further works which in particular establish bounds for certain components of the Maxwell field, see Donninger, Schlag and Soffer \cite{DonningerSchlagSofferSchwarzschild} and Whiting \cite{WhitingKerrModeStability}. Dafermos \cite{DafermosEinsteinMaxwellScalarStability}, \cite{DafermosBlackHoleNoSingularities} studied the non-linear Einstein-Maxwell-scalar field system under the assumption of spherical symmetry.

Vasy's proof of the meromorphy of the (modified) resolvent of the Laplacian on differential forms on asymptotically hyperbolic spaces \cite{VasyHyperbolicFormResolvent} makes use of the same microlocal framework as the present paper, and it also shows how to link the `intrinsic' structure of the asymptotically hyperbolic space and the form of the Hodge-Laplacian with a `non-degenerately extended' space and operator. For Kerr--de Sitter spacetimes, Dyatlov \cite{DyatlovQNM} defined quasinormal modes or resonances in the same way as they are used here, and obtained exponential decay to constants away from the event horizon for scalar waves. This followed work of Melrose, S\'a Barreto and Vasy \cite{MelroseSaBarretoVasySdS}, where this was shown up to the event horizon of a Schwarzschild--de Sitter black hole, the work of Bony and H\"afner \cite{BonyHaefnerDecay} (following S\'a Barreto and Zworski \cite{SaBarretoZworskiResonances}) on full resonance expansions away from the horizons, and of Dafermos and Rodnianski \cite{DafermosRodnianskiSdS} who proved polynomial decay in this setting. Dyatlov proved exponential decay up to the event horizon for Kerr--de Sitter in \cite{DyatlovQNMExtended}, and significantly strengthened this in \cite{DyatlovAsymptoticDistribution}, obtaining a full resonance expansion for scalar waves.

In the scalar setting too, the wave equation on asymptotically flat spacetimes has received more attention. Dafermos, Rodnianski and Shlapentokh-Rothman \cite{DafermosRodnianskiShlapentokhRothmanDecay}, building on \cite{DafermosRodnianskiKerrBoundedness,DafermosRodnianskiKerrDecaySmall,ShlapentokhRothmanModeStability}, established the decay of scalar waves on all non-extremal Kerr spacetimes, following pioneering work by Kay and Wald \cite{KayWaldSchwarzschild,WaldSchwarzschild} in the Schwarzschild setting. Tataru and Tohaneanu \cite{TataruDecayAsympFlat,TataruTohaneanuKerrLocalEnergy} proved decay and Price's law for slowly rotating Kerr using local energy decay estimates, and Strichartz estimates were proved by Marzuola, Metcalfe, Tataru and Tohaneanu \cite{MarzuolaMetcalfeTataruTohaneanuStrichartz}.

Non-linear results for wave equations on black hole spacetimes include \cite{HintzVasyQuasilinearKdS}, see also the references therein, Luk's work \cite{LukKerrNonlinear} on semilinear forward problems on Kerr, and the scattering construction of dynamical black holes by Dafermos, Holzegel and Rodnianski \cite{DafermosHolzegelRodnianskiKerrBw}. Fully general stability results for Einstein's equations specifically are available for de Sitter space by the works of Friedrich \cite{FriedrichStability}, Anderson \cite{AndersonStabilityEvenDS}, Rodnianski and Speck \cite{RodnianskiSpeckEulerEinsteinDS} and Ringstr\"om \cite{RingstromEinsteinScalarStability}, and for Minkowski space by the work of Christodoulou and Klainerman \cite{ChristodoulouKlainermanStability}, partially simplified and extended by Lindblad and Rodnianski \cite{LindbladRodnianskiGlobalExistence,LindbladRodnianskiGlobalStability}, Bieri and Zipser \cite{BieriZipserStability} and Speck \cite{SpeckEinsteinMaxwell}.

\subsection{Structure of the paper}
\label{SubsecStructure}

In Section~\ref{SecPrelim}, we show how to put the Hodge--de Rham and wave equation on differential forms into the microlocal framework of \cite{VasyMicroKerrdS}; this is used in Section~\ref{SecResonances} in the analysis of resonances in $\Im\sigma\geq 0$, and we prove Theorem~\ref{ThmIntroSummary} there. In Section~\ref{SecApplications}, we apply this result on de Sitter space, where we can take a global point of view which simplifies explicit calculations considerably, and on Schwarzschild--de Sitter space, where such a global picture is not available, but using Theorem~\ref{ThmIntroSummary}, the necessary computations are still very straightforward. In Section~\ref{SecKdS}, we show the perturbation stability of the analysis, in particular deal with Kerr--de Sitter space, and indicate how this gives the forward solvability for quasilinear wave equations on differential forms.

\section{Analytic setup}
\label{SecPrelim}

Recall that we are working on a spacetime $M=\R_t\times X$, equipped with a metric $g$ as in \eqref{EqMetricSpacetime}-\eqref{EqMetricSpacetime2}, where $X$ is the interior of a connected, compact, orientable manifold $\oX$ with non-empty boundary $Y=\pa\oX\neq\emptyset$ and boundary defining function $\alpha\in\CI(\oX)$. Fixing a collar neighborhood of $Y$ identified with $[0,\ep)_\alpha\times Y$, denote by $\Xeven$ the manifold $\oX$ with the smooth structure changed so that only even functions in $\alpha$ are smooth, i.e.\ smooth functions are precisely those for which all odd terms in the Taylor expansion at all boundary components vanish. For brevity, we assume from now on that $Y$ is connected,
\begin{equation}
\label{EqMetricSpaceEven}
  h=\wt\beta(\alpha^2,y)^2\,d\alpha^2+k(\alpha^2,y,dy)
\end{equation}
in a collar neighborhood of $Y$, with $\wt\beta(\alpha^2,y)$ a positive constant at the boundary $\alpha=0$, so $\wt\beta(0,y)=\beta>0$; all of our arguments readily go through in
the case of multiple boundary components. The main examples of spaces which directly fit into this setup are the static patch of de Sitter space (with $1$ boundary component) and Schwarzschild--de Sitter space (with $2$ boundary components); see Section~\ref{SecApplications} for details.

On $M$, we consider the Hodge--de Rham operator $d+\delta$, acting on differential forms. We put its square, the Hodge d'Alembertian
\[
  \Box=(d+\delta)^2,
\]
which is principally scalar, into the microlocal framework developed in \cite{VasyMicroKerrdS}, which we briefly recall below. We shall mainly only make use of two of its consequences: one obtains a precise description of the regularity of resonant states (see Lemma~\ref{LemmaPolesBox} below) and, under additional dynamical hypotheses on the null-geodesic flow (which yield high energy estimates for the operator $\wt\Box(\sigma)$, discussed below), resonance expansions of waves as in Theorem~\ref{ThmIntroKDS}. The reader unfamiliar with the details of~\cite{VasyMicroKerrdS} may simply view these as black boxes; the main results in the present paper are orthogonal to those in the reference.

We renormalize the time coordinate $t$ in the collar neighborhood of $Y$ by writing
\begin{equation}
\label{EqTimeStar}
  t=t_*+F(\alpha),\quad \pa_\alpha F(\alpha)=-\frac{\wt\beta}{\alpha}-2\alpha c(\alpha^2,y)
\end{equation}
with $c$ smooth, hence $F(\alpha)\in-\beta\log\alpha+\CI(\Xeven)$; notice that the above requirement on $F$ only makes sense near $Y$. We introduce the boundary defining function $\mu=\alpha^2$ of $\Xeven$; then one computes
\begin{equation}
\label{EqMetricSpacetimeExt}
  g=\mu\,dt_*^2 - (\wt\beta+2\mu c)\,dt_*\,d\mu + (\mu c^2+\wt\beta c)\,d\mu^2 - k(\mu,y,dy).
\end{equation}
In particular, the determinant of $g$ in these coordinates equals $-\frac{\wt\beta^2}{4}\det(k)$, hence $g$ is non-degenerate up to $Y$. Furthermore, we claim that we can choose $c(\mu,y)$ such that $d t_*$ is timelike on $\R_{t_*}\times\Xeven$; indeed, with $G$ denoting the dual metric to $g$, we require
\begin{equation}
\label{EqTimelikeBdf}
  G(dt_*,dt_*) = -4\wt\beta^{-2}(\mu c^2+\wt\beta c)>0.
\end{equation}
This is trivially satisfied if $c=-\wt\beta/2\mu$, which corresponds to undoing the change of coordinates in \eqref{EqTimeStar}, however we want $c$ to be smooth at $\mu=0$. But for $\mu\geq 0$, \eqref{EqTimelikeBdf} holds provided $-\wt\beta/\mu<c<0$; hence, we can choose a smooth $c$ verifying \eqref{EqTimelikeBdf} in $\mu\geq 0$ and such that moreover $c=-\wt\beta/2\mu$ in $\mu\geq\mu_1$ (intersected with the collar neighborhood of $Y$) for any fixed $\mu_1>0$. Thus, we can choose $F$ as in \eqref{EqTimeStar} with $F=0$ in $\alpha^2\geq\mu_1$ (in particular, $F$ is defined globally on $X$) such that \eqref{EqTimelikeBdf} holds.

Since the metric $g$ in \eqref{EqMetricSpacetimeExt} is stationary ($t_*$-independent) and non-degenerate on $\Xeven$, it can be extended to a stationary Lorentzian metric on an extension $\wt X$ into which $\Xeven$ embeds. Concretely, one defines $\wt X_\delta=(\Xeven\sqcup([-\delta,\eps)_\mu\times Y_y))/\sim$ with the natural smooth structure, where $\sim$ identifies elements of $[0,\eps)_\mu\times Y_y$ with points in $\Xeven$ by means of the collar neighborhood of $Y$. Then, extending $\wt\beta$, $k$, and $c$, and thus $g$, in an arbitrary $t_*$-independent manner to $\wt X_\delta$, the extended metric, which we denote by $\wt g$, is non-degenerate on $\wt X_\delta$, and $d t_*$ remains timelike uniformly on $\R_{t_*}\times\wt X_\delta$, provided one fixes $\delta>0$ to be sufficiently small: indeed, in $\mu<0$, \eqref{EqTimelikeBdf} (with the dual metric $\wt G$ of $\wt g$ in place of $G$) holds for any negative function $c$ as long as $\wt\beta$ remains positive on $\wt X_\delta$. Note that $\mu^{-1}G(d\mu,d\mu)=-4\wt\beta^{-2}<0$ in $\mu>0$, so the level set $\{\mu=-\delta\}$ is \emph{spacelike} for the extended dual metric $\wt G$ if one reduces $\delta>0$ even further (if necessary). We let
\[
  \wt X:=\wt X_\delta
\]
for such a choice of $\delta$. There are two reasons for extending the spacetime a bit beyond the `horizon' $Y$: first, this makes the microlocal \emph{radial point} estimates of \cite{VasyMicroKerrdS} applicable; the microlocal approach is crucial later on, as it is the most stable and straightforward way to obtain the high energy estimates which are needed to deduce an expansion of solutions of the wave equation into quasinormal modes up to exponentially decaying (in $t_*$) remainders---this is discussed in the proofs of Theorems~\ref{ThmDSFull2} and~\ref{ThmSDSDDel2} below. Second, the microlocal framework is stable under perturbations that \emph{do not respect} the warped product structure near $Y$.\footnote{If one is not interested in these two issues, i.e.\ microlocal control and stability under perturbations, one can alternatively use Warnick's approach \cite{WarnickQNMs} to the definition of quasinormal modes.} We remark that instead of the complex absorption in the extension region $\{\mu<0\}\subset\wt X$ which was used in \cite{VasyMicroKerrdS}, we have introduced Cauchy hypersurfaces at $\{\mu=-\delta\}\subset\wt X$ (which may have several connected components) as in\footnote{In the notation of the reference, $\frakt_2=\mu+\delta$, while $\frakt_1$ is only used to define a Cauchy hypersurface $\{\frakt_1=0\}$ where one can impose Cauchy data for the wave or Hodge--de Rham equation; the choice of the latter is very flexible, and we could e.g.\ take $\frakt_1=t_*$.}~\cite[\S2.1.3]{HintzVasySemilinear} and \cite[\S8]{HintzQuasilinearDS}; these are spacelike by construction.

The operator $d+\delta_g$ on $M$ now extends to an operator $d+\delta_{\wt g}$ on $\wt M=\R_{t_*}\times\wt X$. Correspondingly, the wave operator $\Box_g$ on $M$ extends to the wave operator $\Box_{\wt g}$ on $\wt M$. Consider $\Box_{\wt g}$, which is invariant under translations in $t_*$, acting on differential forms which have time dependence $e^{-i t_*\sigma}$; that is, consider the operator
\[
  \wt{\Box}(\sigma) = e^{it_*\sigma}\Box e^{-it_*\sigma}.
\]
(This amounts to formally replacing each $\pa_{t_*}$ in the expression for $\Box$ by $-i\sigma$.) The operator $\wt\Box(\sigma)$ acts on sections of the pullback $\Lambda_{\wt X}\wt M$ of the form bundle $\Lambda\wt M$ under the map $\wt X\to\wt M$, $\wt x\mapsto(0,\wt x)$; Writing differential forms $\wt u$ on $\wt M$ as
\begin{equation}
\label{EqTildeBundleIdentification}
  \wt u=\wt u_T + dt_*\wedge\wt u_N
\end{equation}
with $\wt u_T$ and $\wt u_N$ valued in forms on $\wt X$, we can identify $\Lambda_{\wt X}\wt M$ with $\Lambda\wt X\oplus\Lambda\wt X$.\footnote{At this point, $\wt\Box(\sigma)$ is simply a family of operators depending on $\sigma\in\C$. Its relation to the wave operator $\Box_{\wt g}$ and use for the description of solutions of the wave equation, e.g.\ in the form of partial resonance expansions, requires precise control of $\wt\Box(\sigma)$ as an operator on suitable function spaces as $|\Re\sigma|\to\infty$---these are the high energy estimates mentioned before.}

We also record that $\wt\Box(\sigma)$ is elliptic in $X$: indeed, on $X$, we have
\begin{equation}
\label{EqMellinTransformRelation}
  \wt\Box(\sigma)=e^{-iF\sigma}e^{it\sigma}\Box e^{-it\sigma}e^{iF\sigma}=e^{-iF\sigma}\widehat\Box(\sigma)e^{iF\sigma},
\end{equation}
where $\widehat\Box(\sigma)=e^{it\sigma}\Box e^{-it\sigma}$ is the conjugation of $\Box$ by the Fourier transform in $-t$, and $F$ is as in \eqref{EqTimeStar}; here, we view $\widehat\Box(\sigma)$ as an operator acting on sections of $\Lambda_{\wt X}\wt M|_X$. Now, the latter bundle is isomorphic to $\Lambda X\oplus\Lambda X$, with the isomorphism given by writing differential forms as $u=u_T+dt\wedge u_N$, with $u_T$ and $u_N$ valued in forms on $X$; the relation of the expression of $\widehat\Box(\sigma)$ as a $2\times 2$ block matrix in this bundle decomposition with the decomposition~\eqref{EqTildeBundleIdentification} is given by conjugation by a bundle isomorphism on $\Lambda X\oplus\Lambda X$, which preserves ellipticity.\footnote{Ellipticity is the invertibility of the principal symbol---which in the present case is valued in endomorphisms of $\Lambda_{\wt X}\wt M|_X$---away from the zero section of $T^*X$; invertibility of endomorphisms is preserved by conjugation with isomorphisms.} The principal symbol of $\widehat\Box(\sigma)$ as a second order operator acting on sections of $\Lambda X\oplus\Lambda X$ is given by $(-H)\oplus(-H)$, where $H$ is the dual metric to $h$, here identified with the dual metric function on $T^*X$; this follows from the calculations in the next section. Since $H$ is Riemannian, this implies that $\widehat\Box(\sigma)$, hence $\wt\Box(\sigma)$, is elliptic in $X$.

Consider now $\wt\Box(\sigma)$ as an operator 
\begin{equation}
\label{EqBoxFred}
  \wt\Box(\sigma)\colon\cX^s\to\cY^{s-1},
\end{equation}
where
\[
  \cX^s=\{u\in\Hext^s(\wt X^\circ;\Lambda\wt X\oplus\Lambda\wt X)\colon\wt\Box(\sigma)u\in \cY^{s-1}\},\ \ 
  \cY^{s-1}=\Hext^{s-1}(\wt X^\circ;\Lambda\wt X\oplus\Lambda\wt X);
\]
here, using the notation introduced in~\cite[Appendix~B.2]{HormanderAnalysisPDE}, the bar denotes extendible distributions, i.e.\ $\Hext^s(\wt X^\circ)$ denotes the space of restrictions to $\wt X^\circ$ of $H^s$ functions on a compact manifold without boundary containing $\wt X^\circ$ as an open submanifold. The key result of \cite{VasyMicroKerrdS} is that for any fixed $C\in\R$ and for regularity above a certain threshold,
\begin{equation}
\label{EqThreshold}
  s>1/2+\hat\beta-\beta C,\quad \Im\sigma>-C,
\end{equation}
the operator~\eqref{EqBoxFred} is Fredholm, and indeed invertible for $\Im\sigma\gg 1$: this is \cite[Theorem~1.2]{VasyMicroKerrdS} with $Q_\sigma=0$, $\lambda=0$, $P_\sigma=\wt\Box(\sigma)$. The additional shift $\hat\beta$ is due to the fact that $\wh\Box(\sigma)$, $\sigma\in\R$, is not symmetric with respect to a \emph{positive definite} fiber inner product on the form bundle; see the proof of \cite[Propositions~2.3 and 2.4]{VasyMicroKerrdS}, esp.\ Equation~(2.15) there, for the contribution of $\wh\Box(\sigma)^*-\wh\Box(\sigma)$ to the radial point estimate, as well as \cite[Footnote~5]{HintzVasySemilinear} for the spacetime version of this estimate, i.e.\ prior to conjugating by the Fourier transform in $t_*$. The Fredholm property of~\eqref{EqBoxFred} follows from the ellipticity of $\wt\Box(\sigma)$ in $X$, from real principal type propagation estimates (or more simply, energy estimates) in the extension region $\{\mu<0\}$ where $\wt\Box(\sigma)$ is a hyperbolic (wave-type) operator, and the radial point estimates at $N^*\{\mu=0\}$ which use the source/sink nature of the Hamilton flow of the principal symbol of $\wt\Box(\sigma)$ there (this is the phase space manifestation of the classical red-shift effect); see \cite[\S4.8]{VasyMicroKerrdS} for a verification of these facts for metrics of the form~\eqref{EqMetricSpacetimeExt}.\footnote{See also~\cite[\S2.2]{VasyMicroKerrdS}, where $\Lambda_\pm=\{\mp c\,d\mu\colon c>0\}$ denotes the two components of $N^*\{\mu=0\}$, for a precise description of the relevant dynamical properties: $\Lambda_-$ is a source, $\Lambda_+$ a sink for the Hamilton flow of the principal symbol of $\wt\Box(\sigma)$.}

We remark that since $\wt\Box(\sigma)=(\td(\sigma)+\tdel(\sigma))^2\colon\cX^s\to\cY^{s-1}$ is an analytic family of Fredholm operators with meromorphic inverse, the map
\[
  (\td(\sigma)+\tdel(\sigma))^{-1} := (\td(\sigma)+\tdel(\sigma))\wt\Box(\sigma)^{-1} \colon \cY^{s-1}\to\cY^{s-1}
\]
is meromorphic as well for the same $s$ and $\sigma$ as above, and is a right inverse (away from its poles) of
\begin{equation}
\label{EqHodgeInv}
  \td(\sigma)+\tdel(\sigma)\colon\cZ^{s-1}\to\cY^{s-1},
\end{equation}
where $\cZ^{s-1}=\{u\in\Hext^{s-1}(\wt X;\Lambda\wt X\oplus\Lambda\wt X)\colon(\td(\sigma)+\tdel(\sigma))u\in\cY^{s-1}\}$. Increasing the lower bound on $s$ required in~\eqref{EqThreshold} by $1$, an element $u\in\ker_{\cY^{s-1}}\td(\sigma)+\tdel(\sigma)$ of course satisfies $u\in\cX^{s-1}$, hence $u\in\ker_{\cX^{s-1}}\wt\Box(\sigma)$, thus has above threshold regularity, which means it lies in a finite-dimensional space. Thus, with this increased requirement on $s$, the map~\eqref{EqHodgeInv} is Fredholm, invertible for $\Im\sigma\gg 1$, and satisfies high energy estimates provided $\wt\Box(\sigma)$ does.

To summarize this to the extent needed in the sequel, $\wt\Box(\sigma)$ is an analytic family of Fredholm operators on suitable function spaces, and the inverse family $\wt\Box(\sigma)^{-1}\colon\CI(\wt X;\Lambda\wt X\oplus\Lambda\wt X)\to\CI(\wt X;\Lambda\wt X\oplus\Lambda\wt X)$ (where we use the identification \eqref{EqTildeBundleIdentification}) admits a meromorphic continuation from $\Im\sigma\gg 0$ to the complex plane. Moreover (see \cite[Lemma~3.5]{VasyMicroKerrdS}), the Laurent coefficient at the poles are finite rank operators mapping sufficiently regular distributions to elements of $\CI(\wt X;\Lambda\wt X\oplus\Lambda\wt X)$. Note however that without further assumptions on the geodesic flow (for instance, semiclassical non-trapping or normally hyperbolic trapping), we do not obtain any high energy bounds, i.e.\ polynomial (in $\sigma$) estimates on the operator norm of $\wh\Box(\sigma)^{-1}\colon\cY^{s-1}\to\cX^s$ when $|\Re\sigma|\to\infty$ and $\Im\sigma\geq-C$ for (suitable) $C>0$.

\begin{lemma}
\label{LemmaPolesBox}
  A complex number $\sigma\in\C$ is a resonance of $\Box$, i.e.\ $\wt\Box(\sigma)^{-1}$ has a pole at $\sigma$, if and only if there exists a non-zero $u\in\alpha^{-i\beta\sigma}\CI(\Xeven;\Lambda\Xeven\oplus\Lambda\Xeven)$ (using the identification \eqref{EqTildeBundleIdentification}) such that $\widehat\Box(\sigma)u=0$.
\end{lemma}
\begin{proof}
  If $\sigma\in\C$ is a resonance, then there exists a non-zero $\wt
  u\in\CI(\wt X;\Lambda\wt X\oplus\Lambda\wt X)$ with
  $\wt\Box(\sigma)\wt u=0$. Restricting to $\oX$, this implies by
  \eqref{EqMellinTransformRelation} and \eqref{EqTimeStar} that
  $\widehat\Box(\sigma)u=0$ for $u=e^{iF\sigma}\wt
  u|_{\oX}\in\alpha^{-i\beta\sigma}\CI(\Xeven;\Lambda\Xeven\oplus\Lambda\Xeven)$. If
  $u=0$, then $\wt u$ vanishes to infinite order at $Y$, and since
  $\wt\Box(\sigma)$ is a conjugate of a wave or Klein-Gordon operator
  on an asymptotically de Sitter space, see
  \cite{VasyMinkDSHypRelation}, unique continuation at infinity on the
  de Sitter side as in \cite[Proposition~5.3]{VasyWaveOndS} (which is
  in the scalar setting, but works similarly in the present context
  since it relies on a semiclassical argument in which only the
  principal symbol of the wave operator matters, and this is the same
  in our setting) shows that $\wt u\equiv 0$ on $\wt X$. This is the place where we use that we capped off $\wt X$ outside of $\Xeven$ by a Cauchy hypersurface: (pseudodifferential) complex absorption in principle would have the mildly undesirable effect of allowing for the existence of resonant states supported in $\wt X\setminus\ol X$, see~\cite[Proposition~3.9]{VasyMicroKerrdS}. Hence, $u\neq 0$, as desired.

  Conversely, given a $u\in\alpha^{-i\beta\sigma}\CI(\Xeven;\Lambda\Xeven\oplus\Lambda\Xeven)$ with $\widehat\Box(\sigma)u=0$, we define $\wt u'\in\CI(\wt X;\Lambda\wt X\oplus\Lambda\wt X)$ to be any smooth extension of $e^{-iF\sigma}u$ from $\Xeven$ to $\wt X$. Then $\wt\Box(\sigma)\wt u'$ is identically zero in $X$ and thus vanishes to infinite order at $Y$; hence, we can solve
  \[
    \wt\Box(\sigma)\wt v=-\wt\Box(\sigma)\wt u'
  \]
  in $\wt X\setminus X$ with $\wt v$ vanishing to infinite order at $Y$: this is a wave equation on an asymptotically de~Sitter space, as mentioned above, hence solvability is provided by \cite[Proposition~3.4 and Corollary~3.6]{VasyWaveOndS}. Thus, extending $\wt v$ by $0$ to $X$, we find that $\wt u=\wt u'+\wt v$ is a non-zero solution to $\wt\Box(\sigma)\wt u=0$ on $\wt X$.
\end{proof}

Since $\Box=(d+\delta)^2$, we readily obtain the following analogue of Lemma~\ref{LemmaPolesBox} for $d+\delta$:

\begin{lemma}
\label{LemmaPoles}
  The map\footnote{We drop the bundles from the notation for simplicity.} $\ker_{\CI(\wt X)}(\td(\sigma)+\tdel(\sigma))\to\ker_{\alpha^{-i\beta\sigma}\CI(\Xeven)}(\hd(\sigma)+\hdel(\sigma))$, $\wt u\mapsto e^{iF\sigma}\wt u|_X$, is an isomorphism. 
\end{lemma}
\begin{proof}
  Since $\wt u\in\ker(\td(\sigma)+\tdel(\sigma))$ implies $\wt u\in\ker\wt\Box(\sigma)$, injectivity follows from the proof of Lemma~\ref{LemmaPolesBox}. To show surjectivity, take $u\in e^{iF\sigma}\CI(\Xeven)$ with $(\hd(\sigma)+\hdel(\sigma))u=0$ and choose any smooth extension $\wt u'$ of $e^{-iF\sigma}u$ to $\wt X$. Solving $\wt\Box(\sigma)\wt v'=-(\td(\sigma)+\tdel(\sigma))\wt u'$ with $\supp\wt v'\subset\wt X\setminus X$ and then defining $\wt v=(\td(\sigma)+\tdel(\sigma))\wt v'$, we see that $\wt u=\wt u'+\wt v$ extends $\wt u'$ to $\wt X$ and is annihilated by $\td(\sigma)+\tdel(\sigma)$.
\end{proof}

Thus, when studying the location and structure of resonances, we already have very precise information about regularity and asymptotics (on $X$) of potential resonant states.

\section{Resonances in \texorpdfstring{$\Im\sigma\geq 0$}{the closed upper half plane}}
\label{SecResonances}

Using Lemma~\ref{LemmaPoles}, we now study the resonances of in $\Im\sigma\geq 0$ by analyzing the operator $\hd(\sigma)+\hdel(\sigma)$ (and related operators) on $\Xeven$. Recall that a resonance at $\sigma\in\C$ and a corresponding resonant state $\wt u$ yield a solution $(d+\delta)(e^{-it_*\sigma}\wt u)=0$, hence $\Im\sigma>0$ implies in view of $|e^{-it_*\sigma}|=e^{t_*\Im\sigma}$ that $e^{-it_*\sigma}\wt u$ grows exponentially in $t_*$, whereas resonances with $\Im\sigma=0$ yield solutions which at most grow polynomially in $t_*$ (and do not decay). \emph{We will continue to drop the metric $g$ or $\wt g$ from the notation for brevity.}

In order to keep track of fiber inner products and volume densities, we will use the following notation.
\begin{definition}
\label{DefL2Spaces}
  For a density $\mu$ on $X$ and a complex vector bundle $\cE\to X$ equipped with a positive definite Hermitian form $B$, let $L^2(X,\mu;\cE,B)$ be the space of all sections $u$ of $\cE$ for which $\|u\|_{\mu,B}^2:=\int_X B(u,u)\,d\mu<\infty$.

  If $B$ is merely assumed to be sesquilinear (but not necessarily positive definite), we define the pairing
  \[
    \la u,v\ra_{\mu,B} := \int_X B(u,v)\,d\mu
  \]
  for all sections $u,v$ of $\cE$ for which $B(u,v)\in L^1(X,\mu)$. If the choice of the density $\mu$ or inner product $B$ is clear from the context, it will be dropped from the notation.
\end{definition}

\begin{rmk}
  It will always be clear what bundle $\cE$ we are using at a given time, so $\cE$ will from now on be dropped from the notation; also, $X$ will mostly be suppressed.
\end{rmk}

Since the metric $g$ in \eqref{EqMetricSpacetime} has a warped product structure and $\alpha\,dt$ has unit squared norm, it is natural to write differential forms on $M=\R_t\times X_x$ as
\begin{equation}
\label{EqFormDecomposition}
  u(t,x)=u_T(t,x)+\alpha\,dt\wedge u_N(t,x),
\end{equation}
where the tangential and normal forms $u_T$ and $u_N$ are $t$-dependent forms on $X$, and we will often write this as
\[
  u(t,x)=\begin{pmatrix}u_T(t,x) \\ u_N(t,x) \end{pmatrix}.
\]
Thus, the differential $d$ on $M$ is given in terms of the differential $d_X$ on $X$ by
\begin{equation}
\label{EqSpacetimeD}
  d = \begin{pmatrix} d_X & 0 \\ \alpha^{-1}\pa_t & -\alpha^{-1}d_X\alpha \end{pmatrix}.
\end{equation}
Since the dual metric is given by $G=\alpha^{-2}\pa_t^2-H$, the fiber inner product $G_k$ on $k$-forms is given by
\begin{equation}
\label{EqSpacetimeFiberMetric}
  G_k = \begin{pmatrix} (-1)^k H_k & 0 \\ 0 & (-1)^{k-1}H_{k-1} \end{pmatrix},
\end{equation}
where $H_q$ denotes the fiber inner product on $q$-forms on $X$. Furthermore, the volume density on $M$ is $|dg|=\alpha|dt\,dh|$, and we therefore compute the $L^2(M,|dg|)$-adjoint of $d$ to be
\begin{equation}
\label{EqSpacetimeDel}
  \delta = \begin{pmatrix} -\alpha^{-1}\delta_X\alpha & -\alpha^{-1}\pa_t \\ 0 & \delta_X \end{pmatrix},
\end{equation}
where $\delta_X$ is the $L^2(X,|dh|;\Lambda X,H)$-adjoint of $d_X$; the signs here are due to the signs in~\eqref{EqSpacetimeFiberMetric} which depend on the form degree. Thus,
\begin{equation}
\label{EqDiffHat}
  \hd(\sigma) = \begin{pmatrix} d_X & 0 \\ -i\sigma\alpha^{-1} & -\alpha^{-1}d_X\alpha \end{pmatrix}, \quad
  \hdel(\sigma) = \begin{pmatrix} -\alpha^{-1}\delta_X\alpha & i\sigma\alpha^{-1} \\ 0 & \delta_X \end{pmatrix}.
\end{equation}

In the course of our arguments we will need to justify various integrations by parts and boundary pairing arguments. This requires a precise understanding of the asymptotics of $u_T$ and $u_N$ for potential resonant states $u$ at $Y=\pa\Xeven$. To this end, we further decompose the bundle $\Lambda X\oplus\Lambda X$ near $Y$ by writing $u_T$ as
\begin{equation}
\label{EqBundleDecNearBdy}
  u_T=u_{TT}+d\alpha\wedge u_{TN}
\end{equation}
and similarly for $u_N$, hence
\begin{equation}
\label{EqBundleDecWarped}
  u = u_{TT} + d\alpha\wedge u_{TN} + \alpha\,dt\wedge u_{NT} + \alpha\,dt\wedge d\alpha\wedge u_{NN},
\end{equation}
where the $u_{\bullet\bullet}$ are forms on $X$ valued in $\Lambda Y$. Now for a resonant state $u$, we have
\begin{equation}
\label{EqBundleDecExtended}
  u=\alpha^{-i\beta\sigma}(\wt u'_{TT} + d(\alpha^2)\wedge\wt u'_{TN} + dt_*\wedge\wt u'_{NT} + dt_*\wedge d(\alpha^2)\wedge\wt u'_{NN})
\end{equation}
near $Y$ with $\wt u'_{\bullet\bullet}\in\CI(\Xeven;\Lambda Y)$, which we rewrite in terms of the decomposition \eqref{EqBundleDecWarped} using \eqref{EqTimeStar}, obtaining
\begin{align*}
  u=\alpha^{-i\beta\sigma}\bigl(&\wt u'_{TT} + d\alpha \wedge (2\alpha\wt u'_{TN}-F'(\alpha)\wt u'_{NT}) \\
    & \qquad + \alpha\,dt\wedge\alpha^{-1}\wt u'_{NT} + 2\alpha\,dt\wedge d\alpha\wedge\wt u'_{NN}\bigr);
\end{align*}
hence introducing the `change of basis' matrix
\[
  \sC=
	  \begin{pmatrix}
			1&0&0&0 \\
			0&\alpha&\beta\alpha^{-1}&0 \\
			0&0&\alpha^{-1}&0 \\
			0&0&0&1
	  \end{pmatrix}
\]
and defining the space
\begin{equation}
\label{EqDefCIsigma}
  \CI_{(\sigma)} := \sC\alpha^{-i\beta\sigma}\begin{pmatrix}\CI(\Xeven;\Lambda Y)\\ \CI(\Xeven;\Lambda Y)\\ \CI(\Xeven;\Lambda Y)\\ \CI(\Xeven;\Lambda Y)\end{pmatrix}
  \subset
  	  \begin{pmatrix}
 		\alpha^{-i\beta\sigma}\CI(\Xeven;\Lambda Y) \\
 		\alpha^{-i\beta\sigma-1}\CI(\Xeven;\Lambda Y) \\
 		\alpha^{-i\beta\sigma-1}\CI(\Xeven;\Lambda Y) \\
 		\alpha^{-i\beta\sigma}\CI(\Xeven;\Lambda Y)
	  \end{pmatrix},
\end{equation}
we obtain
\begin{equation}
\label{EqResState}
  \begin{pmatrix}u_{TT}\\u_{TN}\\u_{NT}\\u_{NN}\end{pmatrix}
    = \sC
	  \alpha^{-i\beta\sigma}
	  \begin{pmatrix}\wt u_{TT}\\\wt u_{TN}\\\wt u_{NT}\\\wt u_{NN}\end{pmatrix}
  \in\CI_{(\sigma)}
\end{equation}
with $\wt u_{\bullet\bullet}\in\CI(\Xeven;\Lambda Y)$, where the $u_{\bullet\bullet}$ are the components of $u$ in the decomposition \eqref{EqBundleDecWarped}.

We will also need the precise form of $\hd(\sigma)$ and $\hdel(\sigma)$ near $Y$. Since in the decomposition \eqref{EqBundleDecNearBdy}, the fiber inner product on $\Lambda X$-valued forms is $H=K\oplus\wt\beta^{-2}K$ in view of \eqref{EqMetricSpaceEven}, we have
\begin{equation}
\label{EqDifferentialNearY}
  d_X=\begin{pmatrix}d_Y&0\\ \pa_\alpha&-d_Y\end{pmatrix}
    \quad\tn{and}\quad
  \delta_X=\begin{pmatrix}\delta_Y & \pa_\alpha^* \\ 0& -\wt\beta^2\delta_Y\wt\beta^{-2}\end{pmatrix},
\end{equation}
where $d_Y$ is the differential on $Y$ and $\pa_\alpha^*$ is the formal adjoint of $\pa_\alpha\colon\CI(X;\Lambda Y)\subset L^2(X,|dh|;\Lambda Y,K)\to L^2(X,|dh|;\Lambda Y,\wt\beta^{-2}K)$. Thus, if $\wt\beta$ and $k$ are independent of $\alpha$ near $Y$, we simply have $\pa_\alpha^*=-\beta^{-2}\pa_\alpha$, and in general
\begin{equation}
\label{EqPaAAdj}
  \pa_\alpha^*=-\beta^{-2}\pa_\alpha+\alpha^2 p_1\pa_\alpha+\alpha p_2,\quad p_1,p_2\in\CI(\Xeven).
\end{equation}

Finally, we compute the form of $\hd(\sigma)$ near $Y$ acting on forms as in \eqref{EqResState}:
\begin{equation}
\label{EqDifferentialC}
  \hd(\sigma)\sC=
    \begin{pmatrix}
	  d_Y & 0 & 0 & 0 \\
	  \pa_\alpha & -\alpha d_Y & -\beta\alpha^{-1}d_Y & 0 \\
	  -i\sigma\alpha^{-1} & 0 & -\alpha^{-1}d_Y & 0 \\
	  0 & -i\sigma & -i\sigma\beta\alpha^{-2}-\alpha^{-1}\pa_\alpha & d_Y
	\end{pmatrix}.
\end{equation}
Thus, applying $\hd(\sigma)$ to $u\in\CI_{(\sigma)}$ yields an element
\[
  \hd(\sigma)u
  \in \begin{pmatrix}
 		\alpha^{-i\beta\sigma}\CI(\Xeven;\Lambda Y) \\
 		\alpha^{-i\beta\sigma-1}\CI(\Xeven;\Lambda Y) \\
 		\alpha^{-i\beta\sigma-1}\CI(\Xeven;\Lambda Y) \\
 		\alpha^{-i\beta\sigma}\CI(\Xeven;\Lambda Y)
	  \end{pmatrix},
\]
where we use that there is a cancellation in the $(4,3)$ entry of $\hd(\sigma)\sC$ in view of $(i\sigma\beta\alpha^{-2}+\alpha^{-1}\pa_\alpha)\alpha^{-i\beta\sigma}=0$; without this cancellation, the fourth component of $\hd(\sigma)u$ would only lie in $\alpha^{-i\beta\sigma-2}\CI(\Xeven;\Lambda Y)$. Similarly, we compute
\begin{equation}
\label{EqCodifferentialC}
  \hdel(\sigma)\sC=
    \begin{pmatrix}
	  -\delta_Y & -\alpha^{-1}\pa_\alpha^*\alpha^2 & -\beta\alpha^{-1}\pa_\alpha^*+i\sigma\alpha^{-2} & 0 \\
	  0 & \alpha\wt\beta^2\delta_Y\wt\beta^{-2} & \beta\alpha^{-1}\wt\beta^2\delta_Y\wt\beta^{-2} & i\sigma\alpha^{-1} \\
	  0 & 0 & \alpha^{-1}\delta_Y & \pa_\alpha^* \\
	  0 & 0 & 0 & -\wt\beta^2\delta_Y\wt\beta^{-2}
	\end{pmatrix},
\end{equation}
thus applying $\hdel(\sigma)$ to $u\in\CI_{(\sigma)}$ also gives an element
\[
  \hdel(\sigma)u
  \in \begin{pmatrix}
 		\alpha^{-i\beta\sigma}\CI(\Xeven;\Lambda Y) \\
 		\alpha^{-i\beta\sigma-1}\CI(\Xeven;\Lambda Y) \\
 		\alpha^{-i\beta\sigma-1}\CI(\Xeven;\Lambda Y) \\
 		\alpha^{-i\beta\sigma}\CI(\Xeven;\Lambda Y)
	  \end{pmatrix},
\]
where there is again a cancellation in the $(1,3)$ entry of $\hdel(\sigma)\sC$; without this cancellation, the first component of $\hdel(\sigma)u$ would only lie in $\alpha^{-i\beta\sigma-2}\CI(\Xeven;\Lambda Y)$.

In fact, a bit more is true: namely, one checks that\footnote{Either, this follows by a direct computation; or one notes that these operators are equal (up to a smooth phase factor) to the matrices of the Fourier transforms in $t_*$ of $d$ and $\delta$ with respect to the form decomposition \eqref{EqBundleDecExtended}, which are smooth on the extended manifold $\wt X$.}  $\alpha^{i\beta\sigma}\sC^{-1}\hd(\sigma)\sC\alpha^{-i\beta\sigma}$ and $\alpha^{i\beta\sigma}\sC^{-1}\hdel(\sigma)\sC\alpha^{-i\beta\sigma}$ preserve the space $\CI(\Xeven;\Lambda Y)^4$ (in the decomposition \eqref{EqBundleDecExtended}), hence if $u\in\CI_{(\sigma)}$, then also $\hd(\sigma)u,\hdel(\sigma)u\in\CI_{(\sigma)}$. Since it will be useful later, we check this explicitly for $\sigma=0$ by computing
\begin{equation}
\label{EqDifferentialCinvC}
  \sC^{-1}\hd(0)\sC
   = \begin{pmatrix}
       d_Y & 0 & 0 & 0 \\
	   \alpha^{-1}\pa_\alpha & -d_Y & 0 & 0 \\
	   0 & 0 & -d_Y & 0 \\
	   0 & 0 & -\alpha^{-1}\pa_\alpha & d_Y
     \end{pmatrix}
\end{equation}
and
\begin{equation}
\label{EqCodifferentialCinvC}
  \sC^{-1}\hdel(0)\sC
   = \begin{pmatrix}
       -\delta_Y & -\alpha^{-1}\pa_\alpha^*\alpha^2 & -\alpha^{-1}\pa_\alpha^*\beta & 0 \\
	   0 & \wt\beta^2\delta_Y\wt\beta^{-2} & \beta\alpha^{-2}\wt\beta^2[\delta_Y,\wt\beta^{-2}] & -\beta\alpha^{-1}\pa_\alpha^* \\
	   0 & 0 & \delta_Y & \alpha\pa_\alpha^* \\
	   0 & 0 & 0 & -\wt\beta^2\delta_Y\wt\beta^{-2}
     \end{pmatrix}.
\end{equation}

\subsection{Absence of resonances in \texorpdfstring{$\Im\sigma>0$}{the upper half plane}}
\label{SubsecAbsenceUpperHalfPlane}

The fiber inner product on the form bundle is not positive definite, thus we cannot use standard arguments for (formally) self-adjoint operators to exclude a non-trivial kernel of $\hd(\sigma)+\hdel(\sigma)$. We therefore introduce a different inner product (by which we mean here a non-degenerate sesquilinear form), related to the natural inner product induced by the metric, which does have some positivity properties. Concretely, for $\theta\in(-\pi/2,\pi/2)$, we use the inner product $H\oplus e^{-2i\theta}H$, i.e.\ on pure degree $k$-forms on $M$, the fiber inner product is given by $H_k\oplus e^{-2i\theta}H_{k-1}$ in the decomposition into tangential and normal components as in \eqref{EqFormDecomposition}.

\begin{lemma}
\label{LemmaThetaNondegeneracy}
  Let $\theta\in(-\pi/2,\pi/2)$. Suppose that $u\in L^2(\alpha|dh|;H\oplus H)$ is such that $\la u,u\ra_{H\oplus e^{-2i\theta}H}=0$. Then $u=0$.
\end{lemma}
\begin{proof}
  With $u=u_T+\alpha\,dt\wedge u_N$, we have $\|u_T\|^2_{L^2(\alpha|dh|;H)} + e^{-2i\theta}\|u_N\|^2_{L^2(\alpha|dh|;H)}=0$. Multiplying this equation by $e^{i\theta}$ and taking real parts gives
  \[
    \cos(\theta)\|u\|^2_{L^2(\alpha|dh|;H\oplus H)}=0,
  \]
  hence $u=0$, since $\cos\theta>0$ for $\theta$ in the given range.
\end{proof}

Using the volume density $\alpha|dh|$ to compute adjoints,\footnote{See Definition~\ref{DefL2Spaces} for the notation used here.} we have
\[
  \la\hd(\sigma)u,v\ra_{H\oplus e^{-2i\theta}H} = \la u,\widehat{\delta_\theta}(\sigma)v\ra_{H\oplus e^{-2i\theta}H},\quad u,v\in\CIc(X;\Lambda X\oplus\Lambda X)
\]
for the operator
\[
  \widehat{\delta_\theta}(\sigma)
    = \begin{pmatrix}
		\alpha^{-1}\delta_X\alpha & ie^{2i\theta}\bar\sigma\alpha^{-1} \\
		0 & -\delta_X
	  \end{pmatrix},
\]
which equals $-\hdel(\sigma)$ provided $e^{2i\theta}\bar\sigma=-\sigma$, i.e.\ $\sigma\in e^{i\theta}\cdot i(0,\infty)$.

\begin{rmk}
  Since the inner product $H\oplus e^{-2i\theta}H$ is not Hermitian, we do \emph{not} have $\la\widehat{\delta_\theta}(\sigma)u,v\ra_{H\oplus e^{-2i\theta}H}=\la u,\hd(\sigma)v\ra_{H\oplus e^{-2i\theta}H}$ in general. Rather, one computes
  \begin{equation}
  \label{EqDeltaThetaAdjoint}
  \begin{split}
    \la\widehat{\delta_\theta}(\sigma)u,v&\ra_{H\oplus e^{2i\theta}H} = \overline{\la v,\widehat{\delta_\theta}(\sigma)u\ra_{H\oplus e^{-2i\theta}H}} \\
	& = \overline{\la\hd(\sigma)v,u\ra_{H\oplus e^{-2i\theta}H}} = \la u,\hd(\sigma)v\ra_{H\oplus e^{2i\theta}H}.
  \end{split}
  \end{equation}
\end{rmk}

Now suppose $u\in\CI_{(\sigma)}$ is a solution, with $\im\sigma>0$, of
\begin{equation}
\label{EqDplusDel}
  (\hd(\sigma)+\hdel(\sigma))u=0.
\end{equation}
We claim that every such $u$ must vanish. To show this, we apply $\hd(\sigma)$ to \eqref{EqDplusDel} and pair the result with $u$; this gives
\begin{equation}
\label{EqIBPCodifferential}
\begin{split}
  0&=\la\hd(\sigma)\hdel(\sigma)u,u\ra_{H\oplus e^{-2i\theta}H} = \la\hdel(\sigma)u,\widehat{\delta_\theta}(\sigma)u\ra_{H\oplus e^{-2i\theta}H} \\
    &= -\la\hdel(\sigma)u,\hdel(\sigma)u\ra_{H\oplus e^{-2i\theta}H},
\end{split}
\end{equation}
where we choose $\theta\in(-\pi/2,\pi/2)$ so that $\sigma\in e^{i\theta}\cdot i(0,\infty)$; the integration by parts will be justified momentarily. By Lemma~\ref{LemmaThetaNondegeneracy}, this implies $\hdel(\sigma)u=0$. On the other hand, applying $\hdel(\sigma)$ to \eqref{EqDplusDel} and using \eqref{EqDeltaThetaAdjoint}, we get, for $\sigma\in e^{i\theta}\cdot i(0,\infty)$,
\begin{equation}
\label{EqIBPDifferential}
\begin{split}
  0&=\la\hdel(\sigma)\hd(\sigma)u,u\ra_{H\oplus e^{2i\theta}H} = -\la\widehat{\delta_\theta}(\sigma)\hd(\sigma)u,u\ra_{H\oplus e^{2i\theta}H} \\
  &= -\la\hd(\sigma)u,\hd(\sigma)u\ra_{H\oplus e^{2i\theta}H},
\end{split}
\end{equation}
hence $\hd(\sigma)u=0$ by Lemma~\ref{LemmaThetaNondegeneracy}, again modulo justifying the integration by parts.

Using the splitting \eqref{EqFormDecomposition} and the form \eqref{EqDiffHat} of $\hd(\sigma)$, the second component of the equation $\hd(\sigma)u=0$ gives $i\sigma u_T+d_X\alpha u_N=0$. Taking the $L^2(\alpha|dh|;H)$-pairing of this with $u_T$ gives (the integration by parts to be justified below)
\begin{equation}
\label{EqIBPVanishing}
  0=i\sigma\|u_T\|^2 + \la d_X\alpha u_N,u_T\ra = i\sigma\|u_T\|^2 + \la u_N,\delta_X\alpha u_T\ra,
\end{equation}
and then the first component of $\hdel(\sigma)u=0$, i.e.\ $\delta_X\alpha u_T=i\sigma u_N$, can be used to rewrite the pairing on the right hand side; we obtain $0=i(\sigma\|u_T\|^2 - \bar\sigma\|u_N\|^2)$. Writing $\sigma=ie^{i\theta}\wt\sigma$ with $\wt\sigma>0$ real, this becomes
\begin{equation}
\label{EqAbsenceImPosQualitative}
  0=\wt\sigma(e^{i\theta}\|u_T\|^2 + e^{-i\theta}\|u_N\|^2),
\end{equation}
and taking the real part of this equation gives $u_T=0=u_N$, hence $u=0$.

We now justify the integrations by parts used in \eqref{EqIBPCodifferential} and \eqref{EqIBPDifferential}, which is only an issue at $Y$. First of all, since $u\in\CI_{(\sigma)}$ and $\Im\sigma>0$, the pairings are well-defined in the strong sense that all functions which appear in the pairings are elements of $L^2(\alpha|dh|;H\oplus H)$; in fact, all functions in these pairings lie in $\CI_{(\sigma)}$. In view of the block structure $H\oplus e^{-2i\theta}H=K\oplus\wt\beta^{-2}K\oplus e^{-2i\theta}K\oplus\wt\beta^{-2}e^{-2i\theta}K$ of the inner product, the only potentially troublesome term for the integration by parts is the pairing of the first components, since this is where we need the cancellation of two too singular summands mentioned after \eqref{EqCodifferentialC} to ensure that $\hdel(\sigma)u\in L^2$. Integrating by parts separately in each of the summands of one factor of the $L^2$ pairing, one can only use the cancellation in (i.e.\ the $L^2$-membership of) the other factor; that is, we integrate by parts in a pairing (of the first components) of an element of $\alpha^{-i\beta\sigma}\CI(\Xeven;\Lambda Y)$ (using the cancellation) with one in $\alpha^{-i\beta\sigma-2}\CI(\Xeven;\Lambda Y)$ (not using the cancellation), thus this pairing is still absolutely integrable and the integration by parts is justified. Likewise, the integration by parts used in \eqref{EqIBPDifferential} only has potential issues in the pairing of the fourth components, since we need the cancellation mentioned after \eqref{EqDifferentialC} to ensure that $\hd(\sigma)u\in L^2$. But again, if we only use this cancellation in one of the terms, we pair $\alpha^{-i\beta\sigma}\CI(\Xeven;\Lambda Y)$ against $\alpha^{-i\beta\sigma-2}\CI(\Xeven;\Lambda Y)$, which is absolutely integrable.

In order to justify \eqref{EqIBPVanishing}, we observe using \eqref{EqDifferentialNearY} that near $Y$,
\[
  u_T,d_X\alpha u_N\in\begin{pmatrix}\alpha^{-i\beta\sigma}\CI\\ \alpha^{-i\beta\sigma-1}\CI\end{pmatrix},\quad
  u_N,\delta_X\alpha u_T\in\begin{pmatrix}\alpha^{-i\beta\sigma-1}\CI \\ \alpha^{-i\beta\sigma}\CI\end{pmatrix},
\]
where we write $\CI=\CI(\Xeven;\Lambda Y)$. These membership statements do not rely on any cancellations, and since all these functions are in $L^2(\alpha|dh|;\Lambda Y,K)$ near $Y$, the integration by parts in \eqref{EqIBPVanishing} is justified.

We summarize the above discussion and extend it to a quantitative version:
\begin{prop}
\label{PropAbsenceImPos}
  There exists a constant $C>0$ such that for all $\sigma\in\C$ with $\Im\sigma>0$, we have the following estimate for $u\in\CI_{(\sigma)}$:
  \begin{equation}
  \label{EqAbsenceImPosQuant}
    \|u\|_{L^2(\alpha|dh|;H\oplus H)} \leq C\frac{|\sigma|}{|\Im\sigma|^2}\|(\hd(\sigma)+\hdel(\sigma))u\|_{L^2(\alpha|dh|;H\oplus H)}.
  \end{equation}
\end{prop}
\begin{proof}
  Write $\sigma=ie^{i\theta}\wt\sigma$, $\theta\in(-\pi/2,\pi/2)$, $\wt\sigma>0$, as before. Let $f=(\hd(\sigma)+\hdel(\sigma))u$; in particular $f\in\CI_{(\sigma)}$. Then $\hd(\sigma)\hdel(\sigma)u=\hd(\sigma)f$, so
  \begin{equation}
  \label{EqAbsenceQuant1}
    \la\hdel(\sigma)u,\hdel(\sigma)u\ra_{H\oplus e^{-2i\theta}H}=-\la\hd(\sigma)\hdel(\sigma)u,u\ra_{H\oplus e^{-2i\theta}H}=\la f,\hdel(\sigma)u\ra_{H\oplus e^{-2i\theta}H},
  \end{equation}
  and similarly
  \begin{equation}
  \label{EqAbsenceQuant2}
    \la\hd(\sigma)u,\hd(\sigma)u\ra_{H\oplus e^{2i\theta}H} = \la f,\hd(\sigma)u\ra_{H\oplus e^{2i\theta}H}.
  \end{equation}
  Multiply \eqref{EqAbsenceQuant1} by $e^{i\theta}$, \eqref{EqAbsenceQuant2} by $e^{-i\theta}$ and take the sum of both equations to get
  \begin{align*}
    e^{i\theta}(\|(\hdel(\sigma)u)_T\|^2+\|(\hd(\sigma)&u)_N\|^2) + e^{-i\theta}(\|(\hdel(\sigma)u)_N\|^2+\|(\hd(\sigma)u)_T\|^2) \\
      &=e^{i\theta}\la f,\hdel(\sigma)u\ra_{H\oplus e^{-2i\theta}H} + e^{-i\theta}\la f,\hd(\sigma)u\ra_{H\oplus e^{2i\theta}H}.
  \end{align*}
  Here, the norms without subscript are $L^2(\alpha|dh|;H\oplus H)$-norms as usual. Taking the real part and applying Cauchy--Schwarz to the right hand side produces the estimate
  \begin{equation}
  \label{EqAbsenceQuant3}
    \|\hd(\sigma)u\|+\|\hdel(\sigma)u\| \leq \frac{4}{\cos\theta}\|f\|=\frac{4|\sigma|}{\Im\sigma}\|f\|.
  \end{equation}
  We estimate $u$ in terms of the left hand side of \eqref{EqAbsenceQuant3} by following the arguments leading to \eqref{EqAbsenceImPosQualitative}: put $v=\hd(\sigma)u$ and $w=\hdel(\sigma)u$. Then $i\sigma u_T+d_X\alpha u_N=-\alpha v_N$; we pair this with $u_T$ in $L^2(\alpha|dh|;H)$ and obtain
  \[
    i\sigma\|u_T\|^2 + \la u_N,\delta_X\alpha u_T\ra = -\la\alpha v_N,u_T\ra.
  \]
  Using $-\delta_X\alpha u_T+i\sigma u_N=\alpha w_T$, this implies
  \[
    i\sigma\|u_T\|^2 - i\bar\sigma\|u_N\|^2 = -\la\alpha v_N,u_T\ra + \la u_N,\alpha w_T\ra,
  \]
  thus
  \[
    \wt\sigma(e^{i\theta}\|u_T\|^2+e^{-i\theta}\|u_N\|^2) = \la\alpha v_N,u_T\ra - \la u_N,\alpha w_T\ra.
  \]
  Taking the real part and applying Cauchy--Schwarz, we get
  \[
    (\cos\theta)\|u\| \leq |\sigma|^{-1}(\|\alpha v\|+\|\alpha w\|) \lesssim |\sigma|^{-1}(\|v\|+\|w\|).
  \]
  In combination with \eqref{EqAbsenceQuant3}, this yields \eqref{EqAbsenceImPosQuant}.
\end{proof}

\subsection{Boundary pairing and absence of non-zero real resonances}
\label{SubsecAbsenceNonzeroReal}

We proceed to exclude non-zero real resonances for $d+\delta$ by means of a boundary pairing argument similar to \cite[\S2.3]{MelroseGeometricScattering}.

\begin{prop}
\label{PropAbsenceNonZero}
  Suppose $\sigma\in\R$, $\sigma\neq 0$. If $u\in\CI_{(\sigma)}$ solves $(\hd(\sigma)+\hdel(\sigma))u=0$, then $u=0$.
\end{prop}
\begin{proof}
  We proceed in the usual three steps: \begin{enumerate*} \item vanishing of the leading coefficient at the horizon, \item rapid decay at the horizon, \item unique continuation. \end{enumerate*}

  \textit{Step (1).} Writing $u=u_T+\alpha\,dt\wedge u_N$ as usual, we can expand $(\hd(\sigma)+\hdel(\sigma))u=0$ as
  \begin{gather}
    \label{EqNonZeroEqn} (\alpha d_X-\delta_X\alpha)u_T + i\sigma u_N = 0 \\
	-i\sigma u_T + (-d_X\alpha+\alpha\delta_X)u_N = 0. \nonumber
  \end{gather}
  Applying $(-d_X\alpha+\alpha\delta_X)$ to the first equation and using the second equation to simplify the resulting expression produces a second order equation for $u_T$,
  \begin{equation}
  \label{EqNonZeroEqn2}
    (d_X\alpha\delta_X\alpha + \alpha\delta_X\alpha d_X - d_X\alpha^2 d_X - \sigma^2)u_T=0.
  \end{equation}
  Writing $u_T=u_{TT}+d\alpha\wedge u_{TN}$ as in \eqref{EqBundleDecNearBdy}, we see from the definition of the space $\CI_{(\sigma)}$ that
  \[
    u_T\in\CI_{(\sigma),T}:=\alpha^{-i\beta\sigma}\CI(\Xeven;\Lambda Y)\oplus\alpha^{-i\beta\sigma-1}\CI(\Xeven;\Lambda Y)
  \]
  near $Y$. Notice that for $\sigma\in\R$, the space $\CI_{(\sigma),T}$ barely fails to be contained in $L^2(\alpha|dh|)$.
  
  We will deduce from \eqref{EqNonZeroEqn2} that $u_T=0$; equation \eqref{EqNonZeroEqn} then gives $u_N=0$, as $\sigma\neq 0$. Now, the $L^2(\alpha|dh|;H)$-adjoint of $d_X\alpha$ is $\delta_X\alpha$, hence even ignoring the term $d_X\alpha^2 d_X$, the operator in \eqref{EqNonZeroEqn2} is not symmetric. However, we can obtain a simpler equation from \eqref{EqNonZeroEqn2} by applying $d_X$ to it; write $v_T=d_X u_T\in\CI_{(\sigma),T}$, and near $Y$,
  \[
    v_T=\begin{pmatrix}\alpha^{-i\beta\sigma}\wt v_{TT} \\ \alpha^{-i\beta\sigma-1}\wt v_{TN}\end{pmatrix}, \quad \wt v_{TT},\wt v_{TN}\in\CI(\Xeven;\Lambda Y).
  \]
  Then $v_T$ satisfies the equation
  \[
    (d_X\alpha\delta_X\alpha-\sigma^2)v_T=0,
  \]
  and $d_X\alpha\delta_X\alpha$ \emph{is} symmetric with respect to the $L^2(\alpha|dh|;H)$-inner product. We now compute the boundary pairing formula (using the same inner product); to this end, pick a cutoff function $\chi\in\CI(\oX)$ such that in a collar neighborhood $[0,\delta)_\alpha\times Y_y$ of $Y$ in $\oX$, $\chi=\chi(\alpha)$ is identically $0$ near $\alpha=0$ and identically $1$ in $\alpha\geq\delta/2$, and extend $\chi$ by $1$ to all of $\oX$. Define $\chi_\eps(\alpha)=\chi(\alpha/\eps)$ and $\chi'_\eps(\alpha)=\chi'(\alpha/\eps)$. Then
  \begin{equation}
  \label{EqBdyPairing1}
  \begin{split}
    0 &= \lim_{\eps\to 0}(\la(d_X\alpha\delta_X\alpha-\sigma^2)v_T,\chi_\eps v_T\ra - \la v_T,\chi_\eps(d_X\alpha\delta_X\alpha-\sigma^2)v_T\ra) \\
	  &= \lim_{\eps\to 0}\la v_T, [d_X\alpha\delta_X\alpha,\chi_\eps]v_T\ra.
  \end{split}
  \end{equation}
  The coefficients of the commutator are supported near $Y$, hence we use \eqref{EqDifferentialNearY} and \eqref{EqPaAAdj} to compute its form as
  \begin{align*}
    [d_X\alpha\delta_X\alpha,\chi_\eps]
	  &=\left[\begin{pmatrix}d_Y\alpha\delta_Y\alpha & d_Y\alpha\pa_\alpha^*\alpha \\ \pa_\alpha\alpha\delta_Y\alpha & \pa_\alpha\alpha\pa_\alpha^*\alpha + d_Y\alpha\wt\beta^2\delta_Y\wt\beta^{-2}\alpha \end{pmatrix}, \chi_\eps\right] \\
	  &= \begin{pmatrix}
	  		0 & d_Y\alpha[\pa_\alpha^*,\chi_\eps]\alpha \\
			[\pa_\alpha,\chi_\eps]\alpha\delta_Y\alpha & [\pa_\alpha\alpha\pa_\alpha^*\alpha,\chi_\eps]
		\end{pmatrix} \\
	  &= \eps^{-1}
	  	\begin{pmatrix}
	 		0 & -\beta^{-2}(\alpha^2+\cO(\alpha^4))\chi'_\eps d_Y + \alpha^4\chi_\eps'[d_Y,p_1] \\
			\chi'_\eps\alpha\delta_Y\alpha & \chi'_\eps\alpha\pa_\alpha^*\alpha - \pa_\alpha(\alpha^2+\cO(\alpha^4))\beta^{-2}\chi'_\eps
		\end{pmatrix}.
  \end{align*}
  In \eqref{EqBdyPairing1}, the off-diagonal terms of this give terms of the form
  \begin{equation}
  \label{EqBdyPairingVanishingTerm}
    \int_Y\int \alpha^{\mp i\beta\sigma}\alpha^{\pm i\beta\sigma-1}\eps^{-1}\alpha^2\chi'_\eps \wt v \,d\alpha\,|dk|
  \end{equation}
  with $\wt v\in\CI(\Xeven)$, and are easily seen to vanish in the limit $\eps\to 0$. The non-zero diagonal term gives\footnote{Recall that the volume density is given by $\alpha|dh|=\alpha\beta\,d\alpha|dk|$, and the fiber inner product in the $(TN)$-component is $\wt\beta^{-2}K$.} a term which comes from the $\cO(\alpha^4)$ summand and vanishes in the limit $\eps\to 0$, plus
  \begin{align*}
    \eps^{-1}&\la\alpha^{-i\beta\sigma-1}\wt v_{TN},(\chi'_\eps\alpha\pa_\alpha^*\alpha-\pa_\alpha\alpha^2\beta^{-2}\chi'_\eps)\alpha^{-i\beta\sigma-1}\wt v_{TN}\ra_{L^2(X;\alpha\beta\,d\alpha|dk|;\Lambda Y;\wt\beta^{-2}K)} \\
	  &= 2\int_Y\int\la\wt v_{TN},i\beta^{-2}\sigma\wt v_{TN}\ra_K \eps^{-1}\chi'_\eps\,d\alpha|dk| + o(1) \\
	  &\qquad\xrightarrow{\eps\to 0} -2i\beta^{-2}\sigma\|\wt v_{TN}|_Y\|^2_{L^2(Y,|dk|;K)};
  \end{align*}
  here, both summands in the pairing yield the same result, as is most easily seen by integrating by parts in $\alpha$, hence the factor of $2$, and the $o(1)$-term comes from differentiating $\wt v_{TN}$, which produces a term of the form \eqref{EqBdyPairingVanishingTerm}. We thus arrive at
  \[
    0=\la(d_X\alpha\delta_X\alpha-\sigma^2)v_T,v_T\ra-\la v_T,(d_X\alpha\delta_X\alpha-\sigma^2)v_T\ra = -2i\beta^{-2}\sigma\|\wt v_{TN}|_Y\|^2,
  \]
  whence $\wt v_{TN}|_Y=0$ in view of $\sigma\neq 0$, so we in fact have
  \begin{equation}
  \label{EqImprovedOrder}
    v_T = \begin{pmatrix}\alpha^{-i\beta\sigma}\wt v_{TT} \\ \alpha^{-i\beta\sigma+1}\wt v'_{TN}\end{pmatrix}, \quad \wt v'_{TN}\in\CI(\Xeven;\Lambda Y).
  \end{equation}

\textit{Step (2).} For the next step, recall that on a manifold with boundary $\oX$, 0-vector fields, introduced by Mazzeo
and Melrose \cite{MazzeoMelroseHyp} to analyze the resolvent of the
Laplacian on asymptotically hyperbolic spaces, are smooth vector fields that
vanish at $\pa X$, i.e.\ are of the form $\alpha V$, where $V$ is a
smooth vector field on $\oX$, and $\alpha$, as in our case, is a
boundary defining function, i.e.\ in local coordinates a linear
combination, with smooth coefficients, of
$\alpha\pa_\alpha$ and $\alpha\pa_{y_j}$. Further, 0-differential operators,
$A\in\Diff_0(\oX)$, are the differential operators generated by these
(taking finite sums of finite products, with $\CI(\oX)$
coefficients). As a contrast, b-vector fields are merely tangent to
$\pa X$, so in local coordinates they are linear combinations, with
smooth coefficients, of $\alpha\pa_\alpha$ and $\pa_{y_j}$,
and they generate b-differential operators $\Diffb(\oX)$. Often, as in our case, one is considering solutions of
0-differential equations with additional properties, such as having an
expansion in powers of $\alpha$ (and perhaps $\log\alpha$) with smooth
coefficients, i.e.\ polyhomogeneous functions. In these cases
$\alpha\Diffb(\oX)\subset\Diff_0(\oX)$ acts `trivially' on an
expansion in that it maps each term to one with an additional order of
vanishing, so in particular, one can analyze the asymptotic expansion of
solutions of
0-differential equations in this restrictive class by ignoring the
$\alpha\Diffb(\oX)$ terms. Notice that
$\alpha\pa_{y_j}\in\alpha\Diffb(\oX)$ in particular, so the tangential
0-derivatives can be dropped for this purpose. The indicial equation
is then obtained by freezing the coefficients of $A\in\Diff_0(\oX)$ at
$\pa X$, i.e.\ writing it as $\sum_{k,\beta} a_{k,\beta}(\alpha,y)
(\alpha\pa_\alpha)^k(\alpha \pa_y)^\beta$,
 where $a_{k,\beta}$ are bundle endomorphism valued, and restricting
 $\alpha$ to $0$, and dropping all terms with a positive power of
 $\alpha\pa_y$, to obtain $\sum_k
 a_{k,0}(0,y)(\alpha\pa_\alpha)^k$. This can be thought of as a regular-singular ODE
 in $\alpha$ for each $y$; its indicial roots are called the indicial roots
 of the original 0-operator, and they determine the asymptotics of
 solutions of the homogeneous PDE with this a priori form.

  Now $d_X\alpha\delta_X\alpha-\sigma^2\in\Diff_0^2(\oX)$ is a 0-differential operator which equals
  \[
    d_X\alpha\delta_X\alpha-\sigma^2 = \begin{pmatrix} -\sigma^2 & 0 \\ 0 & -\beta^{-2}\pa_\alpha\alpha\pa_\alpha\alpha - \sigma^2 \end{pmatrix}
  \]
  modulo $\alpha\Diffb^2(\oX)$; its indicial roots (i.e.\ the values of $\lambda$ for which $\alpha^{-\lambda}(d_X\alpha\delta_X\alpha-\sigma^2)\alpha^\lambda$, which is a matrix depending polynomially on $\lambda$, is not invertible) are $\pm i\beta\sigma-1$. In particular, $-i\beta\sigma+j$, $j\in\N_0$, is not an indicial root. Thus, a standard inductive argument starting with \eqref{EqImprovedOrder} shows that $v_T\in\CIdot(\oX;\Lambda X)$.

\textit{Step (3).}  Next, we note that $v_T$ lies in the kernel of the operator
  \begin{equation}
  \label{Eq0}
    d_X\alpha\delta_X\alpha+\alpha^2\delta_X d_X-\sigma^2\in\Diff_0^2(\oX;{}^0\Lambda\oX),
  \end{equation}
  which has the same principal part as $\alpha^2\Delta_X$ (computed with respect to the metric $h$), hence is principally a $0$-Laplacian; thus, we can apply Mazzeo's result \cite[Theorem~(13)]{MazzeoUniqueContinuation} on unique continuation at infinity for elliptic second order 0-differential operators such as~\eqref{Eq0} to conclude that the rapidly vanishing $v_T$ must in fact vanish identically.

  We thus have proved $d_X u_T=0$. Since $u_T$ satisfies \eqref{EqNonZeroEqn2}, we deduce that $u_T$ itself satisfies
  \[
    (d_X\alpha\delta_X\alpha-\sigma^2)u_T=0,
  \]
  thus repeating the above argument shows that this implies $u_T=0$, hence $u=0$, and the proof is complete.
\end{proof}

\subsection{Analysis of the zero resonance}
\label{SubsecZeroResonance}

We have shown now that the only potential resonance for $d+\delta$ in $\Im\sigma\geq 0$ is $\sigma=0$, and we proceed to study the zero resonance in detail, in particular giving a cohomological interpretation of it in Section~\ref{SubsecZeroResCohomology}.

We begin by establishing the order of the pole of $(\td(\sigma)+\tdel(\sigma))^{-1}$:

\begin{lemma}
\label{LemmaPoleOrder}
  $(\td(\sigma)+\tdel(\sigma))^{-1}$ has a pole of order $1$ at $\sigma=0$.
\end{lemma}
\begin{proof}
  Since $\td(0)+\tdel(0)$ annihilates constant functions (which are indeed elements of $\CI_{(0)}$), $(\td(\sigma)+\tdel(\sigma))^{-1}$ does have a pole at $0$. Denote the order of the pole by $N$. Then there is a holomorphic family $\wt u(\sigma)\in\CI(\wt X)$ with $\wt u(0)\neq 0$ such that $(\td(\sigma)+\tdel(\sigma))\wt u(\sigma)=\sigma^N\wt v$, where $\wt v\in\CI(\wt X)$. Define $u(\sigma)=e^{iF\sigma}\wt u(\sigma)|_X\in\CI_{(\sigma)}$ and $v(\sigma)=e^{iF\sigma}\wt v|_X\in\CI_{(\sigma)}$, then $(\hd(0)+\hdel(0))u(\sigma)=\sigma^N v(\sigma)$. Moreover, since $(\td(0)+\tdel(0))\wt u(0)=0$ and $\wt u(0)$ is non-zero, Lemma~\ref{LemmaPoles} shows that $u(0)\neq 0$.

  Let us assume now that $N\geq 2$. For $\sigma\in i(0,\infty)$ close to $0$, the quantitative estimate in Proposition~\ref{PropAbsenceImPos} now gives
  \begin{equation}
  \label{EqPoleOrder}
    \|u(\sigma)\| \lesssim |\sigma|^{-1+N}\|v(\sigma)\|\leq|\sigma|\|v(\sigma)\|,
  \end{equation}
  where we use the norm of $L^2(\alpha|dh|;H\oplus H)$.\footnote{Observe that in the notation of Section~\ref{SubsecAbsenceUpperHalfPlane}, we have $\widehat{\delta_0}(0)=-\hdel(0)$, hence using the Riemannian fiber inner product $H\oplus H$ is natural when studying the zero resonance.} Notice that this does not immediately give $u(0)=0$ since $v(0)\notin L^2(\alpha|dh|;H\oplus H)$. However, we can quantify the degeneration of the $L^2$-norm of $v(\sigma)$ as $\sigma\to 0$. To see this, we first observe that the $L^2$-norm of $v(\sigma)$ restricted to the complement of any fixed neighborhood of $Y$ does stay bounded, so it remains to analyze the $L^2$-norms of the four components of $v(\sigma)$ near $Y$ in the notation of \eqref{EqBundleDecWarped}; denote these components by $\alpha^{-i\beta\sigma}\wt v_{TT}(\sigma),\alpha^{-i\beta\sigma-1}\wt v_{TN}(\sigma),\alpha^{-i\beta\sigma-1}\wt v_{NT}(\sigma)$ and $\alpha^{-i\beta\sigma}\wt v_{NN}(\sigma)$, so that the $\wt v_{\bullet\bullet}(\sigma)\in\CI(\Xeven;\Lambda Y)$ uniformly. Since the fiber metric in this basis has a block diagonal form and any $\CI(\Xeven)$-multiple of $\alpha^{-i\beta\sigma}$ is uniformly square-integrable with respect to the volume density $\alpha|dh|$, the degeneration of the $L^2$-norm of $v$ is caused by the $(TN)$ and $(NT)$ components. For these, we compute, with $\wt w(\sigma)\in\CI(\Xeven;\Lambda Y)$ denoting any continuous family supported near $Y$,
  \begin{align*}
    \int_Y \int &\alpha^{2(-i\beta\sigma-1)}\|\wt w\|^2_K\,\alpha\,d\alpha|dk| \\
     &= \|\wt w(0)\|^2_{L^2(Y,|dk|;K)} \int\alpha^{-2i\beta\sigma-1}\chi(\alpha)\,d\alpha + \cO(1),
  \end{align*}
  where $\chi\in\CI(\oX)$ is a cutoff, equal to $1$ near $\alpha=0$. We can rewrite the integral using an integration by parts, which yields
  \[
    \int\alpha^{-2i\beta\sigma-1}\chi(\alpha)\,d\alpha = \frac{1}{2i\beta\sigma}\int\alpha^{-2i\beta\sigma}\chi'(\alpha)\,d\alpha = \cO(|\sigma|^{-1}).
  \]
  Therefore, we obtain the bound $\|v(\sigma)\|=\cO(|\sigma|^{-1/2})$. Plugging this into \eqref{EqPoleOrder}, we conclude using Fatou's Lemma that $u(0)=0$, which contradicts our assumption that $u(0)\neq 0$. Hence, the order of the pole is $N\leq 1$, but since it is at least $1$, it must be equal to $1$.
\end{proof}

Next, we identify the resonant states. \emph{For brevity, we will write $\hd=\hd(0)$, $\hdel=\hdel(0)$ and $\wh\Box=\wh{\Box_g}(0)$.}
\begin{prop}
\label{PropResStates}
  $\ker_{\CI_{(0)}}(\hd+\hdel)$ is equal to the space
  \begin{equation}
  \label{EqZeroResonantStates}
    \cH = \{ u\in\CI_{(0)} \colon \hd u=0, \hdel u=0 \}.
  \end{equation}
\end{prop}
\begin{proof}
  Given $u\in\CI_{(0)}$ with $(\hd+\hdel)u=0$, we conclude that $\wh\Box u=0$. We observe now that $\wh\Box$ is symmetric on $L^2(\alpha|dh|;H\oplus H)$: indeed, $\hd(\sigma)$ and $\hdel(\sigma)$ are block diagonal for $\sigma=0$, see~\eqref{EqDiffHat}, hence are adjoints of one another with respect to $\pm H\oplus\pm H$ for any choice of signs, with opposite signs giving the natural inner product~\eqref{EqSpacetimeFiberMetric}, and both signs positive giving the Riemannian fiber metric $H\oplus H$. Thus, we can obtain information about $u$ by a boundary pairing type argument: concretely, for a cutoff $\chi\in\CI(\oX)$ as in the proof of Proposition~\ref{PropAbsenceNonZero}, identically $0$ near $Y$, identically $1$ outside a neighborhood of $Y$ and a function of $\alpha$ in a collar neighborhood of $Y$, and with $\chi_\eps(\alpha)=\chi(\alpha/\eps)$, $\chi'_\eps(\alpha)=\chi'(\alpha/\eps)$, we have\footnote{The minus sign disappears after the second equality sign due to the change of signs in the used inner product, cf.\ the discussion around~\eqref{EqSpacetimeDel}.}
  \begin{align}
    0&=-\lim_{\eps\to 0}\la\chi_\eps(\hd\,\hdel+\hdel\,\hd)u,u\ra = \lim_{\eps\to 0}(\la\hdel u,\hdel\chi_\eps u\ra+\la\hd u,\hd\chi_\eps u\ra) \nonumber\\
   \label{EqZeroResComm} 
	 &=\lim_{\eps\to 0}(\|\chi_\eps^{1/2}\hdel u\|^2+\|\chi_\eps^{1/2}\hd u\|^2) + \lim_{\eps\to 0}(\la\hdel u,[\hdel,\chi_\eps]u\ra + \la\hd u,[\hd,\chi_\eps]u\ra).
  \end{align}
  Since the commutators are supported near $Y$, we can compute them in the basis \eqref{EqBundleDecWarped}. Let us write $u=\sC\wt u$ as in \eqref{EqResState} with $\sigma=0$, then in view of \eqref{EqCodifferentialC}, we have
  \begin{equation}
  \label{EqCodiffCommutator}
    [\hdel\sC,\chi_\eps]
	  = \eps^{-1}\chi'_\eps
	    \begin{pmatrix}
	      0 & \beta^{-2}\alpha + \cO(\alpha^3) & \beta^{-1}\alpha^{-1}+\cO(\alpha) & 0 \\
		  0 & 0 & 0 & 0 \\
		  0 & 0 & 0 & -\beta^{-2} + \cO(\alpha^2) \\
		  0 & 0 & 0 & 0
	    \end{pmatrix},
  \end{equation}
  and since therefore only the $(TT)$ and $(NT)$ components of $[\hdel\sC,\chi_\eps]\wt u$ are non-zero, we merely compute
  \begin{align*}
    (\hdel\sC\wt u)_{TT} &= -\delta_Y\wt u_{TT} - \alpha^{-1}\pa_\alpha^*\alpha^2\wt u_{TN} - \beta\alpha^{-1}\pa_\alpha^*\wt u_{NT} \\
	   &\quad\in -\delta_Y\wt u_{TT} + 2\beta^{-2}\wt u_{TN} - \beta\alpha^{-1}\pa_\alpha^*\wt u_{NT} + \alpha^2\CI(\Xeven;\Lambda Y), \\
	(\hdel\sC\wt u)_{NT} &= \alpha^{-1}\delta_Y\wt u_{NT} + \pa_\alpha^*\wt u_{NN} \in \alpha^{-1}\delta_Y\wt u_{NT} + \alpha\,\CI(\Xeven;\Lambda Y).
  \end{align*}
  Notice here that $\alpha^{-1}\pa_\alpha=2\pa_\mu$ indeed preserves elements of $\CI(\Xeven;\Lambda Y)$. Now in \eqref{EqZeroResComm}, the pairing corresponding to the $(1,2)$-component of \eqref{EqCodiffCommutator} is of the form \eqref{EqBdyPairingVanishingTerm} (recall that the volume density is $\alpha|dh|=\alpha\wt\beta\,d\alpha|dk|$) and hence vanishes in the limit $\eps\to 0$, and we conclude that
  \begin{equation}
  \label{EqCodiffBdyPairings}
  \begin{split}
    \lim_{\eps\to 0}\la\hdel u,[\hdel,\chi_\eps]u\ra &= -\la\delta_Y\wt u_{TT}|_Y,\wt u_{NT}|_Y\ra + 2\beta^{-2}\la\wt u_{TN}|_Y,\wt u_{NT}|_Y\ra \\
	  &\quad - \beta\la(\alpha^{-1}\pa_\alpha^*\wt u_{NT})|_Y,\wt u_{NT}|_Y\ra - \beta^{-1}\la\delta_Y\wt u_{NT}|_Y,\wt u_{NN}|_Y\ra,
  \end{split}
  \end{equation}
  where we use the $L^2(Y,|dk|;K)$ inner product on the right hand side;\footnote{Recall here that $H\oplus H=K\oplus\wt\beta^{-2}K\oplus K\oplus\wt\beta^{-2}K$, so the $(TT)$ and $(NT)$ components which we are concerned with here do not come with an extra factor of $\wt\beta^{-2}$.} we absorbed the factor $\wt\beta|_Y=\beta$ from the volume density $\alpha\wt\beta\,d\alpha|dk|$ into the functions in the pairings.

  In a similar vein, we can use \eqref{EqDifferentialC} to compute
  \begin{equation}
  \label{EqDiffCommutator}
    [\hd\sC,\chi_\eps]
	  = \eps^{-1}\chi'_\eps
	    \begin{pmatrix}
		  0 & 0 & 0 & 0 \\
		  1 & 0 & 0 & 0 \\
		  0 & 0 & 0 & 0 \\
		  0 & 0 & -\alpha^{-1} & 0
		\end{pmatrix}
  \end{equation}
  and
  \begin{align*}
    (\hd\sC\wt u)_{TN} &= \pa_\alpha\wt u_{TT} - \alpha d_Y\wt u_{TN} - \alpha^{-1}d_Y\beta\wt u_{NT} \\
	  &\quad\in -\beta\alpha^{-1}d_Y\wt u_{NT} + \CI(\Xeven;\Lambda Y), \\
	(\hd\sC\wt u)_{NN} &= -\alpha^{-1}\pa_\alpha\wt u_{NT} + d_Y\wt u_{NN}.
  \end{align*}
  Correspondingly,
  \begin{equation}
  \label{EqDiffBdyPairings}
  \begin{split}
    \lim_{\eps\to 0}\la\hd u,[\hd,\chi_\eps]u\ra &= -\la d_Y\wt u_{NT}|_Y,\wt u_{TT}|_Y\ra + \beta^{-1}\la(\alpha^{-1}\pa_\alpha\wt u_{NT})|_Y,\wt u_{NT}|_Y\ra \\
	  &\quad - \beta^{-1}\la d_Y\wt u_{NN}|_Y,\wt u_{NT}|_Y\ra,
  \end{split}
  \end{equation}
  where we again use the $L^2(Y,|dk|;K)$ inner product on the right hand side; notice with regard to the powers of $\beta$ that on the $(TN)$ and $(NN)$ components, the fiber inner product is $\beta^{-2}K$.

  As a consequence of these computations, we conclude that the pairings in \eqref{EqZeroResComm} stay bounded as $\eps\to 0$, hence $\hd u,\hdel u\in L^2(\alpha|dh|;H\oplus H)$ by Fatou's Lemma. Looking at the most singular terms of $\hd\sC\wt u$ and $\hdel\sC\wt u$ (again using \eqref{EqDifferentialC} and \eqref{EqCodifferentialC}), this necessitates
  \begin{equation}
  \label{EqUNTHarmonic}
    d_Y\wt u_{NT}|_Y=0,\quad \delta_Y\wt u_{NT}|_Y=0.
  \end{equation}
  Therefore, taking \eqref{EqCodiffBdyPairings} and \eqref{EqDiffBdyPairings} into account, \eqref{EqZeroResComm} simplifies to
  \begin{equation}
  \label{EqZeroResCommFull}
  \begin{split}
    0 = \|\hdel u\|^2+&\|\hd u\|^2 + \beta^{-1}\la(\alpha^{-1}\pa_\alpha\wt u_{NT})|_Y,\wt u_{NT}|_Y\ra \\
	  &- \beta\la(\alpha^{-1}\pa_\alpha^*\wt u_{NT})|_Y,\wt u_{NT}|_Y\ra + 2\beta^{-2}\la\wt u_{TN}|_Y,\wt u_{NT}|_Y\ra.
  \end{split}
  \end{equation}
  Moreover, the fourth component of the equation $(\hd+\hdel)\sC\wt u=0$ yields
  \[
    -(\alpha^{-1}\pa_\alpha\wt u_{NT})|_Y+d_Y\wt u_{NN}|_Y-\delta_Y\wt u_{NN}|_Y = 0,
  \]
  which we can pair with $\wt u_{NT}|_Y$ relative to $L^2(Y,|dk|;K)$, and then an integration by parts together with \eqref{EqUNTHarmonic} shows that the first boundary pairing in \eqref{EqZeroResCommFull} vanishes. Likewise, the first component of $(\hd+\hdel)\sC\wt u=0$ gives
  \[
    d_Y\wt u_{TT}|_Y-\delta_Y\wt u_{TT}|_Y+2\beta^{-2}\wt u_{TN}|_Y - \beta(\alpha^{-1}\pa_\alpha^*\wt u_{NT})|_Y = 0,
  \]
  which we can again pair with $\wt u_{NT}|_Y$, and in view of \eqref{EqUNTHarmonic}, we conclude that the second line of \eqref{EqZeroResCommFull} vanishes as well. Thus, finally, \eqref{EqZeroResCommFull} implies that $\hd u=0$ and $\hdel u=0$.

  Conversely, every $u\in\CI_{(0)}$ satisfying $\hd u=0$ and $\hdel u=0$ trivially lies in the kernel of $\hd+\hdel$.
\end{proof}

The above proof in particular shows:
\begin{cor}
\label{CorBoxL2Kernel}
  Let $u=\sC\wt u\in\CI_{(0)}$ be such that $\hd\,\hdel u=0$ (resp.\ $\hdel\,\hd u=0$), and assume that $\wt u_{NT}|_Y=0$.\footnote{The latter is equivalent to assuming $u\in L^2(\alpha|dh|)$.} Then $\hdel u=0$ (resp.\ $\hd u=0$). In particular, $\ker_{\CI_{(0)}\cap L^2}\wh\Box = \cH\cap L^2$.
\end{cor}
\begin{proof}
  Suppose $\hd\,\hdel u=0$. With a cutoff function $\chi_\eps$ as above, we obtain
  \[
    0=-\lim_{\eps\to 0}\la\chi_\eps\hd\,\hdel u,u\ra = \lim_{\eps\to 0}\|\chi_\eps^{1/2}\hdel u\|^2 + \lim_{\eps\to 0}\la\hdel u,[\hdel,\chi_\eps]u\ra.
  \]
  In view of \eqref{EqCodiffBdyPairings} and $\wt u_{NT}|_Y=0$, the second term on the right hand side vanishes, and we deduce $\hdel u=0$. The proof that $\hdel\,\hd u=0$ implies $\hd u=0$ is similar and uses \eqref{EqDiffBdyPairings}.
\end{proof}

\begin{cor}
\label{CorKernelBox}
  We have $\ker\wh\Box=\ker\hd\,\hdel\cap\ker\hdel\,\hd$.
\end{cor}
\begin{proof}
  If $u\in\ker\wh\Box$, then $(\hd+\hdel)u\in\cH$, thus $\hdel(\hd+\hdel)u=\hdel\,\hd u=0$ and $\hd\,\hdel u=0$.
\end{proof}

We record another setting in which the boundary terms in the proof of Proposition~\ref{PropResStates} vanish:
\begin{lemma}
\label{LemmaNoBdyTerms}
  Suppose $v\in\CI_{(0)}$ is a solution of $\hdel\,\hd\,\hdel v=0$. Then $\hd\,\hdel v=0$. Likewise, if $v\in\CI_{(0)}$ is a solution of $\hd\,\hdel\,\hd v=0$, then $\hdel\,\hd v=0$.
\end{lemma}
\begin{proof}
  Write $w=\hdel v\in\CI_{(0)}$. Then $\hdel\,\hd w=0$ implies, by the proof of Proposition~\ref{PropResStates}, that $\hd w\in L^2(\alpha|dh|;H\oplus H)$. Writing $w=\sC\wt w$, this in particular implies $d_Y\wt w_{NT}|_Y=0$; but writing $v=\sC\wt v$, we have
  \[
    \wt w_{NT}=(\sC^{-1}\hdel\sC\wt v)_{NT} = \delta_Y\wt v_{NT} + \alpha\pa_\alpha^*\wt v_{NN},
  \]
  as follows from \eqref{EqCodifferentialCinvC}. Restricting to $Y$, we thus have $\wt w_{NT}|_Y=\delta_Y\wt v_{NT}|_Y$, and hence $0=d_Y\delta_Y\wt v_{NT}|_Y$. We pair this in $L^2(Y,|dk|;K)$ with $\wt v_{NT}$ and integrate by parts, obtaining $\delta_Y\wt v_{NT}|_Y=0$. But this implies that $\wt w_{NT}|_Y=0$. By Corollary~\ref{CorBoxL2Kernel}, this gives $\hd w=\hd\,\hdel v=0$.

  For the second part, we proceed analogously: letting $w=\hd v\in\CI_{(0)}$, we have $\hd\,\hdel w=0$, thus $\hdel w\in L^2$. This gives $\delta_Y\wt w_{NT}|_Y=0$; but by \eqref{EqDifferentialCinvC}, $\wt w_{NT}|_Y=-d_Y\wt v_{NT}|_Y$, therefore $\delta_Y\wt w_{NT}|_Y=0$ implies $d_Y\wt v_{NT}|_Y=0$, so $\wt w_{NT}|_Y=0$, which in turn gives $\hdel w=0$ by Corollary~\ref{CorBoxL2Kernel}, hence $\hdel\,\hd v=0$.
\end{proof}

\subsection{Cohomological interpretation of zero resonant states}
\label{SubsecZeroResCohomology}

\emph{In this section, we will always work with $\sigma=0$ and hence simply write $\hd=\hd(0)$, $\hdel=\hdel(0)$, $\td=\td(0)$, $\tdel=\tdel(0)$, $\wh\Box=\wh\Box(0)$ and $\wt\Box=\wt\Box(0)$.}

The space $\cH$ defined in Proposition~\ref{PropResStates} is graded by the form degree, i.e.
\begin{equation}
\label{EqDplusDelKernel}
  \cH=\bigoplus_{k=0}^n\cH^k,
\end{equation}
where $\cH^k$ is the space of all $u\in\cH$ of pure form degree $k$. In the decomposition \eqref{EqFormDecomposition}, this means that $u_T$ is a differential $k$-form on $X$, and $u_N$ is a differential $(k-1)$-form. Likewise, $\cK:=\ker\wh\Box$ is graded by form degree, and we write
\begin{equation}
\label{EqBoxKernel}
  \ker_{\CI_{(0)}}\wh\Box=\bigoplus_{k=0}^n\cK^k.
\end{equation}
We aim to relate the spaces $\cH^k$ and $\cK^k$ to certain cohomology groups associated with $\oX$. As in the Riemannian setting, the central tool is a Hodge type decomposition adapted to $\hd$ and $\hdel$: 

\begin{lemma}
\label{LemmaHodgeDecomp}
  The following Hodge type decomposition holds on $X$:
  \begin{equation}
  \label{EqHodge}
    \CI_{(0)} = \ker_{\CI_{(0)}}\wh\Box \oplus \ran_{\CI_{(0)}}\wh\Box.
  \end{equation}
\end{lemma}
\begin{proof}
  We first claim that such a decomposition holds on $\wt X$, i.e.\ we claim that
  \begin{equation}
  \label{EqHodgeTilde}
    \CI(\wt X) = \ker\wt\Box \oplus \ran\wt\Box.
  \end{equation}
  First of all, recall that $\wt\Box$ is Fredholm with index $0$ as an operator~\eqref{EqBoxFred} for all sufficiently large $s$, and a complement to its range is given by an $s$-independent finite-dimensional subspace of $\CI(\wt X)$, namely, the kernel of its adjoint. Thus, the range of $\wt\Box\colon\CI(\wt X)\to\CI(\wt X)$ is closed, and its codimension equals the dimension of the kernel of $\wt\Box$. Hence, in order to show \eqref{EqHodgeTilde}, we merely need to check that the intersection of $\ker\wt\Box$ and $\ran\wt\Box$ is trivial.
  
  Thus, let $\wt u\in\ker\wt\Box\cap\ran\wt\Box$, and write $\wt u=\wt\Box\wt v$. Let $v=\wt v|_X$. Then $\wt u\in\ker\wt\Box$ means, restricting to $X$ and using Corollary~\ref{CorKernelBox}, that $\hd\,\hdel\,\hd\,\hdel v=0$ and $\hdel\,\hd\,\hdel\,\hd v=0$. Repeated application of Lemma~\ref{LemmaNoBdyTerms} thus implies $\hdel\,\hd v=0$ and $\hd\,\hdel v=0$, hence $\tdel\,\td\wt v$ and $\td\,\tdel\wt v$ are supported in $\wt X\setminus X$. (This argument shows the uniqueness of the decomposition \eqref{EqHodge}.) Therefore $\wt u$ is a solution of $\wt\Box\wt u=0$ which is supported in $\wt X\setminus X$. By unique continuation at infinity on the asymptotically de Sitter side $\wt X\setminus X$ of $\wt X$, this implies $\wt u\equiv 0$, as claimed.

  Now if $u\in\CI_{(0)}$ is given, extend it arbitrarily to $\wt u\in\CI(\wt X)$, apply \eqref{EqHodgeTilde} and restrict both summands back to $X$. This establishes \eqref{EqHodge}.
\end{proof}

\begin{rmk}
\label{RmkNoDDelHodge}
  The decomposition \eqref{EqHodge} does \emph{not} hold if we replace $\wh\Box$ in \eqref{EqHodge} by $\hd+\hdel$. Indeed, if it did hold, this would say that $\wh\Box u=0$ implies $(\hd+\hdel)u=0$, since $(\hd+\hdel)u$ lies both in $\ker(\hd+\hdel)$ and $\ran(\hd+\hdel)$ in this case. Since certainly $(\hd+\hdel)u=0$ conversely implies $\wh\Box u=0$, this would mean that $\ker\wh\Box=\ker(\hd+\hdel)$. Now by Lemmas~\ref{LemmaPolesBox} and \ref{LemmaPoles}, this in turn would give $\ker\wt\Box=\ker(\td+\tdel)$. Now since $\wt\Box$ and $\td+\tdel$ are Fredholm with index $0$, we could further deduce $\ker\wt\Box^*=\ker(\td+\tdel)^*$, where the adjoints act on the space $\CmIdot(\wt X)$ of supported distributions at the (artificial) Cauchy hypersurface $\pa\wt X$, see \cite[Appendix~B]{HormanderAnalysisPDE}. Since we have $\ker(\td+\tdel)^*\subset\ker\wt\Box^*$ unconditionally, we can show the absurdity of this last equality by exhibiting an element $u$ in $\ker\wt\Box^*$ which does not lie in $\ker(\td+\tdel)^*$. This however is easy: just let $u=1_X$ be the characteristic function of $X$. Then from \eqref{EqDifferentialCinvC} and \eqref{EqCodifferentialCinvC}, we see that $(\td+\tdel)u=\td u$ is a non-zero delta distribution supported at $Y$ which is annihilated by $\tdel$.

  This argument shows that we always have $\ker\wh\Box\supsetneq\ker(\hd+\hdel)$. It is possible though that $\cH^k=\cK^k$ for \emph{some} form degrees $k$ (but this must fail for some value of $k$). For instance, this holds for $k=0$ by Corollary~\ref{CorBoxL2Kernel}. We will give a more general statement below, see in particular Remark~\ref{RmkDDelBoxKernelsEqual}.
\end{rmk}

We now define a complex whose cohomology we will relate to the spaces $\cH^k$ and $\cK^k$: the space $\CI_{(0)}\cap L^2(\alpha|dh|)$ of smooth forms $u=\sC\wt u$ with $\wt u_{NT}|_Y=0$ has a grading corresponding to form degrees, thus
\[
  \cD:=\CI_{(0)}\cap L^2(\alpha|dh|) = \bigoplus_{k=0}^n \cD^k.
\]
Since in the above notation $u\in L^2(\alpha|dh|)$ (and thus $\wt u_{NT}|_Y=0$) is equivalent to $\wt u_{NT}\in\alpha^2\CI(\Xeven;\Lambda Y)$ near $Y$, one can easily check using \eqref{EqDifferentialCinvC} that $\hd$ acts on $\CI_{(0)}\cap L^2(\alpha|dh|)$. We can then define the complex
\[
  0 \to \cD^0 \xra{\hd} \cD^1 \to \ldots \xra{\hd} \cD^n \to 0.
\]
We denote its cohomology by
\begin{equation}
\label{EqL2CohoDef}
  \cH^k_{L^2,\dR} = \ker(\hd\colon\cD^k\to\cD^{k+1})/\ran(\hd\colon\cD^{k-1}\to\cD^k).
\end{equation}
There is a natural map from $\cH^k_{L^2,\dR}$ into $\cH^k$:

\begin{lemma}
\label{LemmaCohoMap}
  Every cohomology class $[u]\in\cH^k_{L^2,\dR}$ has a unique representative $u'\in\cH^k$, and the map $i\colon[u]\mapsto u'$ is injective.
\end{lemma}
\begin{proof}
  Let $[u]\in\cH^k_{L^2,\dR}$, hence $\hd u=0$ and, writing $u=\sC\wt u$, $\wt u_{NT}|_Y=0$. We first show the existence of a representative, i.e.\ an element $u-\hd v$ with $v\in\cD$, which is annihilated by $\hdel$. (Since it is clearly annihilated by $\hd$, this means $u-\hd v\in\cH^k$.) That is, we need to solve the equation $\hdel\,\hd v=\hdel u$ with $v\in\cD$. To achieve this, we use Lemma~\ref{LemmaHodgeDecomp} to write
  \[
    u = u_1 + (\hd\,\hdel+\hdel\,\hd)u_2,\quad u_1\in\ker\wh\Box.
  \]
  By our assumption on $u$ and Corollary~\ref{CorKernelBox}, $u$ and $u_1$ are annihilated by $\hdel\,\hd$, giving $\hdel\,\hd\,\hdel\,\hd u_2=0$. By Lemma~\ref{LemmaNoBdyTerms}, this implies $\hdel\,\hd u_2=0$, hence
  \begin{equation}
  \label{EqCohoMapHodge}
    u=u_1+\hd\,\hdel u_2.
  \end{equation}
  Applying $\hd\,\hdel$, we obtain
  \begin{equation}
  \label{EqDDelDDelU2}
    \hd\,\hdel\,\hd\,\hdel u_2=\hd\,\hdel u \in L^2.
  \end{equation}
  Now writing $u_2=\sC\wt u_2$, and noting that for any $w=\sC\wt w\in\CI_{(0)}$, $(\sC^{-1}\hd\sC\wt w)_{NT}|_Y=-d_Y\wt w_{NT}|_Y$ as well as $(\sC^{-1}\hdel\sC\wt w)_{NT}|_Y=\delta_Y\wt w_{NT}|_Y$ by \eqref{EqDifferentialCinvC} and \eqref{EqCodifferentialCinvC}, the $(NT)$ component of $\sC^{-1}$ times equation \eqref{EqDDelDDelU2} reads $d_Y\delta_Y d_Y\delta_Y\wt u_{2,NT}|_Y=0$, which yields $\delta_Y\wt u_{2,NT}|_Y=0$. As a consequence of this, $v:=\hdel u_2\in L^2$ and therefore $\hd\,\hdel u_2\in L^2$. Hence \eqref{EqCohoMapHodge} gives $u_1\in L^2$; by Corollary~\ref{CorBoxL2Kernel} then, $u_1\in\cH$, in particular $u_1$ is annihilated by $\hdel$. Therefore, applying $\hdel$ to \eqref{EqCohoMapHodge} yields $\hdel(u-\hd v)=0$, as desired.

  Next, we show that the representative is unique: thus, suppose $u-\hd v_1,u-\hd v_2\in\cH^k$ with $u,v_1,v_2\in\cD$, then with $v=v_1-v_2\in\cD$, we have $\hd v\in\cH^k$, thus $\hdel\,\hd v=0$, and by Corollary~\ref{CorBoxL2Kernel}, we obtain $\hd v=0$. Therefore, $u-\hd v_1=u-\hd v_2$, establishing uniqueness, which in particular shows that the map $i$ is well-defined.

  Finally, we show the injectivity of $i$: suppose $u\in\cD$ satisfies $\hd u=0$. There exists an element $v\in\cD$ such that $u-\hd v\in\cH^k$. Now if $i[u]=0$, this precisely means that $u-\hd v=0$; but then $[u]=[\hd v]=0$ in $\cH^k_{L^2,\dR}$.
\end{proof}

From the definition of the space $\cD$, it is clear that $u\in\cH^k$ lies in the image of $i$ if and only if $u\in L^2$, i.e.\ if and only if $r(u)=0$, where $r$ is the map
\begin{equation}
\label{EqRestrictionMap}
  r\colon\CI_{(0)}\to\CI(Y;\Lambda Y),\quad u=\sC\wt u\mapsto \wt u_{NT}|_Y.
\end{equation}
Thus, $r$ extracts the singular part of $u$ and thereby measures the failure of a given form $u\in\CI_{(0)}$ to lie in $\cD$. Observe that if $u=\sC\wt u\in\cH^k$, then $d_Y\wt u_{NT}|_Y=0$ and $\delta_Y\wt u_{NT}|_Y=0$, i.e.\ $r(u)$ is a harmonic form on $Y$. Since the space $\ker(\Delta_{Y,k-1})$ of harmonic forms on the closed manifold $Y$ is isomorphic to the cohomology group $H^{k-1}(Y)$ by standard Hodge theory, we thus obtain:
\begin{prop}
\label{PropSequenceDDel}
  The sequence
  \begin{equation}
  \label{EqSequenceDDel}
    0 \to \cH^k_{L^2,\dR} \xra{i} \cH^k \xra{r} H^{k-1}(Y)
  \end{equation}
  is exact. Here, $i$ is the map defined in Lemma~\ref{LemmaCohoMap}, and $r$ is the restriction map \eqref{EqRestrictionMap} (composed with the identification $\ker(\Delta_{Y,k-1})\cong H^{k-1}(Y)$). Moreover, the map $i\colon\cH^k_{L^2,\dR}\to\cH^k\cap\cD$ is an isomorphism with inverse $\cH^k\cap\cD\ni u\mapsto[u]\in\cH^k_{L^2,\dR}$.
\end{prop}
\begin{proof}
  We only need to check the last claim. If $u\in\cH^k\cap\cD$, then $[u]$ does define a cohomology class in $\cH^k_{L^2,\dR}$, and $i([u])$ is the unique representative of $[u]$ which lies in $\cH^k$. Since $u$ itself is such a representative, we must have $i([u])=u$. For the converse, we note that for any $[u]\in\cH^k_{L^2,\dR}$ we have $i([u])=u-\hd v$ for some $v\in\cD$, hence $[i([u])]=[u-\hd v]=[u]$.
\end{proof}

We can make a stronger statement: if we merely have $u\in\ker\wh\Box$, then the proof of Proposition~\ref{PropResStates} shows that $\hd u,\hdel u\in L^2$, hence $r(u)$ is harmonic. 

\begin{prop}
\label{PropSequenceBox}
  We have a short exact sequence
  \begin{equation}
  \label{EqSequenceBox}
    0 \to \cH^k_{L^2,\dR} \xra{i} \cK^k \xra{r} H^{k-1}(Y) \to 0,
  \end{equation}
  where the first map is $i$ defined in Lemma~\ref{LemmaCohoMap} (composed with the inclusion $\cH^k\hookrightarrow\cK^k$), and the second map is the restriction $r$, defined in \eqref{EqRestrictionMap} (composed with the identification $\ker(\Delta_{Y,k-1})\cong H^{k-1}(Y)$).
\end{prop}
\begin{proof}
  The second map is well-defined by the comment preceding the statement of the proposition. Since the range of $\cH^k_{L^2,\dR}$ in $\cK^k$ consists of $L^2$ forms, we have $r\circ i=0$. Moreover, if $u\in\ker r$, then $u$ is an $L^2$ element of $\ker\wh\Box$, thus $u\in\cH^k$ by Corollary~\ref{CorBoxL2Kernel}. By the remark following the proof of Lemma~\ref{LemmaCohoMap}, therefore $u\in\ran i$.
  
  It remains to show the surjectivity of $r$: thus, let $w\in\ker(\Delta_{Y,k-1})$, and let $u'=\sC\wt u'\in\CI_{(0)}$ be any extension of $w$, i.e.\ $\wt u'_{NT}|_Y=w$. Then $(\hd+\hdel)u'\in\cD$, since its $(NT)$ component vanishes, and thus $\wh\Box u'\in\cD$. Writing $u'=u_1+\wh\Box u_2$ with $u_1\in\ker\wh\Box$, we conclude that $\wh\Box u'=\wh\Box^2 u_2$; taking the $(NT)$ component of this equation gives $0=\Delta_Y^2\wt u_{2,NT}|_Y$ (where we write $u_2=\sC\wt u_2$ as usual), hence $d_Y\wt u_{2,NT}|_Y=0$ and $\delta_Y\wt u_{2,NT}|_Y=0$. But then $\wh\Box u_2\in L^2$. Therefore, $w=r(u')=r(u_1+\wh\Box u_2)=r(u_1)$. Since the degree $k$ part of $u_1$ lies in $\cK^k$ by the definition of $u_1$, we are done.
\end{proof}

\begin{rmk}
  Remark \ref{RmkNoDDelHodge}, which states that $\cH^k\subsetneq\cK^k$ for some values of $k$, implies in particular that the last map of \eqref{EqSequenceDDel} is not always onto.
\end{rmk}

\begin{rmk}
\label{RmkDDelBoxKernelsEqual}
  Since $\dim Y=n-2$, we have $H^{k-1}(Y)=0$ for $k=0$ and $k=n$. Hence, for these extreme values of $k$, Propositions \ref{PropSequenceDDel} and \ref{PropSequenceBox} show $\cH^k=\cK^k\cong\cH^k_{L^2,\dR}$, and this holds more generally for all $k$ for which $H^{k-1}(Y)=0$.
\end{rmk}

The spaces $\cH^k_{L^2,\dR}$ are related to standard cohomology groups associated with the manifold with boundary $\oX$: first, notice that elements of the space $\cD=\CI_{(0)}\cap L^2$ are not subject to any matching condition on singular terms, simply because the singular term ($\wt u_{NT}|_Y$ in the notation used above) vanishes. This means that we can split $\cD$ into tangential and normal forms, $\cD=\cD_T\oplus\cD_N$,\footnote{Thus, elements $(u_T,u_N)\in\cD_T\oplus\cD_N$ are identified with $u_T+\alpha\,dt\wedge u_N\in\cD$.} where $\cD_T$ consists of all $u_T\in\CI(\oX;\Lambda\oX)$ which are of the form
\[
  u_T = \begin{pmatrix} u_{TT} \\ \alpha u_{TN} \end{pmatrix},\quad u_{TT},u_{TN}\in\CI(\Xeven;\Lambda Y),
\]
near $Y$. Thus, elements $u_T\in\cD_T$ are forms of the type $u_T=u_{TT}+d\alpha\wedge \alpha u_{TN}=u_{TT}+\frac{1}{2}d\mu\wedge u_{TN}$ with $u_{TT},u_{TN}$ smooth $\Lambda Y$-valued forms on $\Xeven$; hence, we simply have $\cD_T=\CI(\Xeven;\Lambda\Xeven)$. Likewise, $\cD_N$ consists of all $u_N\in\CI(\oX;\Lambda\oX)$ which are of the form
\[
  u_N = \begin{pmatrix} \alpha u_{NT} \\ u_{NN} \end{pmatrix},\quad u_{NT},u_{NN}\in\CI(\Xeven;\Lambda Y),
\]
near $Y$. Thus, elements $u_N\in\cD_T$ are forms of the type $\alpha u_N=\mu u_{NT}+\frac{1}{2}d\mu\wedge u_{NN}$; therefore, $\alpha\cD_N=\CI_R(\Xeven;\Lambda\Xeven):=\{u\in\CI(\Xeven;\Lambda\Xeven)\colon j^*u=0\}$, where $j\colon\pa\Xeven\hookrightarrow\Xeven$ is the inclusion.

Since the differential $\hd$ on $\cD$ acts as $d_X\oplus(-\alpha^{-1}d_X\alpha)$ on $\cD_T\oplus\cD_N$, the cohomology of the complex $(\cD,\hd)$ in degree $k$ is the direct sum of the cohomology of $(\cD_T,d_X)$ in degree $k$ and of $(\alpha\cD_N,d_X)$ in degree $(k-1)$. Since we identified $\cD_T$ as simply the space of smooth forms on $\Xeven$, the cohomology of $(\cD_T,d_X)$ in degree $k$ equals the absolute cohomology $H^k(\Xeven)\cong H^k(\oX)$.\footnote{We use that $\Xeven$ is diffeomorphic to $\oX$, with diffeomorphism given by gluing the map $\alpha^2\mapsto\alpha$ near $Y$ to the identity map away from $Y$.} Moreover, since $\cD_N$ is the space of smooth forms on $\Xeven$ which vanish at the boundary in the precise sense described above, the cohomology of $(\alpha\cD_N,d_X)$ in degree $k$ equals the relative cohomology $H^k(\Xeven;\pa\Xeven)\cong H^k(\oX;\pa\oX)$ (see e.g.\ \cite[\S5.9]{TaylorPDE}). In summary:
\begin{prop}
\label{PropL2Coho}
  With $\cH^k_{L^2,\dR}$ defined in \eqref{EqL2CohoDef}, there is a canonical isomorphism
  \begin{equation}
  \label{EqL2CohoIsomorphism}
    \cH^k_{L^2,\dR} \cong H^k(\oX) \oplus H^{k-1}(\oX,\pa\oX).
  \end{equation}
\end{prop}

Let us summarize the results obtained in the previous sections:

\begin{thm}
\label{ThmSummary}
  The only resonance of $d+\delta$ in $\Im\sigma\geq 0$ is $\sigma=0$, and $0$ is a simple resonance. Zero resonant states of the extended operator ($d+\delta$ on $\wt M$) are uniquely determined by their restriction to $X$, and the space $\cH$ of these resonant states on $X$ is equal to $\ker_{\CI_{(0)}}\hd(0)\cap\ker_{\CI_{(0)}}\hdel(0)$. Also, resonant states on $\wt X$ are elements of $\ker\td(0)\cap\ker\tdel(0)$. Using the grading $\cH=\bigoplus_{k=0}^n\cH^k$ of $\cH$ by form degrees, there is a canonical exact sequence
  \begin{equation}
  \label{EqThmHarmonicCoho}
    0\to H^k(\oX)\oplus H^{k-1}(\oX,\pa\oX) \to \cH^k \to H^{k-1}(\pa\oX),
  \end{equation}
  where the first map is the composition of the isomorphism \eqref{EqL2CohoIsomorphism} with the map $i$ defined in Lemma~\ref{LemmaCohoMap}, and the second map is the composition of the map $r$ defined in \eqref{EqRestrictionMap} with the isomorphism $\ker(\Delta_{\pa\oX,k-1})\cong H^{k-1}(\pa\oX)$.

  Furthermore, the only resonance of $\Box_g$ in $\Im\sigma\geq 0$ is $\sigma=0$. Zero resonant states\footnote{More precisely, we mean elements of $\ker\wt\Box(0)$; the latter space equals the space of zero resonant states if the zero resonance is simple.} of the extended operator ($\Box_g$ on $\wt M$) are uniquely determined by their restriction to $X$. The space $\cK=\bigoplus_{k=0}^n\cK^k\subset\CI_{(0)}$ of these resonant states on $X$, graded by form degree, satisfying $\cK^k\supset\cH^k$, fits into the short exact sequence
  \begin{equation}
  \label{EqThmBoxCoho}
    0 \to H^k(\oX)\oplus H^{k-1}(\oX,\pa\oX) \to \cK^k \to H^{k-1}(\pa\oX) \to 0,
  \end{equation}
  with maps as above. We moreover have
  \[
    \cK^k\cap L^2 = \cH^k\cap L^2 \cong H^k(\oX)\oplus H^{k-1}(\oX,\pa\oX)
  \]
  where $L^2=L^2(X,\alpha|dh|;H\oplus H)$. More precisely then, the summand $H^k(\oX)$ in \eqref{EqThmHarmonicCoho} and \eqref{EqThmBoxCoho} corresponds to the tangential components (in the decomposition \eqref{EqFormDecomposition}) of elements of $\cH^k\cap L^2$, and the summand $H^{k-1}(\oX,\pa\oX)$ to the normal components.

  Lastly, the Hodge star operator on $M$ induces isomorphisms $\star\colon\cH^k\xra{\cong}\cH^{n-k}$ and $\star\colon\cK^k\xra{\cong}\cK^{n-k}$, $k=0,\ldots,n$.
\end{thm}
\begin{proof}
  We prove the statement about resonant states for $d+\delta$ on the extended space $\wt M$: thus, if $\wt u\in\ker(\td(0)+\tdel(0))$, then the restriction of $\wt u$ to $X$ lies in $\ker\hd(0)\cap\ker\hdel(0)$, therefore $\td(0)\wt u=-\tdel(0)\wt u$ is supported in $\wt X\setminus X$; but then $\wt\Box(0)(\td(0)\wt u)=\td(0)\tdel(0)\td(0)\wt u=0$ and the asymptotically de Sitter nature of $\wt X\setminus X$ implies $\td(0)\wt u\equiv 0$, hence also $\tdel(0)\wt u\equiv 0$, as claimed.

  The only remaining part of the statement that has not yet been proved is the last: viewing $u\in\cH^k$ as a $t$-independent $k$-form on $M=\R_t\times X$ (with the metric \eqref{EqMetricSpacetime}), we have $(d+\delta)u=0$, and for any $t$-independent $k$-form $u$ on $M$, we have that $(d+\delta)u=0$ implies $u\in\cH^k$, where we view the $t$-independent form as a form on $X$ valued in the form bundle of $M$, as explained in Section~\ref{SecPrelim}. Then $u\in\cH^k$ is equivalent to $du=0$, $\delta u=0$, which in turn is equivalent to $\delta(\star u)=0$, $d(\star u)=0$, and thus $\star u\in\cH^{n-k}$. The proof for the spaces $\cK^k$ is the same and uses $\star\Box=\Box\star$.
\end{proof}

This in particular proves Theorem~\ref{ThmIntroSummary}.

\section{Results for static de Sitter and Schwarzschild--de Sitter spacetimes}
\label{SecApplications}

We now supplement the results obtained in the previous section by high energy estimates for the inverse normal operator family and deduce expansions and decay for solutions to Maxwell's equations as well as for more general linear waves on de Sitter and Schwarzschild--de Sitter backgrounds. The rather detailed description of asymptotics in the Schwarzschild--de Sitter setting will be essential in our discussion of Kerr--de Sitter space in Section~\ref{SecKdS}.

\subsection{de Sitter space}
\label{SubsecDS}

De Sitter space is the hyperboloid $\{|x|^2-\wt t^2=1\}$ in $(n+1)$-dimensional Minkowski space, equipped with the induced Lorentzian metric. Introducing $\tau=\wt t^{-1}$ in $\wt t\geq 1$ and adding the boundary at future infinity $\tau=0$ to the spacetime, we obtain the bordified space $N=[0,1)_\tau\times Z$ with $Z=\Sph^{n-1}$, and the metric has the form
\[
  g^0=\tau^{-2}\bar g,\quad \bar g=d\tau^2-h^0(\tau,x,dx),
\]
with $h^0$ even in $\tau$, i.e.\ $h^0$ is a metric on $Z$ which depends smoothly on $\tau^2$; see Vasy~\cite[\S{4}]{VasyMicroKerrdS} for details. Thus, $g^0$ is a 0-metric in the sense of Mazzeo and Melrose \cite{MazzeoMelroseHyp}. Fixing a point $p$ at future infinity, the static model of de Sitter space, denoted $M$, is the interior of the backward light cone from $p$.\footnote{Since $g^0$ and the metric $\bar g$, which is smooth down to $\tau=0$, are conformally related, the images of null-geodesics for both metrics agree.} We introduce \emph{static coordinates} on $M$, denoted $(t,x)\in\R\times X$, where $X=B_1\subset\R^{n-1}$ is the open unit ball in $\R^{n-1}$ and $x\in\R^{n-1}$ are the standard coordinates on $\R^{n-1}$, with respect to which the induced metric on $M$ is given by
\begin{equation}
\label{EqdSMetric}
\begin{gathered}
  g = \alpha^2\,dt^2 - h, \quad \alpha=(1-|x|^2)^{1/2}, \\
  h=dx^2+\frac{1}{1-|x|^2}(x\cdot dx)^2 = \alpha^{-2}\,dr^2+r^2\,d\omega^2,
\end{gathered}
\end{equation}
using polar coordinates $(r,\omega)$ on $\R^{n-1}_x$ near $r=1$, and denoting the round metric on the unit sphere $\Sph^{n-2}$ by $d\omega^2$. We compactify $X$ to the closed unit ball $\Xeven=\overline{B_1}\subset\R^{n-1}$, and denote by $\oX$ the space which is $\Xeven$ topologically, but with $\alpha$ added to the smooth structure. In order to see that the metric $g$ fits into the framework of Theorem~\ref{ThmSummary}, note that $dr=-\alpha r^{-1}\,d\alpha$, so
\[
  h = r^{-2}\,d\alpha^2 + r^2\,d\omega^2,
\]
and $r=(1-\alpha^2)^{1/2}$, thus $h$ is an even metric on the space $\oX$ and has the form \eqref{EqMetricSpacetime2} with $\beta=1$. Using Theorem~\ref{ThmSummary}, we can now easily compute the spaces of resonances:

\begin{thm}
\label{ThmDSFull}
  On an $n$-dimensional static de Sitter spacetime, $n\geq 4$, the spaces of resonances of $\Box$ and $d+\delta$ are
  \[
    \cK^0=\cH^0=\la 1\ra,\quad \cK^n=\cH^n=\la r^{n-2}\,dt\wedge dr\wedge\omega\ra,
  \]
  where $\omega$ denotes the volume form on the round sphere $\Sph^{n-2}$. Furthermore,
  \begin{gather*}
    \cK^1=\la -\alpha^{-2}r\,dr+\alpha^{-1}\,dt \ra, \cH^1=0, \quad \cK^{n-1}=\la \star(-\alpha^{-2}r\,dr+\alpha^{-1}\,dt) \ra, \cH^{n-1}=0, \\
	\cK^k=\cH^k=0,\quad k=2,\ldots,n-2.
  \end{gather*}
\end{thm}
\begin{proof}
  We compute the cohomological data that appear in \eqref{EqThmHarmonicCoho} and \eqref{EqThmBoxCoho} using $\oX\cong\overline{B_1}$ and $\pa\oX\cong\Sph^{n-2}$:
  \begin{align*}
    \dim H^{k-1}(\pa\oX) &= \begin{cases} 0, & k=0,2,\ldots,n-2,n, \\ 1, & k=1,n-1 \end{cases} \\
    \dim H^k(\oX) &= \begin{cases} 1, & k=0 \\ 0, & 1\leq k\leq n, \end{cases} \\
    \dim H^{k-1}(\oX,\pa\oX) &= \begin{cases} 0, & 0\leq k\leq n-1 \\ 1, & k=n. \end{cases}
  \end{align*}
  Thus, we immediately deduce
  \begin{gather*}
    \dim\cK^0=\dim\cK^1=\dim\cK^{n-1}=\dim\cK^n=1, \quad \dim\cK^k=0, \quad 2\leq k\leq n-2, \\
    \dim\cH^0=\dim\cH^n=1, \quad \dim\cH^k=0, \quad 2\leq k\leq n-2.
  \end{gather*}
  Now, since $d+\delta$ annihilates constants, we find $1\in\cK^0=\cH^0$ and $\star 1\in\cK^n=\cH^n$, which in view of the 1-dimensionality of these spaces already concludes their computation.

  In order to compute $\cK^1$, notice that we have $\cK^1\cong H^0(\pa\oX)$ from \eqref{EqThmBoxCoho}, thus an element $u$ spanning $\cK^1$ has non-trivial singular components at $\alpha=0$. One is led to the guess $u=\alpha^{-1}\,d\alpha+\alpha^{-1}\,dt=-\alpha^{-2}r\,dr+\alpha^{-1}\,dt$, which is indeed annihilated by $\Box$; we will give full details for this computation in the next section when discussing Schwarzschild--de Sitter spacetimes, which in the case of vanishing black hole mass are static de Sitter spacetimes, with a point removed, see in particular the calculations following \eqref{EqOneFormAnsatz}; but since $u$ as defined above is smooth at $r=0$, we obtain $\Box u=0$ at $r=0$ as well by continuity. Since $\cK^1$ is 1-dimensional, we therefore deduce $\cK^1=\la u\ra$. One can then check that $(d+\delta)u\neq 0$, and this implies $\cH^1=0$. The corresponding statements for $\cK^{n-1}$ and $\cH^{n-1}$ are immediate consequences of this and the fact that the Hodge star operator induces isomorphisms $\cH^1\cong\cH^{n-1}$ and $\cK^1\cong\cK^{n-1}$.
\end{proof}

In particular:

\begin{thm}
\label{ThmDSFull2}
  On $4$-dimensional static de Sitter space, if $u$ is a solution of $(d+\delta)u=0$ with smooth initial data, then the degree $0$ component of $u$ decays exponentially to a constant, the degree $1,2$ and $3$ components decay exponentially to $0$, and the degree $4$ component decays exponentially to a constant multiple of the volume form. Analogous statements hold on any $n$-dimensional static de Sitter space, $n\geq 5$.
\end{thm}
\begin{proof}
  The high energy estimates for $d+\delta$ required to deduce asymptotic expansions for solutions of $(d+\delta)u=0$ follow from those of its square $\Box$, which is principally scalar and fits directly into the framework recalled in Section~\ref{SecPrelim} above, and is described in detail in~\cite[\S{2-4}]{VasyMicroKerrdS}: we can apply \cite[Theorem~2.14]{VasyMicroKerrdS}, with $R(\sigma)=\wt\Box(\sigma)^{-1}$ for the high energy estimates and then use \cite[Lemma~3.1]{VasyMicroKerrdS} (with $\cP=\Box$, $\cQ=0$, $\tau=e^{-t_*}$) to obtain the resonance expansion.
\end{proof}

By studying the space of dual resonant states, one can in fact easily show that the 0-resonance of $\Box$ is simple and thus deduce exponential decay of smooth solutions to $\Box u=0$ to an element of $\cK^k$ in all form degrees $k=0,\ldots,n$. We give details in the next section on Schwarzschild--de Sitter space.

In the present de Sitter setting, one can deduce asymptotics very easily in a different manner using the global de Sitter space picture, by analyzing indicial operators in the 0-calculus: concretely, we write differential $k$-forms (by which we mean smooth sections of the $k$-th exterior power of the $0$-cotangent bundle of $N$) as
\begin{equation}
\label{EqZeroForms}
  u = \tau^{-k}u_T + \frac{d\tau}{\tau}\wedge\tau^{1-k}u_N,
\end{equation}
where $u_T$ and $u_N$ are smooth forms on $Z$ of form degrees $k$ and $(k-1)$, respectively. One readily computes the differential $d_k$ acting on $k$-forms to be
\[
  d_k = \begin{pmatrix} \tau d_Z & 0 \\ -k+\tau\pa_\tau & -\tau d_Z \end{pmatrix}.
\]
Furthermore, by the choice of basis in \eqref{EqZeroForms}, the inner product on $k$-forms induced by $g^0$ is given by
\[
  G_k^0 = \begin{pmatrix} (-1)^kH_k^0 & 0 \\ 0 & (-1)^{k-1}H_{k-1}^0 \end{pmatrix}.
\]
Using that the volume density is $|dg^0|=\tau^{-n}\,d\tau|dh^0|$, we compute the codifferential $\delta_k$ acting on $k$-forms to be
\[
  \delta_k = \begin{pmatrix} -\tau\delta_Z & -(k-1) + \tau^{n-1}\tau\pa_\tau^*\tau^{1-n} \\ 0 & \tau\delta_Z \end{pmatrix}
   = \begin{pmatrix} -\tau\delta_Z & n-k-\tau\pa_\tau + \cO_{\CI(N)}(\tau) \\ 0 & \tau\delta_Z \end{pmatrix},
\]
where $\pa_\tau^*$ is the $L^2(N,|d\bar g|)$-adjoint (suppressing the bundles in the notation) of $\pa_\tau$, and we use the even-ness of $g^0$ in the second step to deduce $\pa_\tau^*=-\pa_\tau+\cO_{\CI(N)}(\tau)$. Therefore, the indicial roots of $d+\delta$ on the degree $k$-part of the form bundle are $k$ and $n-k$.

Next, for $0\leq k\leq n$, we compute the Hodge d'Alembertian:\footnote{We deal with the cases $k=0$ and $k=n$ simultaneously with $1\leq k\leq n-1$ by implicitly assuming that for $k=0$, only the $(1,1)$-part of this operator acts on $0$-forms, and for $k=n$, only the $(2,2)$-part acts on $n$-forms.}
\begin{align*}
  \Box_k & = d_{k-1}\delta_k + \delta_{k+1}d_k \\
    & =
	\begin{pmatrix}
	  -\tau d_Z\tau\delta_Z-\tau\delta_Z\tau d_Z - P_k & \tau d_Z \\
	  -\tau\delta_Z & -\tau d_Z\tau\delta_Z-\tau\delta_Z\tau d_Z - P_{k-1}
	\end{pmatrix}
	+ \cO_{\Diff_0^1}(\tau)
\end{align*}
where $P_k = (\tau\pa_\tau)^2-(n-1)\tau\pa_\tau + k(n-k-1)$. Thus, the indicial polynomial of $\Box_k$ is
\[
  I(\Box_k)(s) = \begin{pmatrix} s^2-(n-1)s+k(n-k-1) & 0 \\ 0 & s^2-(n-1)s+(k-1)(n-k) \end{pmatrix}.
\]
On tangential forms, the indicial roots of $\Box_k$ are therefore $k,n-1-k$, and on normal forms, they are $k-1,n-k$. We thus have:
\begin{center}
\begin{tabular}{l|l|l|l|l|l}
  form degree & $0$ & $1$ & $2\leq k\leq n-2$ & $n-1$ & $n$ \\
  \hline\hline
  tgt.\ ind.\ roots & $0,n-1$ & $1,n-2$ & $k,n-1-k$ & $0,n-1$ & $-$ \\
  \hline
  norm.\ ind.\ roots & $-$ & $0,n-1$ & $k-1,n-k$ & $1,n-2$ & $0,n-1$
\end{tabular}
\end{center}
Hence in particular, all roots are $\geq 0$, and $0$ is never a double root. Thus, the arguments of \cite{VasyWaveOndS} (which are in the scalar setting, but work in the current setting as well with only minor modifications) show that solutions $u$ to the wave equation on differential $k$-forms on $N$ with smooth initial data at $\tau=\tau_0>0$ decay exponentially (in $-\log\tau$) if $0$ is not an indicial root, and decay to a stationary state if $0$ is an indicial root.\footnote{Of course, since we know all indicial roots, we could be much more precise in describing the asymptotics, but we only focus on the $0$-resonance here.} Explicitly, scalar waves decay to a smooth function on $Z$, $1$-form waves decay to an element of $\frac{d\tau}{\tau}\CI(Z)$, $k$-form waves decay exponentially to $0$ for $2\leq k\leq n-2$, $(n-1)$-form waves decay to an element of $\CI(Z;\Lambda^{n-1}Z)$, and $n$-form waves finally decay to an element of $\frac{d\tau}{\tau}\wedge\CI(Z;\Lambda^{n-1}Z)$.

Since the static model of de Sitter space arises by blowing up a point $p$ at future infinity of compactified de Sitter space and considering the backward light cone from $p$, we can find the resonant states for the static model by simply finding the space of restrictions to $p$ of the asymptotic states described above; but since the fibers of $\Lambda^0(Z)$ and $\Lambda^{n-1}(Z)$ are 1-dimensional, hence we have reproved Theorem~\ref{ThmDSFull}.

We point out that if one wants to analyze differential form-valued waves or solutions to Maxwell's equations on Schwarzschild--de Sitter space, there is no global picture (in the sense of a 0-differential problem) as in the de Sitter case. Thus, the direct approach outlined in the proof of Theorem~\ref{ThmDSFull} is the only possible one in this case, and it is very instructive as it shows even more clearly how the cohomological interpretation of the space of zero resonant states can be used very effectively.

\subsection{Schwarzschild--de Sitter space}
\label{SubsecSDS}

The computation of resonant states for Schwarzschild--de Sitter spacetimes of any dimension is no more difficult than the computation in $4$ dimensions, thus we directly treat the general case of $n\geq 4$ spacetime dimensions. Recall that the metric of $n$-dimensional Schwarzschild--de Sitter space $M=\R_t\times X$, $X=(r_-,r_+)_r\times\Sph^{n-2}_\omega$, with $r_\pm$ defined below, is given by
\[
  g = \mu\,dt^2-(\mu^{-1}\,dr^2+r^2\,d\omega^2),
\]
where $d\omega^2$ is the round metric on the sphere $\Sph^{n-2}$, and $\mu=1-\frac{2M_\bullet}{r^{n-3}}-\lambda r^2$, $\lambda=\frac{2\Lambda}{(n-2)(n-1)}$, where the black hole mass $M_\bullet$ and the cosmological constant $\Lambda$ are positive. We assume that
\begin{equation}
\label{EqSDSNondegeneracy}
  M_\bullet^2\lambda^{n-3}<\frac{(n-3)^{n-3}}{(n-1)^{n-1}},
\end{equation}
which guarantees that $\mu$ has two unique positive roots $0<r_-<r_+$. Indeed, let $\wt\mu=r^{-2}\mu=r^{-2}-2M_\bullet r^{1-n}-\lambda$. Then $\wt\mu' = -2r^{-n}(r^{n-3}-(n-1)M_\bullet)$ has a unique positive root $r_p=[(n-1)M_\bullet]^{1/(n-3)}$, $\wt\mu'(r)>0$ for $r\in(0,r_p)$ and $\wt\mu'(r)<0$ for $r>r_p$; moreover, $\wt\mu(r)<0$ for $r>0$ small and $\wt\mu(r)\to-\lambda<0$ as $r\to\infty$, thus the existence of the roots $0<r_-<r_+$ of $\wt\mu$ is equivalent to the requirement $\wt\mu(r_p)=\frac{n-3}{n-1}r_p^{-2}-\lambda>0$, which leads precisely to the inequality \eqref{EqSDSNondegeneracy}.

Define $\alpha=\mu^{1/2}$, thus $d\alpha=\frac{1}{2}\mu'\alpha^{-1}\,dr$, and
\[
  \beta_\pm:=\mp\frac{2}{\mu'(r_\pm)}>0,
\]
then the metric $g$ can be written as
\begin{equation}
\label{EqSdSMetric}
  g = \alpha^2\,dt^2-h,\quad h=\wt\beta_\pm^2\,d\alpha^2+r^2\,d\omega^2,
\end{equation}
where $\wt\beta_\pm=\mp 2/\mu'(r)$. Thus, if we let $\Xeven=[r_-,r_+]_r\times\Sph^{n-2}_\omega$ with the standard smooth structure, then $\wt\beta_\pm=\beta_\pm$ modulo $\alpha^2\CI(\Xeven)$, where we note that $r$ is a smooth function of $\mu$, thus an \emph{even} function of $\alpha$, near $r=r_\pm$ in view of $\mu'(r_\pm)\neq 0$. The manifold $\oX$ is $\Xeven$ topologically, but with smooth functions of $\alpha=\mu^{1/2}$ added to the smooth structure. We denote $Y=\pa\oX=\Sph^{n-2}\sqcup\Sph^{n-2}$.

By the analysis in Section~\ref{SecPrelim}, all zero resonant states $u$, written in the form \eqref{EqBundleDecWarped} near $Y$, lie in the space $\CI_{(0)}$, defined in \eqref{EqDefCIsigma}. In the current setting, it is more natural to write differential forms as
\begin{equation}
\label{EqSDSDecomposition}
  u=u_{TT}+\alpha^{-1}\,dr\wedge u_{TN} + \alpha\,dt\wedge u_{NT} + \alpha\,dt\wedge\alpha^{-1}\,dr\wedge u_{NN},
\end{equation}
since $\alpha^{-1}\,dr$ has squared norm $-1$ (with respect to the metric $g$). We compute how the matching condition on the singular terms of $u$, encoded in the $\beta_\pm\alpha^{-1}$ entry of the matrix $\sC$, changes when we thus change the basis of the form bundle: namely, we have $\beta_\pm\alpha^{-1}\,d\alpha=(\mp 1+\alpha^2\CI(\Xeven))\alpha^{-1}\alpha^{-1}\,dr$; thus, for $u$ written as in \eqref{EqSDSDecomposition}, we have
\[
  u\in\CI_{(0)} \iff
  \begin{pmatrix}
    u_{TT}\\u_{TN}\\u_{NT}\\u_{NN}
  \end{pmatrix}
   \in \sC_\pm
     \begin{pmatrix}
	   \CI(\Xeven;\Lambda\Sph^{n-2})\\
	   \CI(\Xeven;\Lambda\Sph^{n-2})\\
	   \CI(\Xeven;\Lambda\Sph^{n-2})\\
	   \CI(\Xeven;\Lambda\Sph^{n-2})
	 \end{pmatrix}
\]
near $r=r_\pm$, where
\begin{equation}
\label{EqSDSMatchingCondition}
  \sC_\pm
    =\begin{pmatrix}
         1&0&0&0\\
		 0&\alpha&\mp\alpha^{-1}&0\\
		 0&0&\alpha^{-1}&0\\
		 0&0&0&1
	  \end{pmatrix}.
\end{equation}
We now proceed to compute the explicit form of the operators $d_p,\delta_p$ and $\Box_p$, where the subscript $p$ indicates the form degree on which the operators act. First, we recall \eqref{EqSpacetimeD} and \eqref{EqSpacetimeDel} in the form
\[
  d_p=\begin{pmatrix} d_{X,p} & 0 \\ \alpha^{-1}\pa_t & -\alpha^{-1}d_{X,p-1}\alpha \end{pmatrix},\quad
  \delta_p=\begin{pmatrix} -\alpha^{-1}\delta_{X,p}\alpha & -\alpha^{-1}\pa_t \\ 0 & \delta_{X,p-1} \end{pmatrix},
\]
and these operators act on forms $u=u_T+\alpha\,dt\wedge u_N$, with $u_T$ and $u_N$ differential forms on $X$. Writing forms on $X$ as $v=v_T+\alpha^{-1}\,dr\wedge v_N$, we have
\begin{equation}
\label{EqSDSDiff}
  d_{X,p} = \begin{pmatrix} d_{\Sph^{n-2},p} & 0 \\ \alpha\pa_r & -d_{\Sph^{n-2},p-1} \end{pmatrix}.
\end{equation}
In order to compute the codifferential, we observe that the volume density on $X$ induced by $h$ is given by $\alpha^{-1}r^{n-2}\,dr|d\omega|$, while the induced inner product on the fibers on the bundle of $p$-forms is
\[
  H_p = \begin{pmatrix} r^{-2p}\Omega_p & 0 \\ 0 & r^{-2(p-1)}\Omega_{p-1} \end{pmatrix},
\]
where $\Omega_p$ is the fiber inner product on the $p$-form bundle on $\Sph^{n-2}$. Therefore,
\begin{equation}
\label{EqSDSPaRAdjoint}
\begin{split}
  \delta_{X,p} &= \begin{pmatrix} r^{-2}\delta_{\Sph^{n-2},p} & \pa_{r,p-1}^* \\ 0 & -r^{-2}\delta_{\Sph^{n-2},p-1} \end{pmatrix}, \\
  \pa_{r,p-1}^* &= -\alpha r^{-(n-2)}r^{2(p-1)}\pa_r r^{-2(p-1)}r^{n-2}.
\end{split}
\end{equation}
We obtain:
\begin{lemma}
\label{LemmaSDSOps}
  In the bundle decomposition \eqref{EqSDSDecomposition}, we have
  \begin{equation}
  \label{EqSDSD}
    d_p
      = \begin{pmatrix}
  	    d_{\Sph^{n-2},p} & 0 & 0 & 0 \\
  		\alpha\pa_r & -d_{\Sph^{n-2},p-1} & 0 & 0 \\
  		\alpha^{-1}\pa_t & 0 & -d_{\Sph^{n-2},p-1} & 0 \\
  		0 & \alpha^{-1}\pa_t & -\pa_r\alpha & d_{\Sph^{n-2},p-2}
  	  \end{pmatrix}
  \end{equation}
  and
  \begin{equation}
  \label{EqSDSDel}
    \delta_p
      = \begin{pmatrix}
  	    -r^{-2}\delta_{\Sph^{n-2},p} & -\alpha^{-1}\pa_{r,p-1}^*\alpha & -\alpha^{-1}\pa_t & 0 \\
  		0 & r^{-2}\delta_{\Sph^{n-2},p-1} & 0 & -\alpha^{-1}\pa_t \\
  		0 & 0 & r^{-2}\delta_{\Sph^{n-2},p-1} & \pa_{r,p-2}^* \\
  		0 & 0 & 0 & -r^{-2}\delta_{\Sph^{n-2},p-2}
  	  \end{pmatrix}.
  \end{equation}
  Moreover,
  \begin{equation}
  \label{EqSDSBox}
  \begin{split}
  -&r^2\Box_p=
    \begin{pmatrix}
	  \Delta_{\Sph^{n-2},p} & -2\alpha r d_{p-1} & 0 & 0 \\
	  -2\alpha r^{-1}\delta_p & \Delta_{\Sph^{n-2},p-1} & -r^2\mu^{-1}\mu'\pa_t & 0 \\
	  0 & -r^2\mu^{-1}\mu'\pa_t & \Delta_{\Sph^{n-2},p-1} & -2\alpha r d_{p-2} \\
	  0 & 0 & -2\alpha r^{-1}\delta_{p-1} & \Delta_{\Sph^{n-2},p-2}
	\end{pmatrix} \\
  &\quad +
    \begin{pmatrix}
	  r^2\alpha^{-1}\pa_{r,p}^*\alpha^2\pa_r & 0 & 0 & 0 \\
	  0 & r^2\alpha\pa_r\alpha^{-1}\pa_{r,p-1}^*\alpha & 0 & 0 \\
	  0 & 0 & r^2\pa_{r,p-1}^*\pa_r\alpha & 0 \\
	  0 & 0 & 0 & r^2\pa_r\alpha\pa_{r,p-2}^*
	\end{pmatrix} \\
  &\quad +
    \begin{pmatrix}
	  r^2\mu^{-1}\pa_t^2 & 0 & 0 & 0 \\
	  0 & r^2\mu^{-1}\pa_t^2 & 0 & 0 \\
	  0 & 0 & r^2\mu^{-1}\pa_t^2 & 0 \\
	  0 & 0 & 0 & r^2\mu^{-1}\pa_t^2
	\end{pmatrix}.
  \end{split}
  \end{equation}
\end{lemma}

We can now compute the spaces $\cK$ and $\cH$ of zero resonances for $\Box$ and $d+\delta$ and deduce asymptotics for solutions of $(d+\delta)u=0$:
\begin{thm}
\label{ThmSDSDDel}
  On an $n$-dimensional Schwarzschild--de Sitter spacetime, $n\geq 4$, there exist two linearly independent $1$-forms $u_\pm = f_{1,\pm}(r)\mu^{-1}\,dr + f_{2,\pm}(r)\,dt \in \cK^1=\ker\wh\Box_1\subset \CI_{(0)}$,\footnote{The forms $u_\pm$ have a simple explicit form, see \eqref{EqOneFormAnsatz} and Footnote~\ref{FootOneFormResonance}.} and we then have:
  \[
    \cK^0=\cH^0=\la 1\ra,\quad \cK^n=\cH^n=\la r^{n-2}\,dt\wedge dr\wedge\omega\ra,
  \]
  where $\omega$ denotes the volume form on the round sphere $\Sph^{n-2}$. Furthermore,
  \begin{gather*}
    \cK^1=\la u_+,u_-\ra,\cH^1=0,\quad \cK^{n-1}=\la\star u_+,\star u_-\ra,\cH^{n-1}=0, \\
	\cK^k=\cH^k=0,\quad k=3,\ldots,n-3.
  \end{gather*}
  For $n=4$,
  \[
    \cK^2=\cH^2 = \la\omega,r^{-2}\,dt\wedge dr\ra,
  \]
  while for $n>4$,
  \[
    \cK^2=\cH^2 = \la r^{-(n-2)}\,dt\wedge dr\ra,\quad \cK^{n-2}=\cH^{n-2} = \la \omega \ra.
  \]
\end{thm}

\begin{proof}
  First, we observe that $H^k(\oX)\cong H^k(\Sph^{n-2})$, $H^{k-1}(\oX,\pa\oX)\cong H^{n-k}(\oX)\cong H^{n-k}(\Sph^{n-2})$ by Poincar\'e duality, and $H^{k-1}(\pa\oX)\cong H^{k-1}(\Sph^{n-2})\oplus H^{k-1}(\Sph^{n-2})$. Thus, the short exact sequence \eqref{EqThmBoxCoho} immediately gives the dimensions of the spaces $\cK^k$, and \eqref{EqThmHarmonicCoho} gives the dimensions of $\cH^k$ for all values of $k$ except $k=1$ and $k=n-1$.
  
  We now compute $\cH$ and $\cK$ in the case $n=4$. For $k=0$, the short exact sequence \eqref{EqThmBoxCoho} reads $0\to H^0(\oX)\oplus 0\to\cK^0\to 0\to 0$, and since $H^0(\oX)=\la[1]\ra$, this suggests $1$ as a resonant state for $\Box$ on $0$-forms (i.e.\ functions), and indeed $\Box 1=0$, hence $\cK^0=\la 1\ra$. Theorem~\ref{ThmSummary} also shows that $\cH^0=\cK^0$. Then we immediately obtain $\cH^4=\cK^4=\la\star 1\ra=\la r^2\,dt\wedge dr\wedge\omega\ra$.

  Next, we treat the form degree $k=2$. Then \eqref{EqThmBoxCoho} reads $0\to H^2(\oX)\oplus H^1(\oX,\pa\oX)\to\cK^2\to 0\to 0$. Now $H^2(\oX)=\la[\omega]\ra$, and a generator of $H^1(\oX,\pa\oX)$ is given by the Poincar\'e dual of $\omega$ (which generates $H^2(\oX)$). This suggests the ansatz $u=f(r)\omega$ for an element of $\cK^2=\cH^2$ (the latter equality following from \eqref{EqThmHarmonicCoho}), and then $\star u$ will be the second element of a basis of $\cK^2$. Now, in the decomposition \eqref{EqSDSDecomposition}, we compute using Lemma~\ref{LemmaSDSOps} that $\hdel_2(0)u=0$ for $u=f(r)\omega$, and
  \[
    \hd_2(0)u = \hd_2(0)\begin{pmatrix}f(r)\omega\\0\\0\\0\end{pmatrix} = \begin{pmatrix}0 \\ \alpha f'(r)\omega \\ 0 \\ 0\end{pmatrix},
  \]
  which vanishes precisely if $f(r)$ is constant.

  The analysis of resonant states in form degree $k=1$ is just a bit more involved. Since \eqref{EqThmBoxCoho} now reads $0\to 0\oplus 0\to\cK^1\to H^0(\Sph^2\sqcup\Sph^2)\to 0$, every non-trivial element $u$ of $\cK^1$ fails to be in $L^2(\alpha|dh|)$, and in fact the singular behavior is expected to be $u=\sC_\pm\wt u$ with $\wt u_{NT}|_{r=r_\pm}=c_\pm\in\C$, since $H^0(\Sph^2\sqcup\Sph^2)$ is generated by locally constant functions, which are therefore constant on $r=r_-$ as well as on $r=r_+$. We thus make the ansatz
  \begin{equation}
  \label{EqOneFormAnsatz}
    u=\alpha^{-1}f_1(r)\,\alpha^{-1}\,dr+\alpha\,dt\wedge \alpha^{-1}f_2(r).
  \end{equation}
  We then compute
  \[
    -\wh\Box_1(0)u=\begin{pmatrix} 0 \\ \alpha\pa_r\alpha^{-1}\pa_{r,0}^* f_1 \\ \pa_{r,0}^*\pa_r f_2 \\ 0 \end{pmatrix},
  \]
  and by definition of $\pa_{r,p}^*$ in \eqref{EqSDSPaRAdjoint}, this vanishes if and only if $f_1$ and $f_2$ satisfy the ODEs
  \begin{align*}
    \pa_r r^{-2}\pa_r r^2 f_1 &= 0, \\
	r^{-2}\pa_r r^2\pa_r f_2 &= 0.
  \end{align*}
  The general form of the solution is\footnote{\label{FootOneFormResonance}On $n$-dimensional Schwarzschild--de Sitter space, the exponents $2$ and $-2$ in these ODEs get replaced by $n-2$ and $2-n$, and the general forms of the solutions are $f_1(r)=f_{11}r+f_{12}r^{2-n}$ and $f_2(r)=f_{21}+f_{22}r^{3-n}$. The subsequent analysis of the matching conditions goes through with obvious modifications.}
  \[
  \begin{split}
    f_1(r) &= f_{11}r + f_{12}r^{-2}, \\
    f_2(r) &= f_{21} + f_{22}r^{-1},
  \end{split}
  \]
  $f_{jk}\in\C$, $j,k=1,2$. Now recall that resonant states are elements of $\CI_{(0)}$ and thus satisfy a matching condition in the singular components, which is captured by the matrix \eqref{EqSDSMatchingCondition}. Concretely, we require $f_2(r_-)=f_1(r_-)$ and $f_2(r_+)=-f_1(r_+)$; in terms of $f_{jk}$, $j,k=1,2$, these conditions translate into
  \[
    \begin{pmatrix}
	  r_- & r_-^{-2} & -1 & -r_-^{-1} \\
	  r_+ & r_+^{-2} & 1 & r_+^{-1}
	\end{pmatrix}
	\begin{pmatrix}
	  f_{11} \\ f_{12} \\ f_{21} \\ f_{22}
	\end{pmatrix}
	=
	\begin{pmatrix}
	  0 \\ 0
	\end{pmatrix}.
  \]
  Since the $2\times 4$ matrix on the left has rank $2$, we get a $2$-dimensional space of solutions. In fact, it is easy to see that we can freely specify the values $f_1(r_-)$ and $f_1(r_+)$, and $f_1$ and $f_2$ are then uniquely determined. To be specific, we can for instance define $u_+\in\cK^1$ to be the 1-form with $f_1(r_-)=0,f_1(r_+)=1$, and $u_-\in\cK^1$ to be the 1-form with $f_1(r_-)=1,f_1(r_+)=0$, and we then have $\cK^1=\la u_+,u_-\ra$, as claimed.

  Next, since $\cH^1\subset\cK^1$, computing $\cH^1$ simply amounts to finding all linear combinations of $u_-$ and $u_+$ which are annihilated by both $\hd_1(0)$ and $\hdel_1(0)$. But
  \[
    \hd_1(0)\begin{pmatrix}0 \\ \alpha^{-1}f_1(r) \\ \alpha^{-1}f_2(r) \\ 0 \end{pmatrix} = \begin{pmatrix}0\\0\\0\\-\pa_r f_2\end{pmatrix}=0
  \]
  requires $f_2$ to be constant, and
  \[
    \hdel_1(0)\begin{pmatrix}0 \\ \alpha^{-1}f_1(r) \\ \alpha^{-1}f_2(r) \\ 0 \end{pmatrix} = \begin{pmatrix}-\alpha^{-1}\pa_{r,0}^*f_1\\0\\0\\0\end{pmatrix}=0
  \]
  implies $r^{-2}\pa_r r^2 f_1=0$, hence $f_1(r)=f_1(r_-)(r/r_-)^{-2}$. The matching condition requires $f_1(r_+)=f_1(r_-)(r_+/r_-)^{-2}=-f_2(r_+)=-f_2(r_-)=-f_1(r_-)$ and is therefore only satisfied if $f_1(r_-)=0$, which implies $f_1\equiv 0$ and $f_2\equiv 0$. This shows that $\cH^1=0$ and finishes the computation of the spaces of resonances for $n=4$. The computation for spacetime dimensions $n\geq 5$ is completely analogous.
\end{proof}

In particular:
\begin{thm}
\label{ThmSDSDDel2}
  On $4$-dimensional Schwarzschild--de Sitter space, if $u$ is a solution of $(d+\delta)u=0$ with smooth initial data, then the degree $0$ component of $u$ decays exponentially to a constant, the degree $1$ and degree $3$ components decay exponentially to $0$, the degree $2$ component decays exponentially to a linear combination of $\omega$ and $r^{-2}\,dt\wedge dr$, and the degree $4$ component decays exponentially to a constant multiple of the volume form. Analogous statements hold on any $n$-dimensional Schwarzschild--de Sitter space, $n\geq 5$.
\end{thm}
\begin{proof}
  This follows from the above computations combined with high energy estimates for $d+\delta$, which follow from those for $\Box$, and Lemma~\ref{LemmaPoleOrder}.\footnote{If $(\hd(\sigma)+\hdel(\sigma))^{-1}$ had a second order pole at $0$, then solutions to $(d+\delta)u=0$ would generically blow up linearly; the simplicity of the pole ensures that solutions stay bounded with the asymptotic stationary state given by an element of $\cH$.} Once we check the normally hyperbolic nature of the trapping and show that the subprincipal symbol of $\Box$ (or a conjugated version thereof), relative to a \emph{positive definite} fiber inner product, at the trapping is smaller than $\numin/2$, where $\numin$ is the minimal expansion rate in the normal direction at the trapped set, we can use Dyatlov's result \cite{DyatlovSpectralGaps} to obtain a spectral gap below the real line, i.e.\ the absence of resonances in a small strip below the reals, which combines with the general framework of \cite{VasyMicroKerrdS} to yield the desired resonance expansion of solutions with exponentially decaying error terms: concretely, the semiclassical estimate \cite[Theorem~1]{DyatlovSpectralGaps}, in the microlocalized form given in \cite[Theorem~4.7]{HintzVasyQuasilinearKdS}, can be combined with the semiclassical estimates at the horizons (i.e.\ radial points) given in~\cite[Propositions~2.10 and 2.11]{VasyMicroKerrdS}, the real principal type propagation estimates elsewhere on the semiclassical characteristic set, as well as semiclassical elliptic estimates away from the characteristic set; see \cite[\S\S{4.4, 5}]{HintzVasyQuasilinearKdS} for further details.
  
  The dynamics of the Hamilton flow at the trapping only depend on properties of the scalar principal symbol $g$ of $\Box$. For easier comparison with \cite{DyatlovWaveAsymptotics,VasyMicroKerrdS,WunschZworskiNormHypResolvent}, we consider the operator $\cP=-r^2\Box$. We take the Fourier transform in $-t$, obtaining a family of operators on $X$ depending on the dual variable $\tau$, and then do a semiclassical rescaling, multiplying $\wh\cP$ by $h^2$, giving a second order semiclassical differential operator $P_h$, with $h=|\tau|^{-1}$, and we then define $z=h\tau$. Introduce coordinates on $T^*X$ by writing $1$-forms as $\xi\,dr+\eta\,d\omega$, and let
  \[
    \Delta_r := r^2\mu = r^2(1-\lambda r^2)-2M_\bullet r^{5-n},
  \]
  then the semiclassical principal symbol $p$ of $P_h$ is
  \[
    p = \Delta_r\xi^2 - \frac{r^4}{\Delta_r}z^2 + |\eta|^2,
  \]
  and correspondingly the Hamilton vector field is
  \[
    H_p = 2\Delta_r\xi\pa_r - \Bigl(\pa_r\Delta_r\xi^2-\pa_r\Bigl(\frac{r^4}{\Delta_r}\Bigr)z^2\Bigr)\pa_\xi + H_{|\eta|^2}
  \]
  We work with real $z$, hence $z=\pm 1$. We locate the trapped set: if $H_p r=2\Delta_r\xi=0$, then $\xi=0$, in which case $H_p^2 r=2\Delta_r H_p\xi=2\Delta_r\pa_r(r^4/\Delta_r)z^2$. Recall the definition of the function $\wt\mu=\mu/r^2=\Delta_r/r^4$, then we can rewrite this as $H_p^2r=-2\Delta_r\wt\mu^{-2}(\pa_r\wt\mu)z^2$. We have already seen that $\pa_r\wt\mu$ has a single root $r_p\in(r_-,r_+)$, and $(r-r_p)\pa_r\wt\mu<0$ for $r\neq r_p$. Therefore, $H_p^2r=0$ implies (still assuming $H_p r=0$) $r=r_p$. Thus, the only trapping occurs in the cotangent bundle over $r=r_p$: indeed, define $F(r)=(r-r_p)^2$, then $H_pF=2(r-r_p)H_p r$ and $H_p^2F=2(H_p r)^2+2(r-r_p)H_p^2 r$. Thus, if $H_p F=0$, then either $r=r_p$, in which case $H_p^2 F=2(H_p r)^2>0$ unless $H_p r=0$, or $H_p r=0$, in which case $H_p^2 F=2(r-r_p)H_p^2 r>0$ unless $r=r_p$. So $H_p F=0,p=0$ implies either $H_p^2 F>0$ or $r=r_p,H_p r=0$. Therefore, the trapped set in $T^*X$ is given by
  \[
    (r,\omega;\xi,\eta) \in \Gamma_\semi := \Bigl\{(r_p,\omega;0,\eta) \colon \frac{r^4}{\Delta_r}z^2=|\eta|^2\Bigr\},
  \]
  and $F$ is an escape function. The linearization of the $H_p$-flow at $\Gamma_\semi$ in the normal coordinates $r-r_p$ and $\xi$ equals
  \begin{align*}
    H_p\begin{pmatrix}r-r_p \\ \xi\end{pmatrix} &= \begin{pmatrix} 0 & 2r_p^4\wt\mu|_{r=r_p} \\ 2(n-3)r_p^{-4}(\wt\mu|_{r=r_p})^{-2} z^2 & 0 \end{pmatrix}\begin{pmatrix} r-r_p \\ \xi \end{pmatrix} \\
	 & \quad + \cO(|r-r_p|^2+|\xi|^2),
  \end{align*}
  where we used $\pa_{rr}\wt\mu|_{r=r_p} = -2(n-3)r_p^{-4}$, which gives $\pa_r\wt\mu=-2(n-3)r_p^{-4}(r-r_p)+\cO(|r-r_p|^2)$. The eigenvalues of the linearization are therefore equal to $\pm\numin$, where
  \[
    \numin = 2r_p\left(\frac{n-1}{1-\frac{n-1}{n-3}r_p^2\lambda}\right)^{1/2}>0.
  \]
  The expansion rate of the flow within the trapped set is $0$ by spherical symmetry, since integral curves of $H_p$ on $\Gamma_\semi$ are simply unit speed geodesics of $\Sph^{n-2}$. This shows the normal hyperbolicity (in fact, $r$-normal hyperbolicity for every $r$) of the trapping.
  
  It remains to bound the imaginary part of $\cP=-r^2\Box_g$ in terms of $\numin$ in order to obtain high energy estimates below the real line. More precisely, we need to show that
  \[
    Q:=|\tau|^{-1}\sigma_1\left(\frac{1}{2i}(\cP-\cP^*)\right) < \frac{\numin}{2}
  \]
  at the trapped set (cf.\ the discussion in \cite[\S{5.4}]{HintzVasyQuasilinearKdS}), where we take the adjoint with respect to some Riemannian inner product $B$, to be chosen, on the bundle $\Lambda^p\Sph^{n-2}\oplus\Lambda^{p-1}\Sph^{n-2}\oplus\Lambda^{p-1}\Sph^{n-2}\oplus\Lambda^{p-2}\Sph^{n-2}$; notice that $Q$ is a self-adjoint section of the endomorphism bundle of this bundle.
  
  If one does not allow more general \emph{pseudodifferential} inner products $B$, one can arrange this for a restricted range of black hole parameters in $3+1$ dimensions. Indeed, a natural guess is to use $B=H\oplus H$ in the tangential-normal decomposition~\eqref{EqBundleDecNearBdy}, thus
  \[
    B = r^{-2p}\Omega_p \oplus r^{-2(p-1)}\Omega_{p-1} \oplus r^{-2(p-1)}\Omega_{p-1} \oplus r^{-2(p-2)}\Omega_{p-2}.
  \]
  In this case, the expression~\eqref{EqSDSBox} shows that the only parts of $\cP$ that are not symmetric with respect to $B$ at the spacetime trapped set
  \[
    \Gamma=\{(t,r_p,\omega;\tau,0,\eta)\colon \frac{r^4}{\Delta_r}\tau^2=|\eta|^2\},
  \]
  are the $(2,3)$ and $(3,2)$ components; thus, taking adjoints with respect to $B$, we compute
  \[
    Q=\begin{pmatrix}
	    0 & 0 & 0 & 0 \\
		0 & 0 & \pm r^2\mu^{-1}\mu' & 0 \\
		0 & \pm r^2\mu^{-1}\mu' & 0 & 0 \\
		0 & 0 & 0 & 0
	  \end{pmatrix}
  \]
  at $\Gamma$, with the sign depending the sign of $\tau$. Now $(\mu/r^2)'=0$ at $r=r_p$ implies $\mu^{-1}\mu' r_p^2=2r_p$; the eigenvalues of $Q$ are therefore $\pm 2r_p$, and they are bounded by $\numin/2$ if and only if
  \[
    r_p^2\lambda > \frac{(5-n)(n-3)}{4(n-1)},
  \]
  which in spacetime dimensions $n\geq 5$ is always satisfied. In dimension $n=4$ however, the condition becomes $r_p^2\lambda > 1/12$, or
  \[
    9 \Lambda M_\bullet^2 > \frac{1}{4},
  \]
  while the non-degeneracy condition~\eqref{EqSDSNondegeneracy} requires $9 \Lambda M_\bullet^2<1$. Therefore, only for very massive black holes or very large cosmological constants does the above choice of positive definite inner product $B$ yield a sufficiently small imaginary part of $\cP$.\footnote{In fact, one can check that for parameters $M_\bullet$ and $\Lambda$ with $9 \Lambda M_\bullet^2\leq 1/4$, the endomorphism $Q$ is not bounded by $\numin/2$ for any choice of $B$.} To overcome this problem, one needs to allow $B$ to be a \emph{pseudodifferential inner product} on the form bundle, introduced by Hintz \cite{HintzPsdoInner}: such an inner product depends on the position in phase space rather than physical space; equivalently, one can replace $\cP$ by $\cQ\cP\cQ^{-1}$, where $\cQ$, an elliptic pseudodifferential operator acting on the form bundle, is chosen in such a way that $\cQ\cP\cQ^{-1}$, relative to a Riemannian inner product on the form bundle, e.g.\ $B$, has (arbitrarily) small imaginary part, which is in particular bounded by $\numin/2$. That such an inner product can be chosen is proved for general tensor-valued waves on Schwarzschild--de Sitter spacetimes with spacetime dimension $n\geq 4$ in \cite[Theorem~4.8]{HintzPsdoInner}; see also \cite[Theorem~2.1]{HintzPsdoInner} for the resulting resonance expansion of tensor-valued waves. The point of view of pseudodifferential inner products shows precisely which structure of the subprincipal part of $\Box$ at the trapped set makes such a choice of a pseudodifferential inner product (equivalently, a choice of a conjugating operator $\cQ$) possible.
\end{proof}

We can in fact prove boundedness and asymptotics for solutions of the wave equation on differential forms in all form degrees as well. To begin, write
\begin{equation}
\label{EqDDelInvExpansion}
  (\td(\sigma)+\tdel(\sigma))^{-1} = \sigma^{-1}A_{-1} + \cO(1),\quad A_{-1}=\sum_{j=1}^4 \la\cdot,\psi_j\ra\phi_j,
\end{equation}
near $\sigma=0$, where $\{\phi_j\}_{j=1,\ldots,4}$ is a basis of the space of resonant states and $\{\psi_j\}_{j=1,\ldots,4}$ is a basis of the space $\cH_*=\ker(\td(0)+\tdel(0))^*$ of dual states.\footnote{After choosing the $\phi_j$, say, the $\psi_j$ are uniquely determined.} Therefore, we need to understand the dual states of $d+\delta$ in order to understand the order and structure of the pole of $\wt\Box(\sigma)^{-1}=\bigl((\td(\sigma)+\tdel(\sigma))^{-1}\bigr)^2$ as $\sigma=0$. Notice here that the adjoint $(\td(\sigma)+\tdel(\sigma))^*$ acts on distributions on $\wt X$ which are \emph{supported} at the Cauchy hypersurface $\pa\wt X$ (see \cite[Appendix~B]{HormanderAnalysisPDE} for this and related notions). In particular, an element $\wt u\in\ker(\td(\sigma)+\tdel(\sigma))^*$ satisfies $\wt u\in\ker\wt\Box(\sigma)$ and is a supported distribution at $\pa\wt X$, thus by local uniqueness, $\wt u$ vanishes in the hyperbolic region $\wt X\setminus X$, hence $\supp\wt u\subset\oX$.

\begin{lemma}
\label{LemmaSDSDualStates}
  The spaces $\cH_*$ and $\cK_*$ of dual states for $d+\delta$ and $\Box$, respectively, on $n$-dimensional Schwarzschild--de Sitter space, $n\geq 4$, are graded by form degree, $\cH_*=\bigoplus_{k=0}^n\cH_*^k$, $\cK_*=\bigoplus_{k=0}^n\cK_*^k$, and have the following explicit descriptions:
  \begin{gather*}
    \cK_*^0 = \la 1_X \ra, \cH_*^0 = 0, \quad \cK_*^n = \la 1_X r^{n-2}\,dt\wedge dr\wedge\omega \ra, \cH_*^n = 0, \\
	\cH_*^1 = \la \delta_{r=r_-}\,dr, \delta_{r=r_+}\,dr \ra, \quad \cH_*^{n-1} = \la \delta_{r=r_-}\,dr\wedge\omega, \delta_{r=r_+}\,dr\wedge\omega \ra, \\
	\cH_*^k=0,\ \ k=2,\ldots,n-2,
  \end{gather*}
  where $\omega$ denotes the volume form on the round sphere $\Sph^{n-2}$. Furthermore, $\cK_*^1=\cH_*^1$, $\cK_*^{n-1}=\cH_*^{n-1}$ and
  \[
	\cK_*^k=0,\ \ k=3,\ldots,n-3.
  \]
  For $n=4$,
  \[
    \cK_*^2=\la 1_X\omega,1_X r^{-2}\,dt\wedge dr\ra,
  \]
  while for $n>4$,
  \[
    \cK_*^2=\la 1_X r^{2-n}\,dt\wedge dr\ra,\quad \cK_*^{n-2}= \la 1_X\omega \ra.
  \]
  We have $\la\phi,\psi\ra=0$ for all $\phi\in\cH$, $\psi\in\cH_*$.
\end{lemma}
\begin{proof}
  For computing the dual resonant states, we need to compute the form of $\wt\Box(0)$ near the two components of $\pa\oX=\Sph^{n-2}\sqcup\Sph^{n-2}$. Since dual states are supported in $\Xeven$, it suffices to compute $\sC_\pm^{-1}\wh\Box(0)\sC_\pm$, since any smooth extension of this operator to $\wt X$ agrees with $\wt\Box(0)$ in $X$ and to infinite order at $\pa\Xeven$,\footnote{Since the Schwarzschild--de Sitter metric is analytic, we in fact do not have to make any choices.} thus the difference annihilates dual states. Using Lemma~\ref{LemmaSDSOps}, we compute
  \begin{align*}
    -&\sC_\pm^{-1}\wh\Box_p(0)\sC_\pm
	  =r^{-2}
	    \begin{pmatrix}
		  \Delta_p&0&0&0 \\
		  0&\Delta_{p-1}&0&0 \\
		  0&0&\Delta_{p-1}&0 \\
		  0&0&0&\Delta_{p-2}
		\end{pmatrix} \\
   & + \begin{pmatrix}
         \alpha^{-1}\pa_{r,p}^*\alpha^2\pa_r & -2\alpha^2 r^{-1}d_{p-1} & \pm 2r^{-1}d_{p-1} & 0 \\
		 -2r^{-3}\delta_p & \pa_r\alpha^{-1}\pa_{r,p-1}^*\alpha^2 & \pm(2(p-1)-(n-2))r^{-2} & \mp 2r^{-1}d_{p-2} \\
		 0 & 0 & \alpha\pa_{r,p-1}^*\pa_r & -2\alpha^2 r^{-1}d_{p-2} \\
		 0 & 0 & -2r^{-3}\delta_{p-1} & \pa_r\alpha\pa_{r,p-2}^*
       \end{pmatrix},
  \end{align*}
  where the Laplace operators, differentials and codifferentials are the operators on $\Sph^{n-2}$. This does extend to an operator acting on smooth functions on $(r_\pm-\delta,r_\pm+\delta)\times\Sph^{n-2}$, $\delta>0$ small, near $r_\pm$.

  Now for $p=0$, clearly $\alpha^{-1}\pa_{r,0}^*\alpha^2\pa_r 1_X=\mp \alpha^{-1}\pa_{r,0}^*(\mu\delta_{r=r_\pm})=0$, hence $\cK_*^0=\la 1_X\ra$. (Observe that since $\wt\Box_0(0)$ is Fredholm of index $0$ and has a $1$-dimensional kernel according to Theorem~\ref{ThmSDSDDel}, the space of dual $0$-form resonances is $1$-dimensional as well.) Likewise, for $p=n$, we have
  \[
    \pa_r\alpha\pa_{r,n-2}^*(1_X r^{n-2}\,dt\wedge dr\wedge\omega) = -\pa_r\mu r^{n-2}\pa_r(1_X\,dt\wedge dr\wedge\omega) = 0,
  \]
  confirming $\cK_*^n=\la 1_X r^{n-2}\,dt\wedge dr\wedge\omega\ra$. By completely analogous arguments, we find $1_X r^{2-n}\,dt\wedge dr\in\cK_*^2$ and $1_X\omega\in\cK_*^{n-2}$.

  In order to proceed, notice that $\td(0)+\tdel(0)$ maps $\cK_*$ into $\cH_*$. Hence, we can find dual states for $d+\delta$ by applying $\td(0)+\tdel(0)$ to the dual states of $\Box$ that we have already identified. For this computation, we note
  \begin{align*}
    \sC_\pm^{-1}\hd_p(0)\sC_\pm
	  &= \begin{pmatrix}
	      d_{\Sph^{n-2},p} & 0 & 0 & 0 \\
		  \pa_r & -d_{\Sph^{n-2},p-1} & 0 & 0 \\
		  0 & 0 & -d_{\Sph^{n-2},p-1} & 0 \\
		  0 & 0 & -\pa_r & d_{\Sph^{n-2},p-2}
	    \end{pmatrix}, \\
	\sC_\pm^{-1}\hdel_p(0)\sC_\pm
	  &= \begin{pmatrix}
	       -r^2\delta_{\Sph^{n-2},p} & -\alpha^{-1}\pa_{r,p-1}^*\alpha^2 & \pm\alpha^{-1}\pa_{r,p-1}^* & 0 \\
		   0 & r^{-2}\delta_{\Sph^{n-2},p-1} & 0 & \pm\alpha^{-1}\pa_{r,p-2}^* \\
		   0 & 0 & r^{-2}\delta_{\Sph^{n-2},p-1} & \alpha\pa_{r,p-2}^* \\
		   0 & 0 & 0 & -r^{-2}\delta_{\Sph^{n-2},p-2}
	     \end{pmatrix}.
  \end{align*}
  Thus, $(\td_0(0)+\tdel_0(0))1_X$ and $(\td_2(0)+\tdel_2(0))(1_X r^{2-n}\,dt\wedge dr)$ are both linear combinations of $\delta_{r=r_\pm}\,dr$, hence $\delta_{r=r_\pm}\,dr\in\cH_*^1\subset\cK_*^1$, and similarly $(\td_n(0)+\tdel_n(0))(1_X\star 1)$ and $(\td_{n-2}(0)+\tdel_{n-2}(0))(1_X\omega)$ are both linear combinations of $\delta_{r=r_\pm}\,dr\wedge\omega$, hence $\delta_{r=r_\pm}\,dr\wedge\omega\in\cH_*^{n-1}\subset\cK_*^{n-1}$.

  We have therefore identified $4$ and $8$ linearly independent dual states for $d+\delta$ and $\Box$, which is equal to the dimensions of $\cH$ and $\cK$, respectively, and since $\td(0)+\tdel(0)$ and $\wt\Box(0)$ have index $0$, all dual states are linear combinations of these, i.e.\ we have thus identified a basis of the spaces of dual states. The orthogonality of resonant and dual states for $d+\delta$ follows immediately from the explicit forms of both derived in Theorem~\ref{ThmSDSDDel} and in this lemma: all dual states have form degree $1$ or $n-1$, while all resonant states have form degree $0$, $2$, $n-2$ or $n$.
\end{proof}

The orthogonality statement in Lemma~\ref{LemmaSDSDualStates} combined with \eqref{EqDDelInvExpansion} immediately gives $A_{-1}^2=0$, hence the coefficient of $\sigma^{-2}$ in the Laurent expansion of $\wh\Box(\sigma)^{-1}$ at $\sigma=0$ vanishes. For precisely those form degrees $0\leq p\leq n$ for which $\cK^p$ is non-trivial, $\wh\Box(\sigma)^{-1}$ does have a simple pole at $\sigma=0$, and
\[
  \wh\Box_p(\sigma)^{-1} = \sigma^{-1}\sum_{j=1}^{\dim\cK^p} \la\cdot,\psi'_j\ra\phi'_j + \cO(1),
\]
where $\phi'_j$ and $\psi'_j$ run over a basis of $\ker\wh\Box_p(0)\cong\cK^p$ and $\cK_*^p=\ker\wh\Box_p(0)^*$, respectively.\footnote{After choosing the $\phi'_j$, the $\psi'_j$ are uniquely determined, and vice versa.}

\begin{thm}
\label{ThmSDSBox}
  On $4$-dimensional Schwarzschild--de Sitter space, if $0\leq p\leq 4$ and $u$ is a differential form of degree $p$ which solves $\Box u=0$ with smooth initial data, then $u$ decays exponentially to
  \begin{itemize}
    \item a constant for $p=0$,
	\item a linear combination of $u_+$ and $u_-$, defined in the statement of Theorem~\ref{ThmSDSDDel}, for $p=1$,
	\item a linear combination of $\omega$ and $r^{-2}\,dt\wedge dr$ for $p=2$,
	\item a linear combination of $\star u_+$ and $\star u_-$ for $p=3$ and
	\item a constant multiple of $r^2\,dt\wedge dr\wedge\omega$ for $p=4$.
  \end{itemize}
  Analogous statements hold on any $n$-dimensional Schwarzschild--de Sitter space, $n\geq 5$.
\end{thm}

\section{Results for Kerr--de Sitter space}
\label{SecKdS}

We now prove that some of the results obtained in the previous section for the $4$-dimensional Schwarzschild--de Sitter spacetime are stable under perturbations \emph{which do not respect the warped product structure imposed in \S\ref{SecPrelim}}, which in particular allows us to draw conclusions about asymptotics for solutions of $(d+\delta)u=0$ or $\Box u=0$ on Kerr--de Sitter space with very little effort, even though the latter does not satisfy the requirements of Section~\ref{SecPrelim}. Thus, fixing the black hole mass $M_\bullet$ and the cosmological constant $\Lambda>0$, denote by $g_a$ the Kerr--de Sitter metric with angular momentum $a$; thus, $g_0$ is the Schwarzschild--de Sitter metric.\footnote{Assuming the non-degeneracy condition \eqref{EqSDSNondegeneracy}, which ensures that the cosmological horizon lies outside the black hole event horizon, the same will be true for small $|a|$, which is the setting in which work here. In general, one would need to assume that $\Lambda$, $M_\bullet$ and $a$ are such that the non-degeneracy condition \cite[(6.2)]{VasyMicroKerrdS} holds.} Only very basic facts about the metric will be used; we refer the reader to \cite[\S{6}]{VasyMicroKerrdS} for details and further information. We will write $\delta_{g_a}$ for the codifferential with respect to the metric $g_a$. We furthermore denote by $M=\R_t\times X$ the domain of exterior communications, and by $\wt M=\R_{t_*}\times\wt X$ the `extended' spacetime, with $t_*$ defined in\footnote{Our $t$, $t_*$ are denoted $\tilde t$, $t$, respectively in the reference.}~\cite[Equation~(6.4) and beginning of \S6.4]{VasyMicroKerrdS}.

To begin, recall that the scalar wave equation (and by essentially the same arguments the wave equation on differential forms, since the principal symbol of the Hodge d'Alembertian is scalar, see also \cite[\S{4}]{VasyHyperbolicFormResolvent} for a discussion in a related context) on the Kerr--de Sitter spacetime fits into the microlocal framework developed in \cite{VasyMicroKerrdS}. In particular, asymptotics for waves follow directly from properties of the Mellin transformed normal operator family, and moreover the analysis of the latter is stable under perturbations. In the present context, this concretely means that for any $\eps>0$, there exists $a_\eps>0$ such that for all angular momenta $a$ with $|a|<a_\eps$, the meromorphic family of operators $\cR_a(\sigma):=(\td(\sigma)+\wt{\delta_{g_a}}(\sigma))^{-1}$ has no poles in $|\sigma|\geq\eps$, $\Im\sigma\geq 0$, and such that moreover all poles in $|\sigma|<\eps$ are perturbations of the pole of $\cR_0(\sigma)$ at $0$, in the sense the sum of the ranks of the resonances (i.e.\ poles of $\cR_a(\sigma)$) in $|\sigma|<\eps$ equals the corresponding sum for the Schwarzschild--de Sitter metric, which equals $4$ by Theorem~\ref{ThmSDSDDel}; we refer the reader to \cite[Appendix~A]{HintzThesis} (which extends the perturbative discussion of~\cite[\S2.7]{VasyMicroKerrdS}) for definitions and details, and here merely point out the presently relevant consequence that sum of the dimensions of the kernels of $\td(\sigma)+\wt{\delta_{g_a}}(\sigma)$ for $|\sigma|<\eps$ is at most $4$-dimensional. Now, Lemma~\ref{LemmaSDSDualStates} suggests considering dual resonant states instead (which have a simpler form); the same stability result as for $\cR_a(\sigma)$ holds for $\cR_a^*(\sigma):=((\td(\sigma)+\wt{\delta_{g_a}}(\sigma))^*)^{-1}$. However, just as in the case of Schwarzschild--de Sitter space, we can immediately write down $4$ linearly independent dual $0$-resonant states for $d+\delta_{g_a}$: namely, apply $\td(0)+\wt{\delta_{g_a}}(0)$ to $1_X$ (this is a dual resonant state for $\Box_{g_a}$), which produces a sum of $\delta$-distributions supported at the horizons $r=r_\pm$, and splitting this up into the part supported at $r_-$ and the part supported at $r_+$, we obtain $2$ linearly independent dual resonant states for $d+\delta$ in form degree $1$. The same procedure can be applied to $\star_{g_a}1_X$, yielding $2$ linearly independent dual resonant states for $d+\delta$ in form degree $3$ (which are simply the Hodge duals of the dual states in form degree $1$). Hence,
\begin{equation}
\label{EqKDSDDelRes}
  \cH_a:=\ker(\td(0)+\wt{\delta_{g_a}}(0)),
\end{equation}
which has the same dimension as
\begin{equation}
\label{EqKDSDDelAdjointRes}
  \cH_{a,*}:=\ker(\td(0)+\wt{\delta_{g_a}}(0))^*,
\end{equation}
is at least $4$-dimensional for small $|a|$, but it is also at most $4$-dimensional by the above perturbation stability argument! Hence, for small $|a|$, we deduce that $0$ is the only pole of $\cR_a(\sigma)$, i.e.\ the only resonance of $d+\delta_{g_a}$, in $\Im\sigma\geq 0$ (and also the only pole of $\cR_a^*(\sigma)$ in this half space), and is simple. 

We can use this in turn to prove the stability of the zero resonance for $\Box_{g_a}$ in all form degrees. Let $\pi_k\colon\CI(\wt M;\Lambda\wt M)\to\CI(\wt M;\Lambda\wt M)$ denote the projection onto differential forms with pure form degree $k\in\{0,\ldots,4\}$, which induces a map on $\CI(\wt X;\Lambda\wt X\oplus\Lambda\wt X)$. Let
\begin{equation}
\label{EqKDSBoxRes}
  \cK_a:=\ker\wt{\Box_{g_a}}(0)=\bigoplus_{k=0}^4\cK_a^k
\end{equation}
be the grading of the zero resonant space of $\Box_{g_a}$ by form degree, likewise
\begin{equation}
\label{EqKDSBoxAdjointRes}
  \cK_{a,*}:=\ker\wt{\Box_{g_a}}(0)^*=\bigoplus_{k=0}^4\cK_{a,*}^k
\end{equation}
for the space of dual resonant states. Observe that $\pi_k\cH_a\subseteq\cK_a^k$, since $u\in\cH_a$ implies $0=\pi_k\Box_{g_a}u=\Box_{g_a}\pi_k u$. Now, since $\Box_{g_a}1=0$, we have $\cK_a^0=\la 1\ra$ for small $|a|$ by stability, likewise $\cK_a^4=\la\star_{g_a}1\ra$. Furthermore, $\cK_a^2$ is at most $2$-dimensional for small $|a|$ (since $\cK_0^2$ is $2$-dimensional), but also $\cK_a^2\supseteq\pi_2\cH_a$; now $\pi_2\cH_0$ is $2$-dimensional by Theorem~\ref{ThmSDSDDel} and $\cH_a$ depends smoothly on $a$, thus $\cK_a^2=\pi_2\cH_a$ is $2$-dimensional for small $|a|$; therefore $\cK_a^2=\pi_2\cH_a$ is $2$-dimensional. Finally, we have $\cH^1_{a,*}\subseteq\cK^1_{a,*}$, hence by the analysis of $d+\delta_{g_a}$ above, $\cK^1_{a,*}$, hence $\cK^1_a$, is at least $2$-dimensional, but since $\cK^1_0$ is $2$-dimensional, we must in fact have $\dim\cK^1_a=2$ for small $|a|$; likewise $\dim\cK^3_a=2$. Hence, we have $\dim\cK_a^k=\dim\cK_0^k$ for $k=0,\ldots,4$, which in particular means that the zero resonance of $\Box_{g_a}$ is the only resonance in $\Im\sigma\geq 0$, and the resonance is simple.

We now summarize the above discussion, including a small improvement. The following theorem is completely parallel to Theorem~\ref{ThmSDSDDel}, Lemma~\ref{LemmaSDSDualStates} and Theorem~\ref{ThmSDSBox} for Schwarzschild--de Sitter spacetimes, extending these to Kerr--de Sitter spacetimes with small angular momentum:

\begin{thm}
\label{ThmKDS}
  For small $|a|$, the only resonance of $d+\delta_{g_a}$ in $\Im\sigma\geq 0$ is a simple resonance at $\sigma=0$, likewise for $\Box_{g_a}$. The spaces $\cH_a$ and $\cH_{a,*}$ of resonant and dual resonant states for $d+\delta_{g_a}$ are graded by form degree as $\cH_a=\bigoplus_{k=0}^4\cH_a^k$, $\cH_{a,*}=\bigoplus_{k=0}^4\cH_{a,*}^k$, in particular $\cH_a^k=\ker\td_k(0)\cap\ker(\wt{\delta_{g_a}})_k(0)$,\footnote{The subscript denotes the degree of differential forms on which the respective operator acts.} with
  \[
    \cH_a^0=\la 1\ra,\quad \cH_a^1=0,\quad \cH_a^2=\la u_{a,1},u_{a,2}\ra,\quad \cH_a^3=0,\quad \cH_a^4 = \la\star_{g_a}1\ra
  \]
  for some $2$-forms $u_{a,1},u_{a,2}$, which can be chosen to depend smoothly on $a$,\footnote{We derive an explicit expression in Remark~\ref{RmkKDSExplicit} below.} with $u_{0,1}=r^{-2}\,dt\wedge dr$, $u_{0,2}=\omega$ in the notation of Theorem~\ref{ThmSDSDDel}, and
  \begin{gather*}
    \cH_{a,*}^0=0,\quad \cH_{a,*}^1=\la\delta_{r=r_-}\,dr,\delta_{r=r_+}\,dr\ra, \\
	\cH_{a,*}^2=0,\quad \cH_{a,*}^3=\star_{g_a}\cH_{a,*}^1,\quad \cH_{a,*}^4=0.
  \end{gather*}
  For the spaces $\cK_a$ and $\cK_{a,*}$ of resonant and dual resonant states for $\Box_{g_a}$, we have
  \[
    \cK_a^0=\cH_a^0,\quad \cK_a^1=\la u_{a,+},u_{a,-}\ra,\quad \cK_a^2=\cH_a^2,\quad \cK_a^3=\star_{g_a}\cK_a^1,\quad \cK_a^4=\cH_a^4
  \]
  for some $1$-forms $u_{a,\pm}$, which can be chosen to depend smoothly on $a$, with $u_{0,\pm}=u_\pm$ in the notation of Theorem~\ref{ThmSDSBox}, and
  \begin{gather*}
    \cK_{a,*}^0=\la 1_X\ra,\quad \cK_{a,*}^1=\cH_{a,*}^1, \\
	\cK_{a,*}^2=\la 1_X u_{a,1},1_X u_{a,2}\ra,\quad \cK_{a,*}^3=\cH_{a,*}^3,\quad \cK_{a,*}^4=\la \star_{g_a}1_X\ra.
  \end{gather*}
  In particular, the form degree $k$ part of a solution $u$ to $(d+\delta_{g_a})u=0$, resp.\ $\Box_{g_a}u=0$, with smooth initial data decays exponentially to an element of $\cH_a^k$, resp.\ $\cK_a^k$, for $k=0,\ldots,4$.
\end{thm}

\begin{rmk}
  Since for all $k=0,\ldots,4$, either $\cH_a^k=0$ or $\cH_{a,*}^k=0$, hence $\cH_a$ and $\cH_{a,*}$ are orthogonal, we obtain another proof, as in the Schwarzschild--de Sitter case, of the fact that $\Box_{g_a}$ acting on differential forms only has a simple resonance at $0$.
\end{rmk}

\begin{proof}[Proof of Theorem~\ref{ThmKDS}.]
  We only need to prove that the space $\cH_a$ is graded by form degree: let $\pi_\even=\pi_0+\pi_2+\pi_4$ denote the projection onto even form degree parts, then since $d+\delta_{g_a}$ maps even degree forms to odd degree forms and vice versa, $\pi_\even$ maps $\cH_a$ into itself. Now suppose $u\in\pi_\even\cH_a$, and write $u=u_0+u_2+u_4$ with $u_k=\pi_k u$, $k=0,2,4$. Then $0=\pi_1(d+\delta_{g_a})u=du_0+\delta_{g_a}u_2$,\footnote{We use the identification of resonant states with $t_*$-independent forms as in the proof of Theorem~\ref{ThmSummary}.} and applying $\delta_{g_a}$ to this equation gives $0=\Box_{g_a}u_0$, which implies $u_0\in\cK_a^0$, i.e.\ $u_0$ is a constant, as discussed before the statement of the theorem. Likewise, $u_4\in\cK_a^4$, so $u_4$ is the Hodge dual of a constant. Therefore, $d+\delta_{g_a}$ annihilates both $u_0$ and $u_4$, hence $u_2\in\cH_a$. This argument shows that in fact $\pi_2\cH_a\subset\cH_a$. Since $\pi_2\cH_a$ is $2$-dimensional, as noted above, we have
  \[
    \la 1\ra \oplus \pi_2\cH_a \oplus \la\star_{g_a}1\ra \subseteq \cH_a,
  \]
  with both sides having the same dimension (namely, $4$), and thus equality holds, providing the grading of $\cH_a$ by form degree.
\end{proof}

This in particular proves Theorem~\ref{ThmIntroKDS}.

\begin{rmk}
\label{RmkQuasilinear}
  Observe that \emph{all} ingredients in the Fredholm analysis of the normal operator family, which here in particular involves estimates at normally hyperbolic trapping, as well as \emph{all} of the above arguments which lead to a characterization of the spaces of resonances are stable in the sense that they apply to \emph{any} stationary perturbation of a given Schwarzschild--de Sitter spacetime ($4$-dimensional for the above, but similar arguments apply in all spacetime dimensions $\geq 4$), not only to slowly rotating Kerr--de Sitter black holes.
  
  In fact, using the analysis of operators with non-smooth coefficients developed in \cite{HintzQuasilinearDS} and extended in \cite{HintzVasyQuasilinearKdS}, we can deduce decay and expansions in the exact same form as in the above theorem for waves on spacetimes which are merely `asymptotically stationary' and close to Schwarzschild--de Sitter, i.e.\ for which the metric tensor differs from a stationary metric close to Schwarzschild--de Sitter by an exponentially decaying symmetric $2$-tensor (with suitable regularity). This shows at once that \emph{quasilinear} wave equations on differential forms of the form $\Box_{g(u,\nabla u)}u=q(u,\nabla u)$ with small initial data can be solved globally, provided $g(0,0)$ is close to the Schwarzschild--de Sitter metric, and the non-linearity $q$ annihilates $0$-resonant states; to give an (artificial) example, on $2$-forms, one could take $q(u,\nabla u)=|du|^2 u$.
\end{rmk}

\begin{rmk}
\label{RmkKDSExplicit}
  In the case of the Kerr--de Sitter metric, we can in fact explicitly write down $u_{a,1}\in\cH_a^2$ (and then take $u_{a,2}=\star_{g_a}u_{a,1}$ to obtain a basis of $\cH_a^2$). Indeed, on the Kerr spacetime, Andersson and Blue \cite{AnderssonBlueMaxwellKerr} give the values of the spin coefficients of the Maxwell field for the Coulomb solution in \cite[\S{3.1}]{AnderssonBlueMaxwellKerr}, and reconstructing the Maxwell field itself (in the basis given by wedge products of differentials of the Boyer--Lindquist coordinates $t,r,\theta,\phi$) is then an easy computation using the explicit form of the null tetrad given in \cite[Introduction, \S{2.4}]{AnderssonBlueMaxwellKerr}.\footnote{In the definition of $\phi_0$ in \cite[\S{2.4}]{AnderssonBlueMaxwellKerr}, the second summand $F[\hat{\mathbf{\Theta}},\hat{\mathbf{\Phi}}_{\mathrm{PNV}}]$ should be replaced by $F[\bar{\mathbf{m}},\mathbf{m}]$ to yield the correct result, see also \cite[Equation (2)$\dag$]{ChandrasekharMaxwellKerr}.} A tedious but straightforward calculation shows that the resulting $2$-form
  \begin{align*}
    u_{a,1} := F_{a,TR}(r,&\theta)\,(dt-a\sin^2\theta\,d\phi)\wedge dr \\
	  &+ F_{a,\Theta\Phi}(r,\theta)\sin\theta\,d\theta\wedge(a\,dt-(r^2+a^2)\,d\phi) \\
  \end{align*}
  with
  \[
	F_{a,TR}(r,\theta)=\frac{r^2-a^2\cos^2\theta}{(r^2+a^2\cos^2\theta)^2},\quad F_{a,\Theta\Phi}(r,\theta)=\frac{2ar\cos\theta}{(r^2+a^2\cos^2\theta)^2}
  \]
  is a solution of Maxwell's equations on Kerr--de Sitter space as well, i.e.\ when the cosmological constant is positive.
\end{rmk}


\end{document}